\documentclass[10pt]{amsart}
\usepackage{amsmath}
\usepackage{amsthm}
\usepackage{amsfonts}
\usepackage{mathrsfs}
\usepackage{amssymb}
\usepackage{amscd}
\usepackage[all]{xy}
\usepackage{pgf}
\usepackage{tikz}
\usetikzlibrary{cd}

\usepackage{hyperref}

\usepackage{amscd}

\usepackage[
    hscale=0.7,
    vscale=0.75,
    headheight=13pt,
    centering,
    ]{geometry}

\usepackage{enumerate}
    \newtheorem{Lem}{Lemma}[section]
    \newtheorem{Prop}[Lem]{Proposition}

      \newtheorem*{thma}{Theorem A}      
      \newtheorem*{thmb}{Theorem B}   
      \newtheorem*{thmc}{Theorem C}

    \newtheorem{Thm}[Lem]{Theorem}  
    \newtheorem{Cor}[Lem]{Corollary}
        \newtheorem{Lem-Def}[Lem]{Lemma-Definition}
		\newtheorem*{Thm*}{Theorem}

\theoremstyle{definition}

    \newtheorem{Def}[Lem]{Definition}
    \newtheorem{Exa}[Lem]{Example}
    \newtheorem{Rem}[Lem]{Remark}
    
    \newtheorem{Not}[Lem]{Notation}

\newcommand{\ra}{\rightarrow}
\newcommand{\mb}{\mathbb}
\newcommand{\mc}{\mathcal}
\newcommand{\E}{\mathcal E}

\newcommand{\T}{\mathcal T}
\newcommand{\col}{\colon}
\newcommand{\ora}{\overrightarrow}
\newcommand{\ol}{\overline}
\newcommand{\wh}{\widehat}
\newcommand{\wt}{\widetilde}

\newcommand{\D}{\mathcal{D}}
\newcommand{\PS}{\mathbf{PSD}}
\newcommand{\ps}{\mathcal{PSD}}

\newcommand{\qs}{\mathcal{QD}}
\newcommand{\J}{\mathcal{J}}

\newcommand{\R}{\mathbb{R}}

\newcommand{\trop}{\textnormal{trop}}
\newcommand{\Aut}{\textnormal{Aut}}

\DeclareMathOperator{\id}{id}

\DeclareMathOperator{\Div}{Div}
\DeclareMathOperator{\codim}{codim}
\DeclareMathOperator{\pol}{pol}

\DeclareMathOperator{\ord}{ord}
\DeclareMathOperator{\Pic}{Pic}
\DeclareMathOperator{\cone}{cone}

\DeclareMathOperator{\CYC}{{\bf CYC}}

\DeclareMathOperator{\supp}{supp}

\DeclareMathOperator{\Sym}{Sym}
\DeclareMathOperator{\Spec}{Spec}
\DeclareMathOperator{\Proj}{Proj}
\DeclareMathOperator{\st}{st}
\DeclareMathOperator{\leg}{leg}

\DeclareMathOperator{\Ex}{Ex}

\begin{document}

\thanks{The second author was supported by CNPq-PQ, 301671/2019-2. The third author was supported by Capes.}

\newcommand\encircle[1]{
  \tikz[baseline=(X.base)] 
    \node (X) [draw, shape=circle, inner sep=0] {\strut #1};}

\setcounter{tocdepth}{1}

 \title[The moduli space of quasistable spin curves]{The moduli space of quasistable spin curves}

\author{Alex Abreu, Marco Pacini, and Danny Taboada}

%\email{alexbra1@gmail.com, pacini.uff@gmail.com,  danny\textunderscore daft@hotmail.com}

\maketitle

\begin{abstract}
We study a compactification of the moduli space of theta characteristics, giving a modular interpretation of the geometric points and describing the boundary stratification. This space is different from the moduli space of spin curves. The modular description and the boundary stratification of the new compactification are encoded by a tropical moduli space. We show that this tropical moduli space is a refinement of the moduli space of spin tropical curves. We describe explicitly the induced decomposition of its cones.
\end{abstract}

\bigskip



\tableofcontents

\section{Introduction}

\subsection{Background and motivation}

The goal of this paper is the study of a  compactification, denoted $\ol{\mc Q}_{g,n}$, of the moduli space $\mc S_{g,n}$ of theta characteristics
on $n$-pointed smooth curves. 
A distinguished compatification, $\ol{\mc S}_{g,n}$, of $\mc S_{g,n}$ over the moduli space $\ol {\mc M}_{g,n}$ of $n$-pointed stable curves was introduced by Cornalba in \cite{Co}. The stack $\ol{\mc S}_{g,n}$ is finite over $\ol{\mc M}_{g,n}$ and parametrizes \emph{spin curves}, which are certain pairs of  nodal curve and invertible sheaves on it. The construction was later generalized in \cite{CCC} for  higher roots of any sheaf.  There are at least two other different approaches to this compactification problem. The first uses torsion-free rank-1 sheaves on stable curves and is introduced by Jarvis (see \cite{J1} and \cite{J2}). The second involves the use of twisted curves (nodal curves with a stacky structure at the nodes) and is employed by Abramovich and Jarvis in \cite{AJ} and by Chiodo in \cite{Ch}.  
All these compactifications of $\mc S_{g,n}$ are isomorphic to $\ol{\mc S}_{g,n}$.

%Over the last decade, interest in higher spin curves has been revived by the generalized Witten conjecture, which predicts that the intersection theory on theirmoduli spaces is governed by the Gelfand-Diki˘ı (also known as higher KdV) hierarchy. This conjecture is open2 ; see [18] for more details and some recent progress.

We define $\ol{\mc Q}_{g,n}$  as the closure of $\mc S_{g,n}$ in the Esteves universal compactified Jacobian of  quasistable torsion-free rank-1 sheaves on stable curves. We call $\ol{\mc Q}_{g,n}$ the moduli space of \emph{quasistable spin curves}.  
 We give a modular description of the geometric points of $\ol{\mc Q}_{g,n}$ and study its  boundary stratification. The upshot is that $\ol{\mc Q}_{g,n}$ is a compactification of $\mc S_{g,n}$ which is not isomorphic to $\ol{\mc S}_{g,n}$. In a way, this is a natural reflection of the existence of two distinguished and different universal Jacobians: Caporaso universal Picard variety, constructed in \cite{C}, and the Esteves universal compactified Jacobian, introduced by Melo in \cite{M15} bulding upon the work of Esteves  \cite{Es01}. These compactified Jacobians are different. For instance, using the terminology in \cite{Capo}, the first one, in general, is  not of N\'eron type, while the second one is always of N\'eron type. Caporaso compactified Jacobian contains $\ol{\mc S}_{g,n}$ (at least set-theoretically), so it is natural to expect a difference between $\ol{\mc Q}_{g,n}$ and $\ol{\mc S}_{g,n}$.

We achieve our results by resorting to the tropical counterpart of the problem. In the last few years, tropical geometry became a powerful tool to encode moduli problems and describe the boundary stratifications of compactified moduli spaces. The tropical setting often helps in understanding which combinatorial data are relevant to describe the strata. The associated tropical moduli space is a generalized cone complex which can be regarded as a ``dual" of the algebro-geometric space. Usually, we have a reverse-inclusion bijection between the strata of the moduli space and the cells of the cone complex switching dimension and codimension.

This phenomenon appears in the study of  $\ol{\mc M}_{g,n}$, the moduli space of $n$-pointed stable curves, and of its tropical counterpart $M_{g,n}^\trop$, the moduli space of $n$-pointed tropical curves. In fact, a deep connection between $\ol{\mc M}_{g,n}$ and $M_{g,n}^\trop$ was revealed in \cite{ACP}, where Abramovich, Caporaso, and Payne showed that $\ol{M}_{g,n}^\trop$ is identified with the Berkovich skeleton of $\ol{\mc M}_{g,n}$. In \cite{CMP1}, Caporaso, Melo, and the second author, did a similar analysis for the moduli space of spin curves $\ol{\mc S}_{g,n}$. It is worth mentioning that the new insight of the boundary stratification given by tropical geometry was recently used in \cite{CGP} and \cite{CFGP} to give strong results about the rational cohomology of $\mc M_g$ and the $S_n$-equivariant Euler characteristic of the top weight cohomology of $\mc{M}_{g,n}$.

 We proceed in a similar way: we define the relevant tropical moduli space and we use it to describe the geometric points and the boundary stratification of $\ol{\mc Q}_{g,n}$.

\subsection{Outline of the results}

Let us give more details about our results. We start with the tropical setting.
In \cite{AP1} and \cite{AAPT} two universal tropical Jacobians have been introduced. The first one, denoted $J^\trop_{\mu,g,n}$, parametrizes $(\mu,p_0)$-quasistable divisors on tropical curves of genus $g$. The second one, denoted $P^\trop_{\mu,g,n}$, parametrizes $\mu$-polystable divisors on tropical curves of genus $g$. In both cases, $\mu$ denotes a universal polarization.  The two universal tropical Jacobians are different. In fact, in Theorem \ref{thm:refiJ}, we show that $J^\trop_{\mu,g,n}$ is a refinement of $P^\trop_{\mu,g,n}$, and we describe the induced decomposition of the cones of $P^\trop_{\mu,g,n}$. This generalizes to the universal setting a similar result for the Jacobian of a tropical curve, proved in \cite[Proposition 3.7]{AAPT} and \cite[Corollay 5.10]{CPS}.   

Motivated by the algebro-geometric picture, we can consider the locus inside the tropical Jacobian parametrizing square roots of the canonical divisor. This leads to a notion of \emph{tropical theta characteristics} of a tropical curve, first studied by Zharkov in \cite{Zha}.  A moduli space, $T_{g,n}^{\trop}$, of theta characteristics on $n$-pointed tropical curves of genus $g$, is constructed in \cite[Section 4]{CMP2}, based on the existence of certain representatives of theta characteristics that are well-behaved under specializations.  
 In Proposition \ref{prop:equiv} we show that in the equivalence class of a theta characteristic there is exactly one $\mu$-polystable divisor, where $\mu$ is the canonical polarization of degree $g-1$. It is possible to write explicitly this distinguished representative (see equation \eqref{eq:DPoly}). This gives rise to an explicit injection $T_{g,n}^{\trop}\hookrightarrow P^\trop_{\mu,g,n}$. 

However $T^\trop_{g,n}$ is not the tropical counterpart of the moduli space $\ol{\mc S}_{g,n}$ of spin curves. In fact, $T^\trop_{g,n}$ is connected, since all tropical theta characteristic admits a specialization to the unique theta characteristic on the tropical curve consisting of one vertex, while the stack $\ol{\mc S}_{g,n}$ is not (it has two connected components, corresponding to the parity of the space of sections of a spin curve). The right tropical counterpart of $\ol{\mc S}_{g,n}$ is the space $S^\trop_{g,n}$ of spin tropical curves constructed in \cite{CMP1}, parametrizing tropical theta characteristics together with the datum of a sign function. We have a forgetful map $S^\trop_{g,n}\ra T^\trop_{g,n}$, and hence a chain of maps
\[
S^\trop_{g,n}\ra T^\trop_{g,n}\hookrightarrow P^\trop_{\mu,g,n}.
\]

The idea is to pull-back the refinement  $J^\trop_{\mu,g,n}\ra P^\trop_{\mu,g,n}$ of $P^\trop_{\mu,g,n}$ to $S^\trop_{g,n}$ via the chain of maps above. We look for a generalized cone complex, $Q^\trop_{g,n}$, whose associated poset could encode the data of the induced refinement of $S^\trop_{g,n}$. This amounts to impose the right condition on the poset associated to $P^\trop_{\mu,g,n}$. The last one is the poset  of triples $(\Gamma,\E,D)$ where $\Gamma$ is a stable graph of genus $g$ with $n$ legs, $\E$ is a subset of the set of edges of $\Gamma$, and $D$ is a $\mu$-polystable divisor on the graph $\Gamma^\E$  obtained from $\Gamma$ by inserting a vertex in the interior of any edge $e\in \E$. The key condition is:  
 \begin{equation}\label{eq:key-cond}
 2D+\Div(\phi)=K_{\Gamma^\E},
 \end{equation}
 where $\Div(\phi)$ is the divisor associated to some acyclic flow $\phi$ on $\Gamma^\E$ and $K_{\Gamma^\E}$ is the canonical divisor of $\Gamma^\E$. The flow $\phi$ is uniquely determined if $D$ is $\mu$-semistable and plays an important role for us. We refer to  Theorem \ref{thm:uniquephi} for more properties of the flow $\phi$ (for example, $|\phi(e)|\le1$ for every oriented edge $e$ of $\Gamma^\E$).
 Using this terminology, the poset underlying $T^\trop_{g,n}$ is   given by the triples $(\Gamma,\E,D)$, where $D$ is $\mu$-polystable and satisfies equation  \eqref{eq:key-cond}. We can also reinterpret the poset underlying $S^\trop_{g,n}$ simply adding a sign function to these triples. The poset associated to $Q^\trop_{g,n}$ is formed by tuples $(\Gamma,\E,D,s)$ where $D$ is $\mu$-quasistable and satisfies equation \eqref{eq:key-cond}, and $s$ is a sign function. An important step in the construction of the generalized cone complex $Q^\trop_{g,n}$ is the explicit description of the equations of its cones. These are obtained by intersecting the cone $\mathbb R^{E(\Gamma^\E)}_{\ge0}$ with the linear subspace:
\[
\Lambda_{(\Gamma,\E,D,s)}\col \sum_{e\in \gamma} \phi(e) x_e=0
\]
where $\gamma$ runs over the oriented cycles of $\Gamma^\E$ and $x_e$ is the coordinates of $\mathbb R^{E(\Gamma^\E)}$ corresponding to an edge $e$ of $\Gamma^\E$
(see equation \eqref{eq:phieq}).
 The following result is Theorem \ref{thm:refiS}.

\begin{thma}\label{thm:A}
We have a refinement $Q_{g,n}^\trop\ra S_{g,n}^\trop$ of generalized cone complexes. There is an explicit description of the associated decomposition of the cones of $S_{g,n}^\trop$, and of their interiors.
\end{thma}

Next, we come back to algebraic geometry. The moduli space, $\mc S_{g,n}$, of theta characteristics on $n$-pointed smooth curves of genus $g$ naturally sits inside $\ol{\mc J}_{\mu,g,n}$, the Esteves universal compactified Jacobian. As we already mentioned, we define  $\ol{\mc Q}_{g,n}$ as  the closure of $\mc S_{g,n}$ in $\ol{\mc J}_{\mu,g,n}$. For our purposes, it is convenient to interpret $\ol{\mc J}_{\mu,g,n}$ as the moduli space of pair $(X,L)$, where $X$ is a quasistable curve and $L$ is a $(\mu,p_0)$-quasistable invertible sheaf on $X$ with degree $-1$ on exceptional components (see Proposition \ref{prop:isoJP}).  The tropical setting suggests a modular interpretation of  $\ol{\mc Q}_{g,n}$. In fact, given a triple $(\Gamma,\E,D)$ satisfying equation \eqref{eq:key-cond},
we can consider the subgraph $P_{(\E,D)}$ of $\Gamma^\E$ over which  $\phi$ is zero. If $X$ is a nodal curve with dual graph $\Gamma^\E$, we can consider the subcurves of $X$ corresponding to the connected component of $P_{(\E,D)}$, which we call the \emph{$\phi$-subcurves} of $X$. The following result is Theorem \ref{thm:modularQgn}.

\begin{thmb}
Let $(X,L)$ be a pair parametrized by a geometric point of $\ol{\mc J}_{\mu,g,n}$.
%be a quasistable $n$-pointed curve.  Given an invertible sheaf $L$ on $X$, 
The following conditions are equivalent:
\begin{itemize}
    \item[(1)] The point parametrizing $(X,L)$ is contained in $\ol{\mc Q}_{g,n}$.
     \item[(2)] 
The combinatorial type $(\Gamma,\E,D)$ of $(X,L)$ satisfies equation  \eqref{eq:key-cond} and 
%is a $v_0$-quasistable root-graph and 
for every $\phi$-subcurve $Z$ of $X$ with corresponding subgraph $P_Z$ of $\Gamma^\E$, we have an isomorphism: 
   \[
   \Big(L|_Z\otimes \mc O_Z\Big(-\sum p_e\Big)\Big)^{\otimes 2}\cong \omega_Z
   \]
   where $p_e$ runs over the nodes of $X$ corresponding to the oriented edges $e$ of $\Gamma^\E$ with target  contained in $P_Z$ and source not contained in $P_Z$, such that  $\phi(e)=-1$.
   \end{itemize}
\end{thmb}

Finally we study the boundary stratification of the open immersion $\mc S_{g,n}\hookrightarrow \ol{\mc Q}_{g,n}$.  We define $\mc Q_{(\Gamma,\E,D,s)}$ as the locus of points of $\ol{\mc Q}_{g,n}$ whose combinatorial type is a fixed tuple $(\Gamma,\E,D,s)$. 

\begin{thmc}
The following decomposition:
\[
\ol{\mc Q}_{g,n}=\coprod_{(\Gamma,\E,D,s)}\mc Q_{(\Gamma,\E,D,s)}.
\]
%In particular, $\ol{\mc Q}_{g,n}$ has exactly two connected components $\ol{\mc Q}^+_{g,n}$ and $\ol{\mc Q}^-_{g,n}$ corresponding to even and odd quasistable spin curves and:
 is a stratification of $\ol{\mc Q}_{g,n}$ indexed by the poset underlying $Q^\trop_{g,n}$.
\end{thmc}

We have an explicit formula for the codimension of the strata $\mc Q_{(\Gamma,\E,D,s)}$. This allows us to show that the natural forgetful map $\ol{\mc Q}_{g,n}\ra\ol{\mc M}_{g,n}$ is finite exactly over the locus of tree-like curves. In particular,  $\ol{\mc Q}_{g,n}$ is not isomorphic to $\ol{\mc S}_{g,n}$. We believe that Theorem A should help in the construction of an explicit blowup $\ol{\mc Q}_{g,n}\ra \ol{\mc S}_{g,n}$, locally given by a toric morphism induced by the decomposition of cones associated to the refinement $Q^\trop_{g,n}\ra S^\trop_{g,n}$. We plan to investigate this problem in a different paper. We can also give a combinatorial characterization of the codimension 1 strata $\mc Q_{(\Gamma,\E,D,s)}$: there are more boundary divisors in $\ol{\mc Q}_{g,n}$ that in $\ol{\mc S}_{g,n}$. Finally, we have a reverse-inclusion bijection between the strata of $\ol{\mc Q}_{g,n}$ and the cells of  $Q_{g,n}^{\trop}$  switching dimension and codimension, proving, as expected, an interesting ``duality" between $Q^\trop_{g,n}$ and $\ol{\mc Q}_{g,n}$. We refer to Theorem \ref{thm:Q-strat} for the precise statement of the results.

\section{Preliminaries}

\subsection{Graphs}

We denote by $\Gamma$ a weighted graph (or, simply, a \emph{graph}) with legs, where $V(\Gamma)$ and $E(\Gamma)$ are the sets of vertices and edges, $w\col V(\Gamma)\ra \mathbb Z_{\ge0}$ is the weight function, and $L_\Gamma(v)$ (or, simply, $L(v)$) is the number of legs incident to $v\in V(\Gamma)$.

Let $\Gamma$ be a graph. For every subset $\E\subset E(\Gamma)$, we let $\deg_{\E}(v)$ be the number of edges in $\E$ incident to $v$ (with loops counting twice). We set $\deg(v)=\deg_{E(\Gamma)}(v)$. 
Given subsets $V,W\subset V(\Gamma)$, we let $E(V,W)\subset E(\Gamma)$ be the set of edges with one end-vertex in $V$ and the other in $W$. If $\Gamma_1,\Gamma_2$ are subgraphs of $\Gamma$, we abuse notation writing $E(\Gamma_1,\Gamma_2)=E(V(\Gamma_1),V(\Gamma_2))$. We set $E(V)=E(V,V)$.
For every subset $V\subset V(\Gamma)$, we let  $\Gamma(V)$ be the subgraph of $\Gamma$ whose set of vertices is $V(\Gamma)$ and with set of edges $E(V)$. For $V\subset V(\Gamma)$ we set $V^c=V(\Gamma)\setminus V$. 

We denote by $\Aut(\Gamma)$ be the group of automorphism of $\Gamma$. We also let $b_0(\Gamma)$ be the number of connected components of $\Gamma$ and we set $b_1(\Gamma)=\#E(\Gamma)-\#V(\Gamma)+b_0(\Gamma)$. The \emph{genus} of $\Gamma$ is defined as $g(\Gamma)=\sum_{v\in V(\Gamma)} w(v)+b_1(\Gamma)$. We say that $\Gamma$ is \emph{stable} 
%(respectively, \emph{semistable}) 
if $2w(v)-2+\deg_\Gamma(v)+L_\Gamma(v)>0$ %(respectively, $\ge0$ and if equality holds with $w(v)=0$, then $\deg_\Gamma(v)=2$) 
for every $v\in V(\Gamma)$.
Given a subset $\E\subset E(\Gamma)$, we let  $\Gamma_\E$ be the graph obtained by removing $\E$, so that $\Gamma$ and $\Gamma_\E$ have the same vertices and the same legs. We say that $\E$ is \emph{non-disconnecting} if $\Gamma_\E$ is connected. We denote by $\Gamma_{\E,\leg}$ the graph obtained by replacing every edge $e\in \E$ with end-vertices $v_1,v_2$ by a pair of legs, the first incident to $v_1$ and the second to $v_2$. 
A \emph{cycle} of $\Gamma$ is a connected subgraph of $\Gamma$ all of whose vertices $v$ satisfy $\deg_\Gamma(v)=2$. We will often identify a cycle and its set of edges.  A \emph{cyclic subgraph} of $\Gamma$ is a (possibly empty) union of cycles of $\Gamma$ with no common edges. 
 
We define the set $\ora{E}(\Gamma)=E_1\sqcup E_2$, whose elements are called \emph{oriented edges}, where $E_i=E(\Gamma)$. We consider the function $u\col\ora{E}(\Gamma)\ra E(\Gamma)$ such that $u(e)=e$, and  the functions $s,t\col \ora{E}(\Gamma)\ra V(\Gamma)$ (the \emph{source and target}), such that
\begin{itemize}
    \item[(1)] $\{s(e),t(e)\}$ is the set of vertices incident to $u(e)$, $\forall\;e\in \ora{E}(\Gamma)$;
    \item[(2)] $s(e_1)=t(e_2)$ for oriented edges $e_1\in E_1$, $e_2\in E_2$ such that  $u(e_1)=u(e_2)\in E(\Gamma)$.
\end{itemize}  
For every subset $\E\subset E(\Gamma)$, we let $\ora{\E}=u^{-1}(\E)\subset \ora{E}(\Gamma)$.

An \emph{oriented cycle} of $\Gamma$ is a sequence of oriented edges $e_1,\dots,e_n$ such that $t(e_n)=s(e_1)$ and $t(e_i)=s(e_{i+1})$ for every $i=1,\dots,n-1$ and whose associated set of edges forms a cycle of $\Gamma$.

A flow on a graph $\Gamma$ is a function $\phi\col \ora{E}(\Gamma)\ra \mathbb Z$ such that $\phi(e_1)+\phi(e_2)=0$  for different $e_1,e_2\in \ora{E}(\Gamma)$ such that $u(e_1)=u(e_2)\in E(\Gamma)$. Notice that $\phi(e_1)=0$ if and only if $\phi(e_2)=0$, so, sometimes, we will abuse notation writing $\phi(e)=0$, where $e=u(e_1)=u(e_2)\in E(\Gamma)$. If $v\col E(\Gamma)\ra \ora{E}(\Gamma)$ is a section of $u$, we will often define $\phi$ just defining it over the image of $v$.

Given a flow $\phi$ on a graph $\Gamma$ and vertices $v_1,v_2\in V(\Gamma)$, a \emph{$\phi$-path} from $v_1$ to $v_2$ is a sequence $e_1,\dots,e_n\in\ora{E}(\Gamma)$ such that $s(e_1)=v_1$, $t(e_n)=v_2$, $t(e_i)=s(e_{i+1})$ for $i=1,\dots,n-1$ and $\phi(e_i)\ge0$ for $i=1,\dots,n$. A $\phi$-path is called \emph{positive} (\emph{null}) if $\phi(e_i)>0$ (if $\phi(e_i)=0$) for every $i$. A \emph{path} is a $\phi$-path for $\phi=0$. A \emph{$\phi$-cycle} is a $\phi$-path whose edges form a cycle of $\Gamma$; it is called \emph{positive} ({\em null}) if it is so as a $\phi$-path.  A \emph{$\phi$-loop} is a $\phi$-cycle consisting of a single oriented edge.
A flow $\phi$ is \emph{acyclic} if there are no  non-null $\phi$-cycles in $\Gamma$. If $\phi$ is acyclic and $\phi(e)\ne 0$ for every $e\in \ora{E}(\Gamma)$, then there is at least a \emph{source} $v$ and a \emph{sink} $v'$ of $\phi$, that is, $\phi(e)>0$ if $s(e)=v$ and $\phi(e)>0$ if $t(e)=v'$.

A \emph{$\phi$-cut} of $\Gamma$ is a subset $E(W_1,W_2)\subset E(\Gamma)$ for some disjoint subsets $W_1,W_2\subset V(\Gamma)$ such that for every  $v_1\in W_1$ and $v_2\in W_2$ there is no $\phi$-path from $v_1$ to $v_2$ and from $v_2$ to $v_1$ whose edges are  contained in $\ora{E}(\Gamma)\setminus\ora{E}(W_1,W_2)$. 

 We define the graph $\Gamma/\E$ obtained by the contraction of the edges of $\E$. Notice that there is a natural surjection $V(\Gamma)\ra V(\Gamma/\E)$ and an identification $E(\Gamma/\E)=E(\Gamma)\setminus \E$. 
A graph $\Gamma$ \emph{specializes} to a graph $\Gamma'$, and we write $\iota\col\Gamma\ra \Gamma'$, if there is an isomorphism $\Gamma'\cong \Gamma/\E$. A specialization $\Gamma\ra \Gamma'$ is equipped with a surjective map $\iota\col V(\Gamma)\ra V(\Gamma')$ and an injective map $E(\Gamma')\ra E(\Gamma)$.

Let $\Gamma$ be a graph. 
A \emph{divisor} $D$ on $\Gamma$ is a function $D\col V(\Gamma)\ra \mathbb Z$. We often denote a divisor by $D=\sum_{v\in V(\Gamma)}D(v)v$. The degree of $D$ is the integer $\deg D=\sum_{v\in V(\Gamma)} D(v)$. We denote by $\Div(\Gamma)$ the group of divisors of $\Gamma$. A \emph{polarization} $\mu$ on $\Gamma$ is a function $\mu\col V(\Gamma)\ra \mathbb{R}$ such that $\sum_{v\in V(\Gamma)} \mu(v)$ is an integer, called the \emph{degree} of the polarization. Given a flow $\phi$ on $\Gamma$, the \emph{divisor of $\phi$} is defined as $\Div(\phi)=\sum_{e\in \ora{E}(\Gamma)} \phi(e) t(e)$. 
For a vertex $v\in V(\Gamma)$, we define 
%the divisor $\Div(v)=\sum_{\{e\in \ora{E}(\Gamma)|s(e)=v\}} (t(e)-s(e))$ on $\Gamma$ and 
the acyclic flow $\phi_v$ such that for $e\in \ora{E}(\Gamma)$ we have
\[
\phi_v(e)=\begin{cases}
\begin{array}{cl}
1     & \text{ if } s(e)=v\ne t(e) \\
 -1    & \text{ if } t(e)=v\ne s(e) \\
 0 & \text{ otherwise}.
\end{array}
\end{cases}
\]
 We set $\Div(v):=\Div(\phi_v)$. A \emph{principal divisor} on $\Gamma$ is a divisor of type $\sum_{v\in V(\Gamma)} a_v\Div(v)$, for $a_v\in \mathbb Z$. Of course, this divisor is equal to $\Div(\phi)$, where $\phi$ is the %flow associated to such a principal divisor is the 
acyclic flow $\phi=\sum_{v\in V(\Gamma)} a_v \phi_v$.

Given a specialization $\iota \col \Gamma\ra \Gamma'$ of graphs, a divisor $D$ and a polarization $\mu$ on $\Gamma$, we define the divisor and polarization $\iota_*(D)$ and $\iota_*(\mu)$ on $\Gamma'$ taking every $v'\in V(\Gamma')$ to
\[
\iota_*(D)(v')=\sum_{v\in \iota^{-1}(v')} D(v)
\;\text{ and }\;
\iota_*(\mu)(v')=\sum_{v\in \iota^{-1}(v')} \mu(v).
\]
Given a flow $\phi$ on $\Gamma$, 
we let $\iota_*\phi$ be the flow on $\Gamma'$ such that $\iota_*\phi(e)=\phi(e)$, $\forall\; e\in E(\Gamma')\subset E(\Gamma)$.

Let $D$ and $\mu$ be a divisor and a polarization on $\Gamma$ of degree $d$. For every subset $W\subset V(\Gamma)$, we set $\delta_{\Gamma,W}:=\#E(W,W^c)$ (or, simply, $\delta_W$) and $\mu(W)=\sum_{v\in W} \mu(v)$, and we define:
\[
\beta_D(W):=\deg D|_W-\mu(W)+\frac{\delta_{\Gamma,W}}{2}.
\]
We say that the divisor $D$ on a connected graph $\Gamma$ is \emph{$\mu$-stable} (respectively, \emph{$\mu$-semistable}), if $\beta_D(W)>0$ (respectively, $\beta_D(W)\ge0$) for every $W\subsetneq V(\Gamma)$. Moreover, $D$ is  \emph{$(v_0,\mu)$-quasistable}, for some $v_0\in V(\Gamma)$, if $\beta_D(W)\ge0$ for every $W\subsetneq V(\Gamma)$, with strict inequality if $v_0\in W$.

 If $\Gamma$ is a non-connected graph and  $\mu$ is  a degree-$d$ polarization on $\Gamma$, the notion of $\mu$-semistability naturally extends for divisors on $\Gamma$ (see \cite[Remark 3.1]{AAPT}).  We say that a divisor $D$ of degree $d$ on $\Gamma$ is $\mu$-\emph{stable} if it is $\mu$-semistable and $D|_{\Gamma'}$ is $\mu|_{\Gamma'}$-stable for every connected component $\Gamma'$ of $\Gamma$.

\begin{Prop}\label{prop:stab-contract}
Let $\iota\col \Gamma\ra \Gamma'$ be a specialization of graphs. Let $D$ and $\mu$ be a divisor and a polarization of degree $d$ on $\Gamma$. We have $\beta_{\iota_*(D)}(W)=\beta_D(\iota^{-1}(W))$ for every subset $W\subset V(\Gamma')$. In particular, if $D$ is $\mu$-stable, $\mu$-semistable, $(v_0,\mu)$-quasistable for some $v_0\in V(\Gamma)$, then $\iota_*(D)$ is $\iota_*(\mu)$-stable, $\iota_*(\mu)$-semistable, $(\iota(v_0),\iota_*(\mu))$-quasistable, respectively. 
\end{Prop}

\begin{proof}
For every subset $W\subset V(\Gamma')$ we have
\begin{align*}
\beta_{\iota_*(D)}(W) & =\deg(\iota_*(D)|_W)-\iota_*(\mu)(W)+\frac{\delta_W}{2}\\
& = \deg(D|_{\iota^{-1}(W)})-\mu(\iota^{-1}(W))+\frac{\delta_{\iota^{-1}(W)}}{2}\\
& = \beta_D(\iota^{-1}(W)).
\end{align*}
The result follows.
\end{proof}

Given a graph $\Gamma$ and a subset $\E\subset E(\Gamma)$, we let $\Gamma^\E$ be the graph obtained by inserting exactly one vertex in the interior of every edge $e\in \E$. Notice that there is a natural inclusion $V(\Gamma)\subset V(\Gamma^\E)$. We call \emph{exceptional} a vertex in $V(\Gamma^\E)\setminus V(\Gamma)$. 
An edge which is incident to an exceptional vertex in the interior of an edge $e\in E(\Gamma)$ is called \emph{exceptional} and we say that it is an edge \emph{over} $e$. 
%If two edges $e_1,e_2\in E(\Gamma^\E)$ are incident to an exceptional vertex in the interior of an edge $e\in E(\Gamma)$, then we say that $e_1,e_2$ are \emph{over} $e$.  
A \emph{pseudo-divisor} on $\Gamma$ is a pair $(\E,D)$, where $\E\subset E(\Gamma)$ and $D$ is a divisor on $\Gamma^\E$ such that $D(v)=-1$ for every exceptional vertex $v$. 

More generally, a \emph{refinement} of a graph $\Gamma$ is a graph $\Gamma'$ obtained by inserting a certain number of vertices in the interior of each edge of $\Gamma$. We have a natural injective map $\Div(\Gamma)\ra \Div(\Gamma')$ induced by the inclusion $V(\Gamma)\subset V(\Gamma')$.

We say that a pseudo-divisor  $(\E,D)$ on $\Gamma$ is \emph{simple} if $\Gamma_\E$ is connected. We say that $(\E,D)$ is $\mu$-stable, $\mu$-semistable, $(v_0,\mu)$-quasistable for some $v_0\in V(\Gamma)$, if so is $D$ on $\Gamma^\E$ with respect to the polarization $\mu^\E$ on $\Gamma^\E$ extending $\mu$ such that $\mu^\E(v)=0$ for every exceptional vertex $v\in V(\Gamma^\E)\setminus V(\Gamma)$.

Let $D$ be a divisor and $\mu$ a polarization on a graph $\Gamma$. For a subset $\E\subset E(\Gamma)$, we let $D_\E$  and $\mu_\E$ be the divisor and the polarization on $\Gamma_\E$ such that $D_\E(v)=D(v)$ and $\mu_\E(v)=\mu(v)+\frac{1}{2}\deg_\E(v)$. For every subset $W\subset V(\Gamma_\E)=V(\Gamma)$, we set
\[
\beta_{\E,D}(W)=\deg(D|_W)-\mu_\E(W)+\frac{\delta_{\Gamma_\E,W}}{2}.
\]
 We say that $(\E, D)$ is $\mu$-polystable
if $\beta_{\E,D}(V)\ge0$ for every subset $V\subset V(\Gamma)$, with strict inequality if $E(V,V^c)\not\subset \E$. Equivalently, $(\E, D)$ is $\mu$-polystable if $D_\E$ is $\mu_\E$-stable on $\Gamma_\E$.

Given a specialization $\iota\col \Gamma\ra \Gamma'$ and a subset $\E\subset E(\Gamma)$ we have a specialization $\wh\iota\col\Gamma^\E\ra \Gamma'^{\E'}$, where $\E'=\E\cap E(\Gamma')$; we say that $\wh\iota$ is \emph{induced} by $\iota$. If $(\E,D)$ is a pseudo-divisor on $\Gamma$, we define the pseudo-divisor  $\iota_*(\E,D)=(\E',\wh\iota_*(D))$ on $\Gamma'$. 

Given a specialization of graphs $\iota\col \Gamma\ra \Gamma'$ and subsets $\E\subset E(\Gamma)$, $\E'\subset E(\Gamma')$, we say that a specialization $\wh \iota\col \Gamma^\E\ra {\Gamma'}^{\E'}$ is $\iota$-compatible, if the  diagrams:
\[
\SelectTips{cm}{11}
\begin{xy} <16pt,0pt>:
\xymatrix{
  V(\Gamma)  \ar[d] \ar[r]^{\iota}     & V(\Gamma') \ar[d]  \\
V(\Gamma^\E)\ar[r]^{\wh\iota}&V({\Gamma'}^{\E'})
 }
\end{xy}
\;\;\;\;\;\;\;\;
\SelectTips{cm}{11}
\begin{xy} <16pt,0pt>:
\xymatrix{
  E(\Gamma')   \ar[r]^{\iota}     & E(\Gamma)   \\
E({\Gamma'}^{\E'})\ar[r]^{\wh\iota} \ar[u] &E(\Gamma^{\E})\ar[u]
 }
\end{xy}
\]
are commutative. If $\iota$ is the identity, we simply say that $\wh\iota$ is compatible. 
  
Given pseudo-divisors $(\E,D)$ on $\Gamma$ and $(\E',D')$ on $\Gamma'$, we say that the triple $(\Gamma,\E,D)$ \emph{specializes} to the triple $(\Gamma',\E',D')$, and we write $(\Gamma,\E,D)\ge (\Gamma',\E',D')$, if there is a specialization $\iota\col \Gamma\ra \Gamma'$ such that $\E'\subset \E\cap E(\Gamma')$, and a $\iota$-compatible specialization $\wh\iota\col \Gamma^\E\ra \Gamma^{\E'}$ such that $\wh\iota_*(D)=D'$. When $\Gamma=\Gamma'$, we will omit $\Gamma$ in the notation $(\Gamma,\E,D)\ge (\Gamma',\E',D')$.

Notice that a specialization of triple $(\Gamma,\E,D)\ge (\Gamma',\E',D')$ is determined by the associated specialization 
%If $\Gamma$ and $\Gamma'$ are graphs and 
$\wh \iota\col\Gamma^\E\ra {\Gamma'}^{\E'}$. Indeed, the unique specialization $\iota\col \Gamma\ra \Gamma'$ for which $\wh\iota$ is $\iota$-compatible is the one contracting the edges $e\in E(\Gamma)$ such that all the edges of $\Gamma^\E$ over $e$ are contracted by $\wh\iota$.

\begin{Rem}\label{rem:minimal-pol}
Given a $\mu$-semistable pseudo-divisor $(\E,D)$ on a graph $\Gamma$, there is a unique minimal $\mu$-polystable pseudo-divisor on $\Gamma$, denoted $\pol(\E,D)$, such that $\pol(\E,D)\ge(\E,D)$ (see \cite[Proposition 5.4]{AAPT}).  
\end{Rem}

\subsection{Tropical curves}

An $n$-pointed tropical curve is a pair $X = (\Gamma, \ell)$ where $\Gamma$ is a graph with $n$ legs and $\ell\col E(\Gamma)\ra \mathbb R_{>0}$ is a function (the \emph{length function}). The graph $\Gamma$ is called a \emph{model} of $X$. We will consider tropical curves up to isometries. Given a point $p\in X$, we define $\deg_X(p)=\deg_\Gamma(p)$,  $w_X(p)=w_\Gamma(p)$ and $L_X(p)=L_\Gamma(p)$ if $p\in V(\Gamma)$, otherwise we set $\deg_X(p)=2$,  $w_X(p)=0$ and $L_X(p)=0$. The \emph{genus} of $X$ is $g=\sum_{v\in X} w_X(v)+b_1(\Gamma)$. We say that $X$ is \emph{stable} if, for every $p\in X$ such that $\deg_X(p)\le1$, we have $\deg_X(p)+2w_X(p)+L_X(p)\ge3$.

Let $X$ be a tropical curve. A \emph{divisor} on  $X$ is a function $\mc D\col X\ra \mathbb Z$ such that $\mc D(p)\ne 0$ only for a finite set of points $p\in X$ (called the \emph{support} of $\mc D$ and written $\supp(\mc D)$). We often denote a divisor by $\mc D=\sum_{p\in X}\mc D(p)p$. The \emph{degree} of $\mc D$ is the integer $\sum_{p\in X} \mc D(v)$.  The \emph{canonical divisor} of $X$ is the divisor  $K_X=\sum_{p\in X}(w_X(p)-2+\deg_X(p))p$. The degree of $K_X$ is $2g-2$. We will implicitly consider a fixed model $\Gamma$ of $X$ such that the support of $K_X$ is contained in $V(\Gamma)$.

A divisor $\mc D$ on $X$ is \emph{$\Gamma$-unitary} (or simply \emph{unitary}), if
for every $e\in E(\Gamma)$ we have $\mc D(p)=0$ for each point $p\in e^\circ$, except for at most one point $p_0\in e^\circ$ for which $\mc D(p_0)=-1$. Given a $\Gamma$-unitary divisor $\mc D$ on $X$, the \emph{combinatorial type} of $\mc D$ is the pseudo-divisor $(\E,D)$ on $\Gamma$ such that
\[
\E=\{e\in E(\Gamma) | D(p)=-1, \text{ for some } p\in e^\circ\},
\]
while $D(v)=\mathcal D(v)$, for every $v\in V(\Gamma)$, and $D(v)=-1$, $\forall$ exceptional vertex $v\in V(\Gamma^\E)$. 

One can introduce stability conditions on divisors on a tropical curve (see \cite[Section 5]{AAPT} and \cite[Section 5]{AAPT}). It is easy to see that a quasistable or polystable divisor is always unitary.  In this paper we 
%By  , we can 
use the equivalent definition that a unitary divisor $\mc D$ on $X$ is \emph{$\mu$-stable}, \emph{$\mu$-semistable}, \emph{$(p_0,\mu)$-quasistable} for a point $p_0\in V(\Gamma)$, \emph{$\mu$-polystable} if so is its combinatorial type $(\E,D)$, respectively (see \cite[Proposition 5.3]{AP1} and \cite[Definition 5.8]{AAPT}).

A \emph{rational function} on $X$ is a piecewise linear continuous function $f\col X\ra \mathbb R$ with integer slopes. Given a point $p\in X$, we let $\ord_p(f)$ be the sum of the slopes of $f$ outgoing $p$. 
A \emph{principal divisor} on $X$ is a divisor $\Div(f)=\sum_{p\in X} \ord_p(f)p$, for some rational function $f$ on $X$. Given divisors $D$ and $D'$ on $X$, we say that $D$ is equivalent to $D'$, written $D\sim D'$, if $D-D'$ is a principal divisor. Given a rational function $f$ on $X$ and a model $\Gamma$ such that the support of $\Div(f)$ is contained in $V(\Gamma)$, we define a flow $\phi_f$ on $\Gamma$ such that $\phi_f(e)$ is the slope of $f$ on $e$, for every $e\in \ora{E}(\Gamma)$.

\begin{Rem}\label{rem:stab-equiv}
By \cite[Theorem 5.6]{AP1} and \cite[Theorem 5.9]{AAPT} given a divisor $\mc D$ on $X$ there is a $(p_0,\mu)$-quasistable divisor $\mc D'$ and a $\mu$-polystable divisor $\mc D''$ on $X$ such that $\mc D$, $\mc D'$, $\mc D''$ are equivalent; the divisor $\mc D'$ is uniquely determined.
\end{Rem}

\section{Tropical moduli spaces}

In this section we recall the tropical moduli space we will deal with. First of all we need to recall the notion of generalized cone complex.

Given a finite set $S\subset \mathbb R^n$, we let
\[
\cone(S):=\Big\{\sum_{s\in S} \lambda_s s | \lambda_s\in \mathbb R_{\ge0}\Big\}.
\]
A \emph{polyhedral cone} is a subset $\sigma\subset \mathbb R^n$ such that $\sigma=\cone(S)$ for some finite set $S\subset \mathbb R^n$. A cone $\sigma=\cone(S)$ is \emph{rational} if $S\subset \mathbb Z^n$.  Every polyhedral cone $\sigma\subset \mathbb R^n$ is the intersection of finitely many closed half spaces of $\mathbb R^n$. We denote by $\dim(\sigma)$ the \emph{dimension} of $\sigma$, that is, the dimension of the minimal linear subspace of $\mathbb R^n$ containing $\sigma$.  
 The relative interior $\sigma^\circ$ of a cone $\sigma$ is the interior of $\sigma$ inside this minimal linear subspace.
A \emph{face} of a cone $\sigma$ is either $\sigma$ or the intersection of $\sigma$ with some linear subspace $H\subset \mathbb R^n$ of codimension
one such that $\sigma$ is contained in one of the two closed half-spaces determined by $H$. A
face of $\sigma$ is also a polyhedral cone. A \emph{facet} is a face of codimension $1$. We will simply use the terminology \emph{cone} to mean \emph{rational polyhedral cone}.

A morphism $f\col \sigma\ra \sigma'$ between cones $\sigma\subset\mathbb R^n$ and $\sigma'\subset \mathbb R^m$ is the restriction to $\sigma$
of an integral linear transformation $T\col \mathbb R^n\ra \mathbb R^m$ such that $T(\sigma)\subset \sigma'$. We say that
$f$ is an isomorphism if there exists an inverse morphism of cones $f^{-1}\col \sigma'\ra \sigma$. 
 A morphism $f\col \sigma\ra \sigma'$ is called a \emph{face morphism} if $f$ is an isomorphism between $\sigma$ and a (not
necessarily proper) face of $\sigma'$.

A \emph{generalized cone complex} $\Sigma$ is the colimit (as a topological space) of a finite diagram of cones, denoted $D$, with face morphisms. A generalized cone complex is written as  $\Sigma=\underset{\longrightarrow}{\lim}\, \sigma_x$, where $x\ra \sigma_x$ is a contravariant functor from a category $\mathbf C$ to the category of rational polyhedral cones. The group $\Aut(x)$ acts on $\sigma_x$ as a subgroup of automorphisms of $\sigma_x$. Notice that the set $\mc P$ of  isomorphism classes of objects of $\mathbf C$ is a poset. 
%Let $\sigma^\circ_x$ be the interior of $\sigma_x$ and set $\Sigma^\circ_x=\sigma^\circ_x/\Aut(x)$. 
%We can present $\Sigma$ as 
%\[
%\Sigma=\coprod_{x\in \mc P} \Sigma^\circ_x=\coprod_{x\in \mc P}\sigma^\circ_x/\Aut(x).
%\]
%Notice that the closure  $\Sigma_x$ of $\Sigma^\circ_x$ has a decomposition $\Sigma_x=\coprod_{x\ge x'} \Sigma^\circ_{x'}$.

A \emph{morphism} $\Sigma\ra \Sigma'$ between generalized cone complexes $\Sigma$ and $\Sigma'$ with diagrams of cones $D$ and $D'$ is a continuous map of topological spaces  such that for every cone $\sigma\in D$ there exists a cone $\sigma'\in D'$ such that the induced map $\sigma\ra\Sigma'$  factors through a cone morphism $\sigma\ra \sigma'$.
We say that a morphism of generalized cone complexes $ \Sigma\ra \Sigma'$ is a \emph{refinement} if it is a bijection.

\begin{Def}
Let $\mc M$ be a Deligne-Mumford stack and $\mc P$ be a poset.  A \emph{stratification of $\mc M$ indexed by $\mc P$} is a decomposition $\mc M=\coprod_{x\in \mc P} \mc M_x$ of $\mc M$ such that 
\begin{enumerate}
    \item $\mc M_x$ is an irreducible substack of $\mc M$;
    \item the closure $\ol{\mc M}_x$ of $\mc M_x$ has a decomposition $\ol{\mc M}_x=\coprod_{x'\ge x} \mc M_{x'}$.
\end{enumerate}
 We call $\mc M_x$ a \emph{stratum} of the stratification. 
 %We say that a generalized cone complex $\Sigma$ with associated poset $\mc P$  is \emph{dual} to $\mc M$ if   $\dim \mc M_p=\codim \Sigma_p$.
\end{Def}

%If $\Sigma$ is dual to a moduli space $\mc M$, then we can view $\Sigma$ as a tropical version of $\mc M$. In general, it is interesting to give $\Sigma$ a modular interpretation in terms of the combinatorial data encoded by the objects of $\mc M$. Morphisms of generalized cone complexes...

\subsection{The moduli space of tropical curves}

Generalized cone complexes are often used to parametrize tropical objects. The prototype example is the moduli space of pointed tropical curve $M_{g,n}^{\trop}$. 
The moduli space $M_{g,n}^{\trop}$, first constructed
in \cite{BMV}, has the structure of a generalized cone complex (see \cite{ACP}).

%The points in $M_{g,n}^{\trop}$ are in bijective correspondence with (isomorphism classes of) tropical curves of genus $g$.

We let $\mathbf G_{g,n}$ be the category whose objects are stable graphs of genus $g$ with $n$ legs and whose arrows are specialization of graphs. We let  $\mathcal{G}_{g,n}=\mathbf G_{g,n}/\sim$ be the poset whose elements are isomorphism classes of graphs in $\mathbf G_{g,n}$.

Given a graph $\Gamma$ we define the cone $\tau_\Gamma:=\mathbb R^{E(\Gamma)}_{\ge0}$ and its open subcone $\tau_\Gamma^\circ:=\mathbb R^{E(\Gamma)}_{>0}$. Notice that given a specialization of graphs $\Gamma\ra \Gamma'$, then we have a natural face morphism $\tau_{\Gamma'}\ra \tau_\Gamma$
%obtained by intersecting $\tau_\Gamma$ with the linear space  with equations $x_e=0$,
whose image is the face of $\tau_\Gamma$ where the coordinates
 corresponding to $E(\Gamma) \setminus E(\Gamma')$ vanish.
%for every $e\in E(\Gamma)\setminus E(\Gamma')$, where $x_e$ is the coordinate of $\mathbb R^{E(\Gamma)}$ corresponding to $e$. 
We define the generalized cone complex:
\[
M_{g,n}^{\trop}=\underset{\longrightarrow}{\lim}\;\tau_\Gamma=\coprod_{[\Gamma]}\tau^\circ_\Gamma/\Aut(\Gamma)
\]
where the limit is taken over the category $\mathbf{G}_{g,n}$ and the union is taken over all equivalence classes $[\Gamma]$ in $\mc G_{g,n}$.
The points in $M_{g,n}^{\trop}$ are in bijective correspondence with (isomorphism classes of) tropical curves of genus $g$.

\smallskip

Let $\ol{\mc M}_{g,n}$ be the moduli stack of $n$-pointed stable curves, and $\mc M_{g,n}$ be the locus corresponding to smooth curves. Given an $n$-pointed  stable curve $X$, we let $\Gamma_X$ be the usual dual graph of $X$, whose vertices and edges are the components and the nodes of $X$, respectively. For every graph $\Gamma$ with $n$ legs, we define the locus $\mc M_\Gamma$ in $\ol{\mc M}_{g,n}$ given by:
\[
\mc M_\Gamma:=\{[X]\in \ol{\mc M}_{g,n} \big| \Gamma_X\cong \Gamma\}.
\]
 
    We have a stratification $\ol{\mc M}_{g,n}=\sqcup_{\Gamma\in \mc G_{g,n}} \mc M_\Gamma$ indexed by $\mc G_{g,n}$.
Recall that $\mc M_\Gamma$ is irreducible thanks to the existence of a natural \'etale finite map $\widetilde{\mc M}_\Gamma\ra \mc M_\Gamma$, where  
\begin{equation}\label{eq:MGamma-tilde}
\widetilde{\mc M}_\Gamma=\prod_{v\in V(\Gamma)} \mc M_{w_\Gamma(v),\deg_\Gamma(v)+L_\Gamma(v)}.
\end{equation}

\subsection{Universal tropical Jacobians}

A \emph{universal genus-$g$ polarization $\mu$ (of degree $d$)} is the datum of a polarization $\mu_{\Gamma}$ (of degree $d$) for every stable genus-$g$ graph with $n$ legs such that $\iota_*(\mu_{\Gamma})=\mu_{\Gamma'}$, for every specialization of graphs $\iota\col\Gamma\ra \Gamma'$. The \emph{universal genus-$g$ canonical polarization of degree $d$} takes every graph $\Gamma$ and a vertex $v\in V(\Gamma)$ to:  
\[
\mu_\Gamma(v):=\frac{d(2w_\Gamma(v)-2+\deg_\Gamma(v))}{2g-2}=\frac{d}{2g-2}K_\Gamma(v).
\]

Let $\mu$ be a universal genus-$g$ polarization.
We denote by $\mathbf{QD}_{\mu,g,n}$ the category whose objects are triples $(\Gamma,\E,D)$, where $\Gamma$ is a stable graph of genus $g$ with $n$ legs, and $(\E,D)$ is a $(v_0,\mu)$-quasistable pseudo-divisor on $\Gamma$ (the first leg is incident to $v_0$), and whose arrows are given by specialization of triples. We let $\mathcal{QD}_{\mu,g,n}=\mathbf{QD}_{\mu,g,n}/\sim$ be the poset of the isomorphism classes of triples $(\Gamma,\E,D)$.

Given a graph 
 $\Gamma$ and a $v_0$-quasistable  
 %and $\mu$ a degree-$d$ polarization on $\Gamma$.
  pseudo-divisor $(\E,D)$ on $\Gamma$, we define
\[
\tau_{(\Gamma,\E,D)}:=\R^{E(\Gamma^\E)}_{\geq0}\quad\text{and}\quad\tau^\circ_{(\Gamma,\E,D)}:=\R^{E(\Gamma^\E)}_{>0}.
\]
Notice that, given a specialization $i\col(\Gamma,\E,D)\ge(\Gamma',\E',D')$, we have a natural face morphism  $\tau_{(\Gamma',\E',D')}\ra\tau_{(\Gamma,\E,D)}$.

\begin{Def}
Let $\mu$ be a universal genus-$g$  polarization. The \emph{universal tropical Jacobian of $\mu$-quasistable divisors} is the generalized cone complex:
\[
J^\trop_{\mu,g,n}:=\underset{\longrightarrow}{\lim}\;\tau_{(\Gamma,\E,D)}=\coprod_{[\Gamma,\E,D]}\tau^\circ_{(\Gamma,\E,D)}/\Aut(\Gamma,\E,D),
\]
where the limit is taken over the category $\mathbf{QD}_{\mu,g,n}$ and the union is taken over all equivalence classes $[\Gamma,\E,S]$ in $\qs_{\mu,g,n}$ (see  \cite[Definition 5.11]{AP1}).
\end{Def}

By \cite[Proposition 5.12]{AP1} the cone 
complex $J^\trop_{\mu,g,n}$ parametrizes equivalence classes $(X, \mc D)$, where X is a stable $n$-pointed tropical curve of genus $g$ and $\mc D$ is a 
$(p_0,\mu)$-quasistable divisor on $X$. 
By \cite[Theorem 5.14]{AP1} there is a morphism of generalized cone complexes
\[
\pi^\trop_J\col J_{\mu,g,n}^\trop\ra M_{g,n}^\trop
\]
taking a class $(X,\D)$ to the tropical curve $X$. Indeed, we have the linear map: 
\begin{equation}\label{eq:sE}
\mc S_\E\col \mathbb R^{E(\Gamma^\E)}\ra\mathbb R^{E(\Gamma)}
\end{equation}
taking $(x_e)_{e\in E(\Gamma^\E)}$ to $(y_e)_{e\in E(\Gamma)}$, where $y_e=x_{e_1}+x_{e_2}$ if $e_1,e_2\in E(\Gamma^\E)$ are over $e$, and $y_e=x_e$, otherwise. This map induces a morphism of cones $\tau_{(\Gamma,\E,D)}\ra\tau_\Gamma$  factoring the natural composition
\[
\tau_{(\Gamma,\E,D)}\ra J_{\mu,g,n}^\trop\stackrel{\pi^\trop_J}{\ra} M_{g,n}^\trop.
\]

Consider a graph $\Gamma$ and a subset $\E\subset E(\Gamma)$. For each subset $V \subset V (\Gamma)$ such that
$E(V, V^c)\subset \E$, we define $\alpha_V\col E(\Gamma^\E) \ra \{0, \pm1\}$ as $\alpha_V (e) = 1$ (respectively, $-1$), if $e$ is
incident to $V$ (respectively, to $V^c$) and lies over an edge in $E(V, V^c)$. Otherwise, we
define $\alpha_V (e) = 0$. Denote by $u_e$ the vectors of the canonical basis of $\R^{E(\Gamma^\E)}$, for $e\in E(\Gamma^\E)$. We define the vector in $\R^{E(\Gamma^\E)}$:
\begin{equation}\label{eq:wV}
u_V =\sum_{e\in E(\Gamma^\E )}
\alpha_V (e)u_e.
\end{equation}
 We let $\mathcal{L}_\E\subset \R^{E(\Gamma^\E)}$
be the subspace generated by all vectors $u_V$, as $V$ runs through all subsets of $V (\Gamma)$
such that $E(V,V^c) \subset \E$. Notice that $\dim \mc L_\E=b_0(\Gamma_\E)-1$.  Of course, if $(\E,D)$ is simple, then $\mathcal{L}_\E=0$.  We also denote by
\begin{equation}\label{eq:proy}
   \mathcal{T}_{\E}\col\R^{E(\Gamma^\E)}\ra\R^{E(\Gamma^\E)}/\mathcal{L}_\E 
\end{equation}
the quotient linear map.  
For a pseudo-divisor $(\E,D)$ on $\Gamma$, we define the cone:
$$\sigma_{(\Gamma,\E,D)}=\mathcal{T}_\E(\mathbb R^{E(\Gamma^\E)}_{\ge0})\subset\R^{E(\Gamma^\E)}/\mathcal{L}_\E.$$

\begin{Rem}\label{rem:inclu}
If  $(\Gamma,\E,D)\geq(\Gamma',\E',D')$ is a specialization, then $\mathcal{L}_{\E'}=\mathcal{L}_{\E}\cap\R^{E(\Gamma^{\E'})}$, and hence we have a natural inclusion $\sigma_{(\Gamma',\E',D')}\subset\sigma_{(\Gamma,\E,D)}$ which is a face morphism. By \cite[Proposition 6.10]{AAPT}, every face of $\sigma_{(\Gamma,\E,D)}$ arises in this way.
\end{Rem}

If $\mu$ is a genus-$g$ universal polarization, we denote by $\PS_{\mu,g,n}$ the category
whose objects are triples $(\Gamma, \E, D)$ where $\Gamma$ is a genus-$g$ stable graph with $n$ legs and $(\E, D)$ is a $\mu$-polystable pseudo-divisor on $\Gamma$, and morphisms are given by specializations of triples. As usual, we let $\ps_{\mu,g,n}$ be the poset $\PS_{\mu,g,n}/ \sim$, where $\sim$ means isomorphism.

\begin{Def}
Let $\mu$ be a genus-$g$ universal polarization. The \emph{universal tropical Jacobian of $\mu$-polystable divisors} is the generalized cone complex:
\[
P^\trop_{\mu,g,n} = \underset{\longrightarrow}{\lim}\;\sigma_{(\Gamma,\E,D)} =\coprod_{[\Gamma,\E,D]}\sigma^0_{(\Gamma,\E,D)}/\Aut(\Gamma, \E, D),
\]
where the limit is taken over the
category $\PS_{\mu,g,n}$ and the union is taken over all equivalence classes $[\Gamma, \E, D]$ in
$\ps_{\mu,g,n}$.
\end{Def}

\begin{Prop}\label{prop:modularP}
The generalized cone complex $P^\trop_{\mu,g,n}$ parametrizes equivalence classes of pairs $(X, \D)$, where $X$ is a stable $n$-pointed tropical curve of genus $g$ and $\D$ is a $\mu$-polystable
divisor on $X$, and two pairs $(X,\D)$ and $(X', \D')$ are equivalent if there is an isomorphism $i\col X\ra X'$ such that $i_*(\D)$ is linearly equivalent to $\D'$.
\end{Prop}

\begin{proof}
The argument is the same as in \cite[Proposition 6.14]{AAPT}.
\end{proof}

By Proposition \ref{prop:modularP}, we can define a natural forgetful morphism of cone complexes 
\[
\pi^\trop_P\col P_{\mu,g,n}^\trop\ra M_{g,n}^\trop
\]
taking the class of a pair $(X,\D)$ to the class of the tropical curve $X$. Indeed, the map of cones  
%$s_{(\Gamma,\E,D)}\col 
$\tau_{(\Gamma,\E,D)}\ra \tau_\Gamma$ induced by equation \eqref{eq:sE} factors through a map of cones $\sigma_{(\Gamma,\E,D)}\ra \tau_\Gamma$ and the last map factors the natural composition:
\[
\sigma_{(\Gamma,\E,D)}\ra P_{\mu,g,n}^\trop\stackrel{\pi^\trop_P}{\ra} M_{g,n}^\trop.
\]

We can also define the map
\[
\rho^\trop_P\col J^\trop_{\mu,g,n}\ra P^\trop_{\mu,g,n}
\]
taking the point parametrizing the class of a pair $(X,\D)$ to the point parametrizing the class of the pair $(X,\D')$, where $\D'$ is a polystable divisor linearly equivalent to $\D$ (the existence follows by \cite[Theorem 5.9]{AAPT}).

The map  $\rho^\trop_P$ is not only a morphism of generalized cone complexes: it is a refinement. Before giving the precise statement, we need an easy lemma.

\begin{Lem}\label{lem:polyspec}
Let $(\Gamma,\E,D)\ge(\Gamma',\E',D')$ be a specialization of triples. Let $\iota\col \Gamma\ra \Gamma'$ be the induced specialization of graphs. If $(\E,D)$ is $\mu$-polystable on $\Gamma$ and $(\E',D')$ is $(v_0,\iota_*(\mu))$-quasistable on  $\Gamma'$, then there is a specialization of triples $(\Gamma,\E,D)\ge(\Gamma',\pol(\E',D'))$ factoring $(\Gamma,\E,D)\ge (\Gamma',\E',D')$.
\end{Lem}

\begin{proof}
We have specializations $(\Gamma,\E,D)\ge (\Gamma',\iota_*(\E,D))\ra (\Gamma',\E',D')$. By \cite[Lemma 5.2]{AAPT} the pseudo-divisor $\iota_*(\E,D)$ is $\iota_*(\mu)$-polystable on $\Gamma'$. By the minimal property of
$\pol(\E',D')$ in Remark \ref{rem:minimal-pol}, we have a specialization $\iota_*(\E,D)\ge\pol(\E',D')$, concluding the proof.
\end{proof}

\begin{Thm}\label{thm:refiJ}
The map $\rho^\trop_P\col J^\trop_{\mu,g,n}\ra P^\trop_{\mu,g,n}$ is a refinement of generalized cone complexes. For every $\mu$-polystable pseudo-divisor $(\E,D)$ on a stable graph $\Gamma$ with $n$ legs, we have decompositions
\[
\sigma^\circ_{(\Gamma,\E,D)}=\underset{(\E,D)=\pol(\E',D')}{\coprod_{(\Gamma,\E',D')\in\mathbf{QD}_{\mu,g,n}}} \tau^\circ_{(\Gamma,\E',D')}
\quad\text{ and }
\quad
\sigma_{(\Gamma,\E,D)}=\underset{(\Gamma,\E,D)\ge(\Gamma',\E',D')}{\coprod_{(\Gamma',\E',D')\in\mathbf{QD}_{\mu,g,n}}} \tau^\circ_{(\Gamma',\E',D')}.
\]
\end{Thm}

\begin{proof}
First of all, if $\D$ is a $\mu$-polystable divisor on a tropical curve $X$ equivalent to a $(p_0,\mu)$-quasistable divisor $\D'$ on $X$, then the combinatorial type of $\D$ is equal to $\pol(\E',D')$, where $(\E',D')$ is the combinatorial type of $\D'$ (see the proof of \cite[Theorem 5.9 (1)]{AAPT}).  Thus we can describe locally the map $\rho^\trop_P$ as follows. 
Let $(\E',D')$ be a $(v_0,\mu)$-quasistable pseudo-divisor on a stable graph $\Gamma$ with $n$ legs, and set $(\E,D)=\pol(\E',D')$. 
Since  $(\E',D')$ is simple, we have that $\tau_{(\Gamma,\E',D')}=\sigma_{(\Gamma,\E',D')}$. Since $(\E,D)\geq (\E',D')$, by Remark \ref{rem:inclu} we have an inclusion $\tau_{(\Gamma,\E',D')}=\sigma_{(\Gamma,\E',D')}\hookrightarrow\sigma_{(\Gamma,\E,D)}$ which factors the composition
\[
\tau_{(\Gamma,\E',D')}=\sigma_{(\Gamma,\E',D')}\ra J_{\mu,g,n}^\trop\stackrel{\rho^\trop_P}{\ra} P_{\mu,g,n}^\trop.
\]
Then $\rho^\trop_P$ is a morphism of generalized cone complexes. The morphism $\rho^\trop_P$ is also bijective by Remark \ref{rem:stab-equiv},
hence it is a refinement of generalized cone complexes.

Denote $\mathbf{PSD}:=\mathbf{PSD}_{\mu,g,n}$ and $\mathbf{QD}:=\mathbf{QD}_{\mu,g,n}$.
 Given $(\Gamma,\E',D')\in \mathbf{QD}$, by \cite[Proposition 6.9]{AAPT} we have $\tau^\circ_{(\Gamma,\E',D')}=\sigma^\circ_{(\Gamma,\E',D')}\subset \sigma^\circ_{(\Gamma,\pol(\E',D'))}$. Thus, the decomposition of $\sigma^\circ_{(\Gamma,\E,D)}$ is as stated.
 Consider $(\Gamma,\E,D)\in \mathbf{PSD}$. By the definition of $P_{\mu,g,n}^\trop$ and the decomposition just proved, we have:
\begin{align*}
\sigma_{(\Gamma,\E,D)}& =\underset{(\Gamma,\E,D)\ge(\Gamma',\E'',D'')}{\coprod_{(\Gamma',\E'',D'')\in \mathbf{PSD}}} \sigma^\circ_{(\Gamma',\E'',D'')}\\ &=\underset{(\Gamma,\E,D)\ge(\Gamma',\E'',D'')}{\coprod_{(\Gamma',\E'',D'')\in \mathbf{PSD}}}\;\; \underset{(\E'',D'')=
 \pol(\E',D')}{\coprod_{(\Gamma',\E',D')\in \mathbf{QD}}} \tau^\circ_{(\Gamma',\E',D')}.
\end{align*}
By Lemma \ref{lem:polyspec} and by the existence of a specialization $\pol(\E',D')\ge(\E',D')$, the decomposition of $\sigma_{(\Gamma,\E,D)}$ is as stated.
\end{proof}

\section{Root-graphs and their categories}

In this section we introduce the notion of pseudo-root and root-graph, and study  stability conditions on these objects. We will define the category of root-graphs, which will be important in defining later some tropical moduli spaces. 

\subsection{Pseudo-roots on graphs}

Let $g$ be a nonnegative integer.
From now on, $\mu$ will be the canonical polarization of degree $g-1$, defined over a graph $\Gamma$ as the polarization taking $v\in V(\Gamma)$ to
\[
    \mu(v)=\frac{2w(v)-2+\deg(v)}{2}.
\]
 All the stability conditions will be taken with respect to $\mu$.

\begin{Def}
Let $\Gamma$ be a graph. A \emph{pseudo-root} on $\Gamma$ is a pseudo-divisor $(\E,D)$ on $\Gamma$ such that $2D+\Div(\phi)=K_{\Gamma^\E}$, for some acyclic flow $\phi$ on $\Gamma^\E$. 
 A pseudo-root $(\E,D)$ is, respectively, \emph{stable}, \emph{semistable},  \emph{polystable}, \emph{$v_0$-quasistable} for some vertex $v_0$ of $\Gamma$, if so is the pseudo-divisor $(\E,D)$. (Notice that a pseudo-root has degree $g-1$, where $g$ is the genus of its underlying graph.)
\end{Def}

 We will see in Proposition \ref{thm:uniquephi} that $\phi$ is uniquely determined if $(\E,D)$ is semistable. 
Given a pseudo root $(\E,D)$ on a graph $\Gamma$ with associated acyclic flow $\phi$, we define
\begin{equation}\label{eq:PED}
P_{(\E,D)}=\{e\in E(\Gamma^\E)| \phi(e)=0\}.
\end{equation}
%We will often view $P_{(\E,D)}$ as a subgraph of $\Gamma$.

\begin{Rem}\label{rem:iota-pseudo-root}
Let $(\E,D)$ be a pseudo-root on a graph $\Gamma$ and $\phi$ be an acyclic flow on $\Gamma^\E$ such that $2D+\Div(\phi)=K_{\Gamma^\E}$. 
%Define $\E_\phi=\{e\in E(\Gamma^\E) | \phi(e)=0\}\subset E(\Gamma^\E)$. 
Consider the specialization $\iota\col \Gamma^\E\ra \Gamma^\E/P_{(\E,D)}$. Since $\iota_*$ is linear on divisors, we have  $2\iota_*(D)+\Div(\iota_*(\phi))=K_{\Gamma^\E/P_{(\E,D)}}$. Notice that $\iota_*(\phi)$ is acyclic by definition. Then $(\emptyset,\iota_*(D))$ is a pseudo-root on $\Gamma^\E/P_{(\E,D)}$. 
\end{Rem}

\begin{Def}
A \emph{root-graph} is a triple  $(\Gamma,\E,D)$, where $\Gamma$ is a graph and $(\E,D)$ is a pseudo-root on $\Gamma$. A root-graph is \emph{stable}, \emph{semistable},  \emph{polystable}, \emph{$v_0$-quasistable} for some vertex $v_0$ of $\Gamma$, if so is its underlying pseudo-root.
\end{Def}

\begin{Lem}\label{lem:stable}
Let $(\Gamma,\E,D)$ be a root-graph. Let $\phi$ be an acyclic flow on $\Gamma^\E$ such that $2D+\Div(\phi)=K_{\Gamma^\E}$. The pseudo-divisor $(\E,D)$ is stable if and only if $\phi=0$.
\end{Lem}

\begin{proof}
If $(\E,D)$ is stable, then $\E=\emptyset$. By Proposition \ref{prop:stab-contract} and Remark \ref{rem:iota-pseudo-root}, we can assume $\phi(e)\ne 0$, $\forall\;e\in \ora{E}(\Gamma)$. Assume that  $E(\Gamma)\ne\emptyset$ and let $v\in V(\Gamma)$ be a sink of $\phi$. Thus $\Div(\phi)(v)\ge \deg_{\Gamma}(v)$. The relation $2D(v)+\Div(\phi)(v)=K_{\Gamma}(v)$ gives us 
\[
D(v)\le -\frac{\deg_{\Gamma}(v)}{2}+w_\Gamma(v)-1+\frac{\deg_{\Gamma}(v)}{2}=\mu(v)-\frac{\deg_{\Gamma}(v)}{2}.
\]
Hence $\beta_D(v)\le 0$, and $D$ is not  $\mu$-stable, a contradiction.

Conversely, if $\phi=0$, then for every proper subset $W\subset V(\Gamma^\E)$ we have % $D(v)=K_{\Gamma^\E}(v)/2$
\[
\beta_D(W)=
\deg D|_W-\mu(W)+\frac{\delta_W}{2}=\frac{\deg K_{\Gamma^\E}}{2}-\mu(W)+\frac{\delta_W}{2}=\frac{\delta_W}{2}.
\]
Hence $\beta_D(W)>0$ and $(\E,D)$ is stable.
\end{proof}

\begin{Lem}\label{lem:phi-lessone}
Let $(\Gamma,\E,D)$ be a semistable root graph. Let $\phi$ be an acyclic flow on $\Gamma^\E$ such that $2D+\Div(\phi)=K_{\Gamma^\E}$. Then $|\phi(e)|\le 1$ for every $e\in \ora E(\Gamma^\E)$.
\end{Lem}

\begin{proof}
By Proposition \ref{prop:stab-contract} and  Remark \ref{rem:iota-pseudo-root},
we can assume $\phi(e)\ne0$, for every $e\in \ora{E}(\Gamma^\E)$.
By contradiction, suppose that $\phi(e_0)\ge2$ for some $e_0\in\ora{E}(\Gamma^\E)$. 
We let $V$ be the set of vertices $v\in V(\Gamma^\E)$ for which   
%containing $t(e_0)$ and all the vertices $v\in V(\Gamma^\E)$ for which 
there is a positive $\phi$-path from $s(e_0)$ to $v$.
%sequence  $e_1,\dots,e_n$ in $\ora{E}(\Gamma^\E)$ such that $t(e_n)=v$, $t(e_{i-1})=s(e_i)$, and $\phi(e_i)\ge1$, for every $i\in\{1,\dots,n\}$. 
Then $t(e_0)\in V$ and, since $\phi$ is acyclic, we have $s(e_0)\in V^c$, so $V$ is a proper nonempty subset of $V(\Gamma)$. By the definition of $V$, we have $\phi(e)\ge1$ for every $e\in \ora{E}(V,V^c)$ with $t(e)\in V$. Hence:
\begin{equation}\label{eq:degphi}
\deg(\Div(\phi)|_V)=\phi(e_0)+\underset{t(e)\in V}{\sum_{e\in \ora{E}(V,V^c)\setminus\{e_0\}}}\phi(e)>\#E(V,V^c)=\delta_{\Gamma^\E,V}.
\end{equation}
On the other hand, 
%by the $\mu$-semistability of $D$, 
we have 
%\[
%\beta_D(V)=\deg(D|_V)-\mu(V)+\frac{\delta_{\Gamma^\E,V}}{2}\geq0,
%\]
%from which we get:
\[
\deg(\Div(\phi)|_V)=\deg (K_{\Gamma^\E}|_V)-2\deg(D|_V)=2\mu(V)-2\deg (D|_V)\leq\delta_{\Gamma^\E,V},
\]
where we use that $\beta_D(V)\ge0$ by the $\mu$-semistability of $D$, and this contradicts equation \eqref{eq:degphi}.
\end{proof}

\begin{Lem}\label{lem:beta-phi}
Let $(\Gamma,\E,D)$ be a semistable root-graph. Let $\phi$ be an acyclic flow on $\Gamma^\E$ such that $2D+\Div(\phi)=K_{\Gamma^\E}$. Given  a subset $V\subset V(\Gamma^\E)$, we have $\beta_D(V)=0$ if and only if $\phi(e)=1$ for every $e\in \ora{E}(V,V^c)$ such that $t(e)\in V$.
\end{Lem}

\begin{proof}
First of all, we have 
\begin{equation}\label{eq:deg-div-phi}
\deg(\Div(\phi)|_V)=\deg(K_{\Gamma^\E}|_V)-2\deg D|_V=2\mu(V)-2\deg D|_V.
\end{equation}
Thanks to equation \eqref{eq:deg-div-phi}, the condition $\beta_D(V)=\deg D|_V-\mu(V)+\frac{\delta_{\Gamma^\E,V}}{2}=0$ is equivalent to the condition 
$\deg(\Div(\phi)|_V)=\delta_{\Gamma^\E,V}$. 
We are done by Lemma \ref{lem:phi-lessone}.
\end{proof}

\begin{Thm}\label{thm:uniquephi}
Let $(\Gamma,\E,D)$ be a semistable root-graph. Let $\phi$ be an acyclic flow on $\Gamma^\E$ such that $2D+\Div(\phi)=K_{\Gamma^\E}$. Then
\begin{enumerate}
%    \item  for every oriented edge $e\in \ora E(\Gamma^\E)$, we have $\phi(e)\le 1$;
  \item[(1)] $\phi(e)=1$ for every $e\in \ora{E}(\Gamma^\E)$ such that $t(e)$ is an exceptional vertex of $\Gamma^\E$;
  %for every exceptional oriented edge $e\in\ora{E}(\Gamma^\E)$ such that $t(e)$ is an exceptional vertex.
   %\item $\phi(e_1)+\phi(e_2)\equiv 0 \mod(2)$ for $e_1,e_2\in \ora{E}(\Gamma^\E)$ over the same edge of $\Gamma$.
  \item[(2)] $\phi$ is the unique acyclic flow such that $2D+\Div(\phi)=K_{\Gamma^\E}$.
  \end{enumerate}
\end{Thm}

\begin{proof}
Let us prove item (1). Let $e\in \ora{E}(\Gamma^\E)$ such that $t(e)$ is an  exceptional vertex of $\Gamma^\E$. The relation $2D+\Div(\phi)=K_{\Gamma^\E}$ implies that $\Div(\phi)(t(e))=2$. Since, by Lemma \ref{lem:phi-lessone}, we have $\phi(e)\leq1$, we deduce that $\phi(e)=1$.
%$t(e_1)=t(e_2)$ and $\phi(e_1)=\phi(e_2)=1$ for every $e_1,e_2\in\ora{\Gamma}^\E$ over an edge $e\in\E$.

Let us prove item (2). Notice that %Since $(\Gamma,\E,D)$ is a root-graph, we have that 
$(\Gamma^\E,\emptyset,D)$ is a semistable root-graph. So assume $\E=\emptyset$. We proceed by induction on $\#E(\Gamma)$, being the statement clear for $\#E(\Gamma)=0$. 
%Let $\phi'$ be another flow such that $2D+\Div(\phi')$. Since $\phi$ is acyclic, there is a sink $v_0\in V(\Gamma)$ of $\phi$. We have
If $(\emptyset,D)$ is stable, then we are done by Lemma \ref{lem:stable}. So we can assume that there is a proper subset $V\subset V(\Gamma)$ such that $\beta_D(V)=0$.
%\[
%\beta_D(W)=\deg(D|_W)-\mu(W)+\frac{\delta_{\Gamma^\E,W}}{2}=0.
%\]
%We deduce that:
%\[
%\deg(\Div(\phi)|_W)=\deg(K_{\Gamma^\E}|_W)-2\deg D|_W=2\mu(W)-2\deg D|_W=\delta_W,
%\]
%and hence, using (1), we obtain that $t(e)\in W$, $s(e)\in W^c$, and  $\phi(e)=1$ for every $e\in E(W,W^c)$. 
By Lemma \ref{lem:beta-phi} we know that $\phi(e)=1$ for every $e\in \ora{E}(V,V^c)$ such that $t(e)\in V$.
Set $\Gamma_1=\Gamma(V)$ and $\Gamma_2=\Gamma(V^c)$. %Set $\E_1=\E\cap E(\Gamma_1)$ and $\E_2=\E\cap E(\Gamma_2)$. 
Define:
\begin{equation}\label{eq:D1D2}
D_1:=D|_{\Gamma_1}
\quad\text{ and }\quad
D_2:=D|_{\Gamma_2}-\sum_{v\in V^c}\deg_{E(V,V^c)}(v)v.
%\underset{s(e)\in W^c}{\sum_{e\in \ora{E}(W,W^c)}} s(e).
\end{equation}

We claim that $(\Gamma_i,\emptyset,D_i)$ is a semistable root-graph for $i=1,2$. Indeed,  consider the flow $\phi_i:=\phi|_{\Gamma_i}$ on $\Gamma_i$, for $i\in\{1,2\}$. First of all,
%$(\emptyset,D_i)$ is a pseudo-root on $\Gamma_i$, as the following computations show:
we have:
\begin{align*}
2 D_1 & = 2 D|_{\Gamma_1^{}}\\
& = K_{\Gamma}|_{\Gamma_1^{}}-\Div(\phi)|_{\Gamma_1^{}}\\
& = K_{\Gamma_1^{}}+\sum_{v\in V}\deg_{E(V,V^c)}(v)v-\Div(\phi_1)-\underset{t(e)\in V}{\sum_{e\in  \ora{E}(V,V^c)}} \phi(e)t(e)\\
& = K_{\Gamma_1^{}}-\Div(\phi_1)
%+\sum_{v\in W}\val_{E(W,W^c)}(v)v-\sum_{e\in  E(W,W^c)} \phi(e)
\end{align*}
and
\begin{align*}
2D_2 & = 2D|_{\Gamma_2^{}}-2\sum_{v\in V^c}\deg_{E(V,V^c)}(v)v\\
& =K_{\Gamma}|_{\Gamma_2^{}}-\Div(\phi)|_{\Gamma_2^{}}-2\sum_{v\in V^c}\deg_{E(V,V^c)}(v)v\\
& = K_{\Gamma_2^{}}-\sum_{v\in V^c}\deg_{E(V,V^c)}(v)v-\Div(\phi_2)-\underset{t(e)\in V^c}{\sum_{e\in  \ora{E}(V,V^c)}} \phi(e)t(e)\\
& = K_{\Gamma_2^{}}-\Div(\phi_2).
\end{align*} 

We now check the semistability of $(\Gamma_i,\emptyset,D_i)$.
Let $\mu_i$ be the canonical polarization of degree $g_{\Gamma_i}-1$ on $\Gamma_i$, which is the polarization taking a vertex $v\in V(\Gamma_i)$ to
\[
\mu_i(v)=\mu(v)-\frac{\deg_{E(V,V^c)}(v)}{2}.
\]

 For every subset $W\subset V(\Gamma_1^{})$, we have
 \begin{align*}
     \beta_{D_1}(W)& =\deg(D_1|_W)-\mu_1(W)+\frac{\delta_{\Gamma_1,W}}{2}\\
     & = \deg(D|_W)-\mu(W)+\frac{\#E(V^c,W)}{2}+\frac{\delta_{\Gamma,W}-\#E(V^c,W)}{2}\\
     & = \beta_D(W)\ge0,
 \end{align*}
 where the last inequality follows by the $\mu$-semistability of $D$ on $\Gamma^{}$.
 So the root-graph $(\Gamma_1,\emptyset,D_1)$ is
 $\mu_1$-semistable. For every subset $W\subset V(\Gamma_2^{})$, we have
\begin{align*}
    \beta_{D_2}(W)& = \deg(D_2|_W)-\mu_2(W)+\frac{\delta_{\Gamma_2,W}}{2}\\
 & =  \deg(D|_W)-\#E(V,W)-\mu(W)+\frac{\#E(V,W)}{2}+\frac{\delta_{\Gamma,W}-\#E(V,W)}{2}\\
 &=\beta_D(W)-\#E(V,W)\\ &=\beta_D(V\cup W)-\beta_D(V)+\#E(V,W)-\#E(V,W)\\
 &=\beta_D(V\cup W)\ge0,
\end{align*}
where we use \cite[Lemma 4.1]{AP1} and the $\mu$-semistability of $D$ on $\Gamma$. So 
 the root-graph $(\Gamma_2,\emptyset,D_2)$ is
 $\mu_2$-semistable. This concludes the proof of the claim.
 
 By the induction hypothesis, $\phi_1$ and $\phi_2$ are uniquely determined, so is $\phi$.
\end{proof}

\begin{Not}\label{not:unique-flow}
Given a semistable root-graph $(\Gamma,\E,D)$, we denote by $\phi_{(\E,D)}$ the unique acyclic flow on $\Gamma^\E$ such that $2D+\Div(\phi_{(\E,D)})=K_{\Gamma^\E}$. 
%Let $(\Gamma,\E,D)$ be a semistable root-graph.
In Figure \ref{fig:pos} we see the possibilities for the value $\phi_{(\E,D)}(e)$ for $e\in \ora{E}(\Gamma^\E)$.
\end{Not}
\begin{figure}[h]
    \centering
    \begin{tikzpicture}
    \draw (-1.5,0) -- (1.5,0);
     \draw (-5.5,0) -- (-2.5,0);
      \draw (2.5,0) -- (5.5,0);
    \draw[fill] (-1.5, 0) node {} (-1.5, 0) circle (0.1);
    \draw[fill] (1.5, 0) node {} (1.5, 0) circle (0.1);
    \draw[fill] (-5.5, 0) node {} (-5.5, 0) circle (0.1);
    \draw[fill] (5.5, 0) node {} (5.5, 0) circle (0.1);
     \draw[fill] (2.5, 0) node {} (2.5, 0) circle (0.1);
    \draw[fill] (-2.5, 0) node {} (-2.5, 0) circle (0.1);
     \draw[fill] (-4, 0) node {} (-4, 0) circle (0.1);
    \draw[->, line  width=0.3mm] (0,0) to (0.01,0);
        \draw[->, line  width=0.3mm] (-4.75,0) to (-4.749,0);
         \draw[->, line  width=0.3mm] (-3.25,0) to (-3.251,0);
         \node at (4,-0.2) {$0$};
          \node at (0,-0.3) {$1$};
           \node at (-4.75,-0.3) {$1$};
            \node at (-3.25,-0.3) {$1$};
            \node at (-4,0.4) {exceptional vertex};
         \draw[->, line  width=0.3mm] (-4,0.3) to (-4,0.15);
    \end{tikzpicture}
\caption{Possible values for the flow $\phi_{(\E,D)}$.}
\label{fig:pos}
\end{figure}
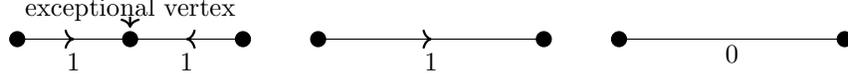

\begin{Rem}
Since $\phi(e')\ne 0$ for every $e'$ lying over an edge $e\in \ora{\E}$, the graph $P_{(\E,D)}$ defined in equation \eqref{eq:PED} can be seen as a subgraph of either one of the following graphs: $\Gamma_\E$, $\Gamma$, $\Gamma^\E$. 
\end{Rem}

%In each situation, it will be clear where $P_{(\E,D)}$ is considered.

We denote by $\ol {P}_{(\E,D)}=\Gamma^\E_{\mc F,\leg}$, where $\mc F=E(\Gamma^\E)\setminus P_{(\E,D)}$.
The set of connected components of 
$\ol {P}_{(\E,D)}$ 
are in natural bijection with the set  $V(\Gamma^\E/P_{(\E,D)})$.

\begin{Exa}\label{exa:Pbar}
Let $\Gamma$ be the graph with two vertices $v_0$ and $v_1$ of weight zero, connected by two edges $e_0,e_1$. Let $\E=\{e_0\}$ and $D=v_0-w_0\in \Div(\Gamma^\E)$, where $w_0$ is the exceptional vertex of $e_0$. Then $(\E,D)$ is a $v_0$-quasistable pseudo-root (hence semistable), and $\phi_{(\E,D)}$ has value $1$ over every $e\in \ora{E}(\Gamma^{\E})$ oriented as in Figure \ref{fig:notspec}. In this case, $P_{(\E,D)}=0$, and $\ol{P}_{(\E,D)}$ consists of the disjoint union of the vertices $v_0,v_1,w_0$, with two legs attached to each vertex and no edge.
  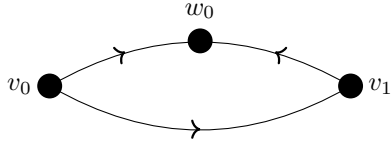
\begin{figure}[h]
    \centering
    \begin{tikzpicture}
    \draw (0,0) to [bend right] (4,0);
    \draw (0,0) to [bend left] (4,0);
    \draw[->, line  width=0.3mm] (0.9975,0.449) to (1.0025,0.4514);
    \draw[->, line  width=0.3mm] (3.0025,0.449) to (2.9975,0.4514);
     \draw[->, line  width=0.3mm] (1.9975,-0.575) to (2.0025,-0.575);
    \node at (0,0) [circle, fill]{};
    \node at (-0.4,0) {$v_0$};
    \node at (4.4,0) {$v_1$};
    \node at (4,0) [circle, fill]{};
    \node at (2,0.6) [circle, fill]{};
    \node at (2,1) {$w_0$};
    \end{tikzpicture}
\caption{The orientation on the graph $\Gamma^\E$.}
\label{fig:notspec}
\end{figure}
\end{Exa}

\begin{Def}\label{def:phi-subgraph}
Given a semistable pseudo-root $(\Gamma,\E,D)$, the  \emph{$\phi_{(\E,D)}$-graphs} of $\Gamma^\E$ are the connected components $P_u$ of $\ol {P}_{(\E,D)}$, for $u\in V(\Gamma^\E/P_{(\E,D)})$.
\end{Def}

\begin{Rem}\label{rem:PED}
Let $(\Gamma,\E,D)$ be a semistable root-graph.
Consider a vertex $u\in V(\Gamma^\E/P_{(\E,D)})$. Using that $2D+\Div(\phi_{(\E,D)})=K_{\Gamma^\E}$, we deduce that
\[
2D|_{P_u}+\sum_{w\ne u}\underset{t(e)\in V(P_u)}{\sum_{e\in \ora{E}(P_u,P_w)}}(\phi_{(\E,D)}(e)-1)t(e)=K_{P_v}.
\]
\end{Rem}

%\begin{Cor}\label{cor:source}
%Let $(\Gamma,\E,D)$ be a $v_0$-quasistable root-graph. Then $\phi_{(\E,D)}$ has exactly one source $u_0\in V(\Gamma^\E/P_{(\E,D)})$ and $v_0\in \iota^{-1}(u_0)$, where $\iota$ is the specialization $\iota\col \Gamma^\E\ra \Gamma^\E/P_{(\E,D)}$ in Remark \ref{rem:iota-pseudo-root}. In particular, $u_0$ is characterized as the unique vertex of $\Gamma^\E/P_{(\E,D)}$ such that 
%\[
%D|_{P_{u_0}}=\frac{K_{P_{u_0}}}{2}+\sum_{w\ne u_0}\;\;\underset{t(e)\in P_{u_0}}{\sum_{e\in E(P_{u_0},P_w)}} t(e) 
%\]
%\end{Cor}

%\begin{proof}
%Since $\phi$ is acyclic, it admits a source $u_0\in V(\Gamma^\E/P{_{(\E,D)}})$. Consider $V=\iota^{-1}(u_0)$. We need only to prove that $v_0\in V$. Since $u_0$ is a source, we have $\phi_{(\E,D)}(e)=1$ for every $e\in \ora{E}(V,V^c)$ such that $s(e)\in V$. By Lemma \ref{lem:beta-phi} we have that $\beta_D(V^c)=0$, and hence $v_0\in V$ by the condition of quasistability.
%\end{proof}

%and let $w_0\in V(\Gamma^\E/\E_{\phi_{(\E,D)}})$ be the source of $\phi_{(\E,D)}$ (see Corollary \ref{cor:source}). 

%For every $w\in V(\Gamma^\E/\E_{\phi_{(\E,D)}})$ we $d(w,w_0)$ be the distance from $w$ to $w_0$,  

%\[
%W_i=\{w\in V(\Gamma^\E/\E_{\phi_{(\E,D)}}) | d(w,W_0)=i\}
%\]

\begin{Def}
We say that a semistable root-graph  $(\Gamma,\E,D)$ \emph{specializes} to a semistable root-graph $(\Gamma',\E',D')$, and we write $(\Gamma,\E,D)\ge (\Gamma',\E',D')$, if there is a specialization of triples $(\Gamma,\E,D)\ge (\Gamma',\E',D')$ with associated specialization $\wh{\iota}\col \Gamma^{\E}\ra {\Gamma'}^{\E'}$ such that 
%such that  $\iota\col\Gamma_1^{\E_1}\ra \Gamma_2^{\E_2}$ such that $\iota_*(D_1)=D_2$ and 
$\wh{\iota}_*(\phi_{(\E,D)})=\phi_{(\E',D')}$. 
\end{Def}

%\begin{Def}
%Let $\Gamma$ be a graph. A \emph{polystable pseudo-root} of $\Gamma$ is a polystable pseudo-divisor $(\E,D)$ such that $2D$ is equivalent to $K_{\Gamma^\E}$.
%\end{Def}

We consider the category $\mathbf{SSR}_{g,n}$ whose objects are semistable roots-graphs $(\Gamma,\E,D)$, where $\Gamma$ is a stable genus-$g$ graph with $n$ legs, and the arrows are given by the specializations of semistable roots-graphs. 
We let $\mathcal{SSR}_{g,n}$ be the poset $\mathbf{SSR}_{g,n}/\sim$, where the relation $\sim$ is isomorphism. 

%Notice that $\mathbf{SSR}_{g,n}$ and $\mathcal{SSR}_{g,n}$ are, respectively, a sub-category of $\mathbf{SSD}_{\mu,g,n}$ and a sub-poset of $\mathcal{SSD}_{\mu,g,n}$.

\begin{Rem}\label{rem:proper}
%Notice that $\PR_{g,n}$ is a sub-category of $\mathbf{QD}_{\mu,g,n}$. It is a proper sub-category. Indeed, 
 Not every semistable pseudo divisor $(\E,D)$ on a graph $\Gamma$ is such that $(\Gamma,\E,D)$ is an object of $\mathbf{SSR}_{g,n}$. 
 %For example, $v_0$-quasistable pseudo-divisor $(\E,D)$ satisfies Condition (2) of Definition \ref{def:v0pseudoroot}. 
 Let $\Gamma$ be the graph with two vertices $v_0$ and $v_1$ of weight zero, and four edges connecting them. Then $(\emptyset,2v_0)$ is a $v_0$-quasistable pseudo-divisor (and hence it is semistable), but there is no acyclic flow $\phi$ on $\Gamma$ such that $2D+\Div(\phi)=K_\Gamma$.  
\end{Rem}

\begin{Rem}
It is not true that every specialization $(\Gamma,\E,D)\ge (\Gamma',\E',D')$ of triples is a morphism in $\mathbf{SSR}_{g,n}$, for objects $(\Gamma,\E,D)$ and $(\Gamma',\E',D')$  in $\mathbf{SSR}_{g,n}$. Let $\Gamma$ and $(\E,D)$ be the graph and the pseudo-divisor on it as in Example \ref{exa:Pbar}.
The pseudo-divisor $(\E',D')$ on $\Gamma$, where $\E'=\emptyset$ and $D'=0$, is a $v_0$-quasistable pseudo-root, with $\phi_{(\E',D')}=0$. We have a specialization $\iota\col (\Gamma,\E,D)\ge (\Gamma,\E',D')$ but, if $\wh \iota\col \Gamma^{\E}\ra {\Gamma'}^{\E'}$ is the associated specialization, then $\wh{\iota}_*(\phi)\ne 0$: this specialization is not a morphism in $\mathbf{SSR}_{g,n}$.
\end{Rem}
  
We denote by $\mathbf{SR}_{g,n}$, $\mathbf{PSR}_{g,n}$ and $\mathbf{QR}_{g,n}$, respectively, the subcategory of $\mathbf{SSR}_{g,n}$ of simple semistable, polystable and $v_0$-quasistable root-graphs. We let $\mathcal{SR}_{g,n}$, $\mathcal{PSR}_{g,n}$ and $\mathcal{QR}_{g,n}$ be respectively the associated posets. Notice that $\mathbf{PSR}_{g,n}$ is a subcategory of $\mathbf{PSD}_{g,n}$, while $\mathbf{QR}_{g,n}$ is a sub-category of both $\mathbf{QD}_{g,n}$ and of $\mathbf{SR}_{g,n}$ (by \cite[Proposition 4.6]{AP1}). 

%while $\mathcal{PSR}_{g,n}$ and $\mathcal{QR}_{g,n}$ are sub-posets of $\mathcal{PSD}_{\mu,g,n}$ and $\mathcal{QD}_{\mu,g,n}$.

%\begin{Prop}
%Let $(\E,D)$ be a simple semistable pseudo-root on a graph $\Gamma$. Then $(\E,D)$ is $v_0$-quasistable for some vertex $v_0$ of $\Gamma$. 
%\end{Prop}

%\begin{proof}
%\end{proof}

\subsection{Polystable root-graphs}

In this section we analyze more closely polystable root-graphs. We show that the category of polystable pseudo-roots  can be identified with the one of cyclic subgraphs.

\begin{Lem}\label{lem:phizero}
Let $(\E,D)$ be a polystable pseudo-root on a graph $\Gamma$. We have $\phi_{(\E,D)}(e)=0$ if and only if $e\in \ora{E}(\Gamma^\E)$ is a non-exceptional oriented edge. In particular, $\Div(\phi_{(\E,D)})$ is a principal divisor on $\Gamma^\E$, with  $\Div(\phi_{(\E,D)})=-\sum_{e\in \E} \Div(w_e)$, where $w_e$ is the exceptional vertex in the interior of $e$.
\end{Lem}

\begin{proof}
Since $(\E,D)$ is polystable, we have that the divisor $D_\E$ is $\mu_\E$-stable on $\Gamma_\E$. Thus $D_1=D_\E-\sum_{v\in V(\Gamma^\E)}\text{deg}_\E(v)$ is $\mu_1$-stable, where %$\mu_1(v)=\mu(v)-\text{val}_\E(v)$ for every $v\in V(\Gamma^\E)$ 
$\mu_1$ is the canonical divisor of degree $g_{\Gamma^\E}-1$ of $\Gamma^\E$. Set $\phi=\phi_{(\E,D)}$. By Theorem \ref{thm:uniquephi} (1), we have  $2D_1+\Div(\phi|_{\Gamma_\E})=K_{\Gamma_\E}$, then by Lemma \ref{lem:stable} we have that $\phi|_{\Gamma_\E}=0$, that is, $\phi(e)=0$ for every non-exceptional oriented edge $e\in \ora{E}(\Gamma^\E)$. We can conclude using Theorem \ref{thm:uniquephi} (1).
\end{proof}

%Let $\Gamma_1,\Gamma_2$ be graphs with $n$ legs and $(\E_i,D_i)$ a polystable pseudo-root on $\Gamma_i$ for $i=1,2$. A \emph{specialization} of terns $(\Gamma_1,\E_1,D_1)\ra (\Gamma_2,\E_2,D_2)$ is the datum of a specialization of graphs $\iota\col \Gamma_1\ra \Gamma_2$ of graphs with $n$ legs such that $(\E_2,D_2)=\iota_*(\E_1,D_1)$.

%We consider the category $\PR_{g,n}$ whose objects are isomorphisms classes of terns $(\Gamma,\E,D)$,  where $\Gamma$ is a stable graph with $n$ legs of genus $g$ and $(\E,D)$ is a polystable square root of $\Gamma$, and where the  arrows $(\Gamma_1,\E_1,D_1)\ra (\Gamma_2,\E_2,D_2)$ are the specializations of terns. 
%We let $\mathcal{PR}_{g,n}$ be the poset $\mathbf{PR}_{g,n}/\sim$, where the relation $\sim$ is isomorphism.
%Notice that $\mathbf{PR}_{g,n}$ and $\mathcal{PR}_{g,n}$ are, respectively, a sub-category of $\mathbf{PSD}_{\mu,g,n}$ and a sub-poset of $\mathcal{PSD}_{\mu,g,n}$.

As in \cite[Section 4]{CMP2},
 we let $\CYC_{g,n}$ be the category  whose objects are pairs $(\Gamma,P)$, where $\Gamma$ is a stable genus-$g$ graph with $n$ legs and $P$ is a cyclic subgraph of $\Gamma$, and whose arrows  $(\Gamma,P)\ra (\Gamma',P')$ are given by a specialization $\iota\col \Gamma\ra \Gamma'$ of graphs with $n$ legs such that $\iota_*(P)=P'$.

For every graph $\Gamma$ and a cyclic subgraph $P$ of $\Gamma$, we denote by $(\E_P,D_P)$ the pseudo-divisor on $\Gamma$ such that  $\E_P=E(\Gamma)\setminus P$ and 
\begin{equation}\label{eq:DP}
D_P=\sum_{v\in V(\Gamma)} \left(\deg_\Gamma(v)-\frac{\deg_P(v)}{2}-1+w_\Gamma(v)\right)v-\sum_{e\in \E_P}w_e
\end{equation}
where $w_e$ is the exceptional vertex inside $e$.

\begin{Prop}\label{prop:EP}
Let $\Gamma$ be a graph. The pseudo-divisors $(\E_P,D_P)$, for $P$ a cyclic subgraph of $\Gamma$, are all the polystable pseudo-roots on $\Gamma$. 
In particular, $\Gamma$ admits exactly $2^{b_1(\Gamma)}$ polystable pseudo-roots.
\end{Prop}

\begin{proof}
%By Lemma \ref{lem:phizero} we know that, for every cyclic subgraph $P$ of $\Gamma$, we have
Let $\phi$ be the acyclic flow associated to the 
principal divisor $-\sum_{e\in \E_P} \Div(w_e)$. We have:
\[
2D_P+\Div(\phi)=\sum_{v\in V(\Gamma)} \left(2\deg_\Gamma(v)-\deg_P(v)-2+2w_\Gamma(v)- \deg_{\E_P}(v)\right)v=K_{\Gamma^\E}.
\]
Moreover, since $\mu$ is the canonical polarization, we have: 
\[
\mu_{\E_P}(v)=\mu(v)+\frac{\deg_{\E_P}(v)}{2}=(D_P)_{\E_P}(v)
\]
for every $v\in V(\Gamma)$. Thus the divisor $(D_P)_{\E_P}$ is $\mu_{\E_P}$-stable on $\Gamma_{\E_P}$, so $(\E_P,D_P)$ is polystable. 
We deduce that $(\E_P,D_P)$ is a polystable pseudo-root of $\Gamma$.

Conversely, let $(\E,D)$ be a polystable pseudo-divisor on $\Gamma$. By Lemma \ref{lem:phizero}:
\begin{eqnarray*}
2D 
&= & K_{\Gamma^\E}+\sum_{e\in \E}\Div(w_e)\\
 &= & \sum_{v\in V(\Gamma)} (\deg_\Gamma(v)-2+2w_\Gamma(v)+ \deg_\E(v))v-2\sum_{e\in \E} w_e\\
  &= & \sum_{v\in V(\Gamma)} (2\deg_\Gamma(v)-\deg_{E(\Gamma)\setminus \E}(v)-2+2w_\Gamma(v))v-2\sum_{e\in \E} w_e.
\end{eqnarray*}
If we set $P:=E(\Gamma)\setminus \E$, we see that $\deg_P(v)$ is even for every $v\in V(\Gamma)$, that is $P$ is a cyclic subgraph of $\Gamma$, and $(\E,D)=(\E_P,D_P)$, as wanted.
\end{proof}

\begin{Rem}\label{rem:cyclic}
Let $(\E,D)$ be a polystable pseudo-root on a graph $\Gamma$. If $(\E,D)=(\E_P,D_P)$ for some cyclic subgraph $P$ of $\Gamma$, then  $P=P_{(\E,D)}$ by Lemma \ref{lem:phizero}. 
\end{Rem}

\begin{Prop}\label{prop:iota*}
Let $\iota\col\Gamma\ra\Gamma'$ be a specialization of graphs. Given a cyclic subgraph $P$ of $\Gamma$, we have  $\iota_*(\E_P,D_P)=(\E_{\iota_*(P)},D_{\iota_*(P)})$ and $\iota_*(\phi_{(\E_P,D_P)})=\phi_{(\E_{\iota_*(P)},D_{\iota_*(P)})}$.
\end{Prop}

\begin{proof}
First of all, notice that  $\iota_*(\E_P)=\E_{\iota_*(P)}$. We now show that $\iota_*(D_P)=D_{\iota_*(P)}$. We can assume that $\iota$ is the contraction of just one edge $e\in E(\Gamma)$. Let $v_1$ and $v_2$ be the (possibly coinciding) vertices of $\Gamma$ incident to $e$ and let $v_0$ be the vertex of $\Gamma'$ to which $e$ contracts. Moreover, we have that $\iota_*(D_P)(v)=D_{\iota_*P}(v)$ for every $v\in V({\Gamma'}^{\E_{\iota_*(P)}})$ with $v\ne v_0$. If $e$ is not a loop, then
\begin{eqnarray*}
\iota_*(D_P)(v_0)&=& \deg_\Gamma(v_1)+\deg_{\Gamma}(v_2)-\frac{\deg_P(v_1)}{2}-\frac{\deg_P(v_2)}{2}-2+\\
&&+w_\Gamma(v_1)+w_\Gamma(v_2)-[e\in \E]\\
& = &  \deg_{\Gamma'}(v_0)-\left(\frac{\deg_P(v_1)}{2}+\frac{\deg_P(v_2)}{2}+[e\in \E]\right)+w_{\Gamma'}(v_0) \\
&=& \deg_{\Gamma'}(v_0)-\left(\frac{\deg_{\iota_*(P)}(v_0)}{2}+1\right)+w_{\Gamma'}(v_0)\\
& = & D_{\iota_*(P)}(v_0).
\end{eqnarray*}
 On the other hand, if $e$ is a loop, we set $\ol v:=v_1=v_2$. Then
\begin{eqnarray*}
\iota_*(D_P)(v_0)&=& \deg_\Gamma(\ol v)-\frac{\deg_P(\ol v)}{2}-1+w_{\Gamma}(\ol v)-[e\in \E]\\
& = &  \deg_{\Gamma'}(v_0)-\left(\frac{\deg_P(\ol v)}{2}+[e\in \E]\right)+\left(1+w_{\Gamma}(\ol v)\right) \\
&=& \deg_{\Gamma'}(v_0)-\left(\frac{\deg_{\iota_*(P)}(v_0)}{2}+1\right)+w_{\Gamma'}(v_0) \\
& = & D_{\iota_*(P)}(v_0).
\end{eqnarray*}
Summing up, we get  $\iota_*(\E_P,D_P)=(\E_{\iota_*(P)},D_{\iota_*(P)})$. 
By Lemma \ref{lem:phizero}, we also have $\iota_*(\phi_{(\E_P,D_P)})=\phi_{(\E_{\iota_*(P)},D_{\iota_*(P)})}$, and we are done.
\end{proof}

\begin{Cor}\label{cor:graph-spec}
Let $\iota\col\Gamma\ra \Gamma'$ be a specialization of graphs. 
If $(\Gamma,\E,D)$ is a polystable root-graph,  
then $(\Gamma',\iota_*(\E,D))$ is a polystable root-graph, and we have a specialization of polystable root-graphs
$(\Gamma,\E,D)\ge (\Gamma',\iota_*(\E,D))$.
\end{Cor}

\begin{proof}
The result readily follows from Propositions \ref{prop:EP} and \ref{prop:iota*}.
\end{proof}

\begin{Prop}\label{prop:polyspec}
Let $(\Gamma,\E,D)\ge (\Gamma',\E',D')$ be a specialization of polystable root-graphs with associated specializations $\iota\col \Gamma\ra \Gamma'$ and $\wh\iota\col\Gamma^{\E}\ra{\Gamma'}^{\E'}$. Then $\E'=\E\cap E(\Gamma')$, so in particular $\wh\iota$ is the specialization induced by $\iota$. 
\end{Prop}

\begin{proof}
Let $e_1,e_2$ be different edges of $\Gamma^{\E}$ incident to the same exceptional vertex of $\Gamma^{\E}$. If $\wh{\iota}$ contracts $e_1$, then $\wh{\iota}$ contracts $e_2$ as well, otherwise $e_2$ would be a non-exceptional edge of ${\Gamma'}^{\E'}$ and hence $\phi_{(\E',D')}(e_2)=0$ by Lemma \ref{lem:phizero}. But $\phi_{(\E',D')}(e_2)=\wh\iota_*(\phi_{(\E,D)})(e_2)=1$ by Theorem \ref{thm:uniquephi}, a contradiction.
\end{proof}

\begin{Cor}\label{cor:isocat}
The categories $\mathbf{PSR}_{g,n}$ and $\CYC_{g,n}$ are canonically isomorphic.
\end{Cor}

\begin{proof}
This is a consequence of Propositions \ref{prop:EP}, \ref{prop:iota*}, \ref{prop:polyspec}. 
\end{proof}

\begin{Lem}
\label{lem:semispec}
Let $(\E,D)$ be a semistable pseudo-root on a graph $\Gamma$. Assume that $V\subset V(\Gamma)$ is a subset with $\beta_{\E,D}(V)=0$ and $E(V,V^c)\not\subset \E$. Then there is a unique semistable pseudo-root $(\E\cup E(V,V^c),D_1)$ on $\Gamma$ admitting a specialization $(\Gamma,\E\cup E(V,V^c),D_1)\ge(\Gamma,\E, D)$ of semistable root-graphs. Moreover, we have $P_{(\E\cup E(V,V^c),D_1)}=P_{(\E,D)}$.
\end{Lem}
\begin{proof}
By \cite[Proposition 5.3]{AAPT} the pseudo-divisor $(\E\cup E(V,V^c),D_1)$ on $\Gamma$, where
\begin{equation*}\label{eq:D1}
D_1(v)=\begin{cases}
D(v)+\text{deg}_{E(V,V^c)\setminus\E}(v)&\text{ if $v \in V$,}\\
D(v)&\text{ if $v \notin V$,}\\
-1 &\text{ if $v$ is exceptional,}
\end{cases}
\end{equation*}
is the unique semistable pseudo-divisor on $\Gamma$ admitting a specialization of pseudo-divisors $(\E\cup E(V,V^c),D_1)\ra(\E,D)$. The associated specialization $\wh\iota\col \Gamma^{\E\cup E(V,V^c)}\ra \Gamma^{\E}$ gives rise to a natural inclusion $E(\Gamma^\E)\subset E(\Gamma^{\E\cup E(V,V^c)})$. We define the flow $\phi_1$ on $\Gamma^{\E\cup E(V,V^c)}$ such that $(\phi_1)|_{\ora{E}(\Gamma^\E)}=\phi_{(\E,D)}$ and $\phi_1(e)=1$ for every  $e\in \ora{E}(\Gamma^{\E\cup E(V,V^c)})\setminus \ora{E}(\Gamma^\E)$ such that $t(e)$ is an exceptional vertex. 
We define
\[
\widetilde{V}=V \cup\{w_e | e\in \E\cap(E(V)\cup E(V,V^c))\},
\] 
where $w_e$ is the exceptional vertex inside $e\in \E$.
 It is easy to check that $\beta_D(\widetilde{V})=\beta_{\E,D}(V)=0$.
So, by Lemma \ref{lem:beta-phi},
we have that $\phi_{(\E,D)}(e)=1$ for every $e\in \ora{E}(V,V^c)\setminus \ora{\E}\subset \ora{E}(\widetilde{V},\widetilde{V}^c)$ such that $t(e)\in V$. We deduce that  $\wh\iota_*(\phi_1)=\phi_{(\E,D)}$.

Notice that $\phi_1$ is an acyclic flow, since $\phi_{(\E,D)}$ is acyclic and $\wh\iota^S_*(P)$ is a non-null $\phi_{(\E,D)}$-cycle on $\Gamma^\E$ for every non-null $\phi_1$-cycle $P$ on $\Gamma^{\E\cup E(V,V^c)}$.

Next, we already observed that
$\phi_{(\E,D)}(e)=1$ for every $e\in \ora{E}(V,V^c)\setminus\ora{\E}$ such that $t(e)\in V$. Then for every $v\in V(\Gamma^{\E\cup E(V,V^c)})$ we have:
\begin{equation*}
\Div(\phi_1)(v)=\begin{cases}
\Div(\phi)(v)-2\text{deg}_{E(V,V^c)\setminus \E}(v)&\text{ if $v \in V$,}\\
\Div(\phi)(v)&\text{ if $v \notin V$,}\\
2 &\text{ if $v$ is exceptional}
\end{cases}
\end{equation*}
and hence
\begin{equation*}
2D_1(v)+\Div(\phi_1)(v)=\begin{cases}
2D(v)+\Div(\phi)(v)&\text{ if $v \in V(\Gamma)$,}\\
0 &\text{ if $v$ is exceptional.}
\end{cases}
\end{equation*}
Since $2D+\Div(\phi)=K_{\Gamma^\E}$, it follows that
 $2D_1+\Div(\phi_1)=K_{\Gamma^{\E\cup E(V,V^c)}}$, and hence $(\E\cup E(V,V^c),D_1)$ is a semistable pseudo-root on $\Gamma$.  By construction, we have a specialization of semistable root-graphs $(\Gamma,\E\cup E(V,V^c),D_1)\ge(\Gamma,\E,D)$ and $P_{(\E\cup E(V,V^c),D_1)}=P_{(\E,D)}$.
\end{proof}

\begin{Prop}\label{prop:minimal-root}
 Let $(\Gamma,\E,D)$ be a semistable root-graph. Then:
 \begin{itemize}
     \item[(1)] $(\Gamma,\pol(\E,D))$ is a polystable root-graph; it is the unique minimal polystable root-graph  admitting a specialization of semistable root-graphs to $(\Gamma,\E,D)$;
     \item[(2)] $P_{(\E,D)}=P_{\pol(\E,D)}$.
 \end{itemize}
\end{Prop}

\begin{proof}
A semistable pseudo-root $(\E,D)$ on $\Gamma$ is polystable if and only if  
%By deﬁnition, if (E,D) is a µ-semistable pseudo-divisor on Γ such that 
$\beta_{\E,D}(V)>0$ 
for every subset $V\subset V(\Gamma)$ such that $E(V,V^c)\not\subset \E$.
%then (E,D) is µ-polystable.
Therefore, by Lemma \ref{lem:semispec}, we can construct a sequence of specializations of semistable root-graphs
\[
(\Gamma,\E_k,D_k)>(\Gamma,\E_{k-1},D_{k-1})>\dots >(\Gamma,\E_0,D_0)=(\Gamma,\E,D),
\]
 where $(\E_k,D_k)$ is polystable and $(\E_j,D_j)$ is not polystable for $j<k$. The sequence of specializations above is the same as in \cite[Equation 21]{AAPT}. By \cite[Proposition 5.4]{AAPT} we have that $(\E_k,D_k)=\pol(\E,D)$, which proves (1). By Lemma \ref{lem:semispec} we also have $P_{(\E_i,D_i)}=P_{(\E,D)}$ for every $i\in\{1,\dots,k\}$, which proves (2).    
\end{proof}

\begin{Cor}\label{cor:tree}
If $(\Gamma,\E,D)$ is a semistable root-graph, then $P_{(\E,D)}$ is a cyclic subgraph of $\Gamma$. In particular no connected component of $\ol{P}_{(\E,D)}$ is a tree with at least one edge. 
\end{Cor}

\begin{proof}
The result follows from Remark \ref{rem:cyclic} and  Proposition \ref{prop:minimal-root}.
\end{proof}

\subsection{Spin pseudo-roots}

%Let $(\Gamma,\E,D)$ be a semistable %pseudo-root. Consider the subgraph of \Gamma$ given by:
%\[
%P_{(\E,D)}=\{e\in E(\Gamma^\E)| \phi_{(\E,D)}(e)=0\}.
%\]

In this section we consider an additional structure to  root-graphs. This is a sign function which, as we shall see later, is the combinatorial counterpart for the parity of the number of sections of a sheaf on curves.

\begin{Def}
Let $\Gamma$ be a graph. A \emph{spin semistable  pseudo-root} on $\Gamma$ is a triple $(\E,D,s)$, where $(\E,D)$ is a semistable pseudo-root on $\Gamma$ and $s\col V(\Gamma/P_{(\E,D)})\ra \mathbb Z/2\mathbb Z$ is a function, called \emph{sign function on $(\Gamma,\E,D)$}, such that $s(v)=0$ for every vertex $v\in V(\Gamma/P_{(\E,D)})$ of weight $w_{\Gamma/P_{(\E,D)}}(v)=0$. We say that $(\E,D,s)$ is \emph{polystable} (respectively, \emph{$v_0$-quasistable for some $v_0\in V(\Gamma)$}) if so is $(\E,D)$. A  \emph{spin semistable} (\emph{polystable, $v_0$-quasistable}) \emph{root-graph} is a tuple $(\Gamma,\E,D,s)$, where $\Gamma$ is a graph and $(\E,D,s)$ is a spin semistable (polystable, $v_0$-quasistable) pseudo-root on $\Gamma$.
%A \emph{signed semistable (stable, polystable, $v_0$-quasistable for a vertex $v_0$ of $\Gamma$) pseudo-root} is a tuple $(\Gamma,\E,D,s)$ where $(\Gamma,\E,D)$ is a semistable (stable, polystable, $v_0$-quasistable for a vertex $v_0$ of $\Gamma$) pseudo-root and  $s\col V(\Gamma/P_{(\E,D)})\ra \mathbb Z/2\mathbb Z$ is a function, called \emph{sign function}, such that $s(v)=0$ for every vertex $v\in V(\Gamma/P_{(\E,D)})$ with $w_{\Gamma/P_{(\E,D)}}(v)=0$. 
We say that $(\Gamma,\E,D,s)$ is \emph{even} or \emph{odd}, respectively, if so is $\sum_{v\in V(\Gamma/P_{(\E,D)})} s(v)$.
\end{Def}

\begin{Lem}\label{lem:iota-s}
If $\iota\col (\Gamma,\E,D)\ra (\Gamma',\E',D')$ is a specialization of semistable root-graphs,
%with associated specialization $\wh\iota\col \Gamma^{\E}\ra {\Gamma'}^{\E'}$, 
then $P_{(\E',D')}=\iota_*(P_{(\E,D)})$. 
In particular, we have an induced specialization $\wt\iota\col \Gamma/P_{(\E,D)}\ra \Gamma'/P_{(\E',D')}$ giving rise to a commutative diagram
\[
\SelectTips{cm}{11}
\begin{xy} <16pt,0pt>:
\xymatrix{
  \Gamma  \ar[d] \ar[rr]^{\iota}    & & \Gamma' \ar[d]  \\
\Gamma/P_{(\E,D)}\ar[rr]^{\wt\iota}&& \Gamma'/P_{(\E',D')}
 }
\end{xy}
\]
\end{Lem}

\begin{proof}
The proof is left to the reader.
%Since $\phi'=\wh\iota_*(\phi)$, it is clear that $P_{(\E',D')}=\wh\iota_*(P_{(\E,D)})$. The last part easily follows from this.
\end{proof}

\begin{Def}
A spin semistable root-graph  $(\Gamma,\E,D,s)$ \emph{specializes} to a  spin semistable root-graph  $(\Gamma',\E',D',s')$, and we write  $(\Gamma,\E,D,s)\ge(\Gamma',\E',D',s')$, if there is a specialization of semistable root-graphs $(\Gamma,\E,D)\ge(\Gamma',\E',D')$ such that $s'=\wt\iota_*(s)$, where $\wt\iota\col \Gamma/P_{(\E,D)}\ra {\Gamma'}/P_{(\E',D')}$ is the induced specialization (see Lemma \ref{lem:iota-s}) and $\wt\iota_*(s)$ is the sign function on $(\Gamma',\E',D')$ taking $u'\in V(\Gamma'/P_{(\E',D')})$ to 
\[
\wt\iota_*(s)(u')=\underset{u\in \wt\iota^{-1}(u')}{\sum_{u\in  V(\Gamma/P_{(\E,D)})}} s(u).
\]
\end{Def}

We consider the category $\mathbf{SS}_{g,n}$ whose objects are spin semistable root-graphs, where the arrows are given by the specializations of spin semistable root-graphs. %We let $\mathcal{SS}_{g,n}$ be the poset $\mathbf{SS}_{g,n}/\sim$, where the relation $\sim$ is isomorphism.

We consider the sub-categories $\mathbf{SP}_{g,n}$ and $\mathbf{SPQ}_{g,n}$ of $\mathbf{SS}_{g,n}$ whose objects are, respectively, spin polystable and spin $v_0$-quasistable root-graphs. 
We let $\mathcal{SP}_{g,n}$ and $\mathcal{SPQ}_{g,n}$ be, respectively, the sub-posets
%of $\mathcal{SS}_{g,n}$ given by
$\mathbf{SP}_{g,n}/\sim$ and $\mathbf{SPQ}_{g,n}/\sim$.

\begin{Rem}\label{rem:parity}
The parity of a spin semistable  pseudo-root is preserved under specializations. Moreover, every even (respectively, odd) spin root-graph admits a specialization to $(\Gamma,\emptyset,D,s)$, where $\Gamma$ is the graph with one vertex $v$ and no edges, $D=(g-1)v$ and $s(v)=0$ (respectively, $s(v)=1$). Hence the poset $\mathcal{SPQ}_{g,n}$ has exactly two connected components $\mathcal{SPQ}^+_{g,n}$ and $\mathcal{SPQ}^-_{g,n}$ corresponding to even and odd spin root-grahs.
\end{Rem}

\begin{Rem}\label{rem:spin-graphs-cat}
Using Corollary \ref{cor:isocat}, it is easy to check that the category of spin graphs of genus $g$ and $n$ legs, as defined in \cite{CMP1}, is canonically isomorphic to the category $\mathbf{SP}_{g,n}$. The isomorphism sends a spin graph $(\Gamma,P,s)$, where $\Gamma$ is a graph, $P$ is a cyclic subgraph of $\Gamma$ and $s\col V(\Gamma/P)\ra \mathbb Z/2\mathbb Z$ is a sign function,  to $(\Gamma,D_P,\E_P,s)$.
\end{Rem}

\begin{Rem}\label{rem:D-split}
Let $(\Gamma,\emptyset,D,s)$ be a spin semistable root graph such that $\beta_D(V)=0$ for a proper subset $V\subset V(\Gamma)$.  Set $\Gamma_1=\Gamma(V)$,  $\Gamma_2=\Gamma(V^c)$.
%and $\E_i=\E\cap E(\Gamma_i)$ for $i=1,2$. 
Define $D_1^V:=D_1$ and $D_2^V:=D_2$ as in equation \eqref{eq:D1D2}. By the claim in the proof of Theorem \ref{thm:uniquephi} we have that $(\Gamma_i,\emptyset,D_i^V)$ is a semistable root-graph for $i=1,2$. By Lemma \ref{lem:beta-phi} we have $\phi_{(\E,D)}(e)=1$, $\forall\;e\in \ora{E}(\Gamma_1,\Gamma_2)$ such that $t(e)\in V$. Then $P_{(\emptyset,D_i^V)}=P_{(\emptyset,D)}\cap \Gamma_i$, so there is an inclusion $\Gamma_i/P_{(\emptyset,D_i^V)}\subset \Gamma/P_{(\emptyset,D)}$. 
We define the spin semistable root-graph   $(\Gamma_i,\emptyset,D_i^V,s_i^V)$, where $s_i^V$ be the restriction of $s$ to $V(\Gamma_i/P_{(\emptyset,D_i^V)})$.

Assume that $\iota\col(\Gamma',\emptyset,D',s')\ra (\Gamma,\emptyset,D,s)$ is a specialization of spin semistable root-graphs and set $V':=\iota^{-1}(V)\subset V(\Gamma')$. By Proposition \ref{prop:stab-contract} we have  $\beta_{D'}(V')=\beta_D(V)=0$. Set  $\Gamma'_1=\Gamma(V')$ and $\Gamma'_2=\Gamma({V'}^c)$. 
By construction, there is a specialization
$(\Gamma'_i,\emptyset,D_i^{V'},s_i^{V'})\ge (\Gamma_i,\emptyset,D_i^V,s_i^V)$ of spin semistable root-graphs for $i=1,2$.
\end{Rem}

\begin{Def}\label{def:spin-pol}
Given a spin semistable pseudo-root $(\E,D,s)$ on a graph $\Gamma$, we define $\pol(\E,D,s)$ as the spin polystable  pseudo-root $(\pol(\E,D),\pol(s))$ on $\Gamma$, where $\pol(s)=s$ (view as a sign function on $(\Gamma,\pol(\E,D))$ thanks to the identification $P_{(\E,D)}=P_{\pol(\E,D)}$, see Proposition \ref{prop:minimal-root} (2)).
\end{Def}

\begin{Prop}\label{prop:minimal-signed-root}
 Let $(\Gamma,\E,D,s)$ be a spin semistable  root-graph. We have that
      $(\Gamma,\pol(\E,D,s))$ is the unique minimal spin polystable root-graph admitting a specialization of spin semistable root-graphs to $(\Gamma,\E,D,s)$.
\end{Prop}

\begin{proof}
By Proposition \ref{prop:minimal-root} we have that  $(\Gamma,\pol(\E,D))$ is the minimal polystable root-graph admitting a specialization to $(\Gamma,\E,D)$ and, since $P_{(\E,D)}=P_{\pol(\E,D)}$, the sign function $\pol(s)=s$ is the unique sign function for which there is a specialization of spin semistable root-graphs $(\Gamma,\pol(\E,D,s))\ge (\Gamma,\E,D,s)$.   
\end{proof}

The transformation $\pol$ gives rise to a commutative diagram:
\[
\SelectTips{cm}{11}
\begin{xy} <16pt,0pt>:
\xymatrix{
 \mathbf{SPQ}_{g,n}\ar[r]\ar[d]^{\pol}& \mathbf{QR}_{g,n}  \ar[d]^{\pol} \ar[r]     & \mathbf{QD}_{\mu,g,n} \ar[d]^{\pol}  \\
\mathbf{SP}_{g,n}\ar[r] &\mathbf{PSR}_{g,n}\ar[r]& \mathbf{PSD}_{\mu,g,n}
 }
\end{xy}
\]

\section{Moduli of tropical theta characteristics}

In this section we introduce and study the moduli space of tropical theta characteristics. We describe an important refinement  of this space.

\subsection{Cones associated to semistable pseudo-roots}

First of all, we need to introduce the cone associated to a semistable root-graph.

 Consider a semistable root-graph $(\Gamma,\E,D)$. Consider the vector space $\mathbb R^{E(\Gamma^\E)}$ with coordinates $x_e$ for $e\in E(\Gamma^\E)$. For every oriented edge $e\in \ora{E}(\Gamma^\E)$, we will abuse notation writing $x_e$ for the coordinates associated to the unoriented edge associated to $e$.  
 We define  $\lambda_{(\Gamma,\E,D)}\subset \mathbb R^{E(\Gamma^\E)}$ as the cone obtained by intersecting $\mb R^{E(\Gamma^\E)}_{\ge0}$ with the linear subspace $\Lambda_{(\Gamma,\E,D)}$  defined by the equations:
\begin{equation}\label{eq:phieq}
\Lambda_{(\Gamma,\E,D)}: \;\;\sum_{e\in \gamma}\phi_{(\E,D)}(e)x_e=0
\end{equation}
where $\gamma$ runs over the set of oriented cycles of $\Gamma^\E$. 

\begin{Rem}\label{rem:containing}
Notice that if $(\E,D)$ is semistable, then $\mathcal L_\E\subset \Lambda_{(\Gamma,\E,D)}$.
\end{Rem}

We let $\lambda^\circ_{(\Gamma,\E,D)}$ be the interior of $\lambda_{(\Gamma,\E,D)}$, that is, $\lambda^\circ_{(\Gamma,\E,D)}=\lambda_{(\Gamma,\E,D)}\cap\mathbb R^{E(\Gamma^\E)}_{>0}$. 
We also set
\begin{equation}\label{eq:bED}
b_0(\E,D):=b_0(\Gamma_\E/P_{(\E,D)})
\quad \;\text{ and } \;\quad
b_1(\E,D):=b_1(\Gamma_\E/P_{(\E,D)}).
\end{equation}

Notice that the number of independent equations in \eqref{eq:phieq} is $b_1(\Gamma^\E/P_{(\E,D)})$, which is equal to $b_1(\E,D)-b_0(\E,D)+1+\#\E$. Hence
\begin{equation}\label{eq:dim-lambda}
\dim\lambda_{(\Gamma,\E,D)}=\# E(\Gamma)-b_1(\E,D)+b_0(\E,D)-1.
\end{equation}

We define
\[
\Sigma_{(\Gamma,\E,D)}=\{\lambda_{(\Gamma',\E',D')} | (\Gamma',\E',D')\in \mathbf{SSR}_{g,n} \text{ and } (\Gamma,\E,D)\ge(\Gamma',\E',D') \}.
\]

\begin{Thm}\label{thm:face-morphism}
Let $(\Gamma,\E,D)$ be a semistable pseudo-root. 
\begin{itemize}
    \item[(1)] If $(\Gamma,\E,D)\ge (\Gamma',\E',D')$ is a specialization of semistable root-graphs, then
     we have 
%an equality$\lambda_{(\Gamma',\E',D')}=\lambda_{(\Gamma,\E,D)}\cap \mathbb R^{E(\Gamma^{\E'})}$ giving rise to 
a face morphism $\lambda_{(\Gamma',\E',D')}\subset \lambda_{(\Gamma,\E,D)}$.
\item[(2)] Every face of $\lambda_{(\Gamma,\E,D)}$ is equal to a cone $\lambda_{(\Gamma',\E',D')}$ for some specialization of semistable root-graphs $(\Gamma,\E,D)\ge(\Gamma',\E',D')$.
\item[(3)] We have a decomposition
\begin{equation}\label{eq:deco-lambda}
\lambda_{(\Gamma,\E,D)}=\underset{(\Gamma,\E,D)\ge (\Gamma',\E',D')}{\coprod_{(\Gamma',\E',D')\in {\bf SSR}_{g,n}}} \lambda^\circ_{(\Gamma',\E',D')}.
\end{equation}
\item[(4)]
The set $\Sigma_{(\Gamma,\E,D)}$ is a fan whose support is $\lambda_{(\Gamma,\E,D)}$.
\end{itemize}
\end{Thm}

\begin{proof}
We prove (1). If $(\Gamma,\E,D)\ge (\Gamma',\E',D')$ is an isomorphism, the result is clear. Otherwise, 
consider the hyperplane $H\subset\mathbb R^{E(\Gamma^{\E})}$ given in the coordinates $(x_e)_{e\in E(\Gamma^{\E})}$ of $\mathbb R^{E(\Gamma^{\E})}$ by the equation
\[
H: \sum_{e\in E(\Gamma^{\E})\setminus E({\Gamma'}^{\E'})} x_e=0.
\]
Notice that $\lambda_{(\Gamma,\E,D)}$ in contained in one of the two half spaces determined by $H$. Moreover, 
the equations defining the intersection $\lambda_{(\Gamma,\E,D)}\cap H$ are obtained by putting $x_e=0$ in equation  \eqref{eq:phieq}, $\forall e\in E(\Gamma^{\E})\setminus E({\Gamma'}^{\E'})$. These equations are the equations defining $\lambda_{(\Gamma',\E',D')}$, hence $\lambda_{(\Gamma',\E',D')}=\lambda_{(\Gamma,\E,D)}\cap H.$
%\lambda_{(\Gamma,\E,D)}\cap \mathbb R^{E({\Gamma'}^{\E'})}$.

Let us prove (3). By (1), for every $(\Gamma',\E',D')$ as in the right hand side of equation \eqref{eq:deco-lambda} we have $\lambda^\circ_{(\Gamma',\E',D')}\subset \lambda_{(\Gamma,\E,D)}$. Let us show that $\lambda^\circ_{(\Gamma_1,\E_1,D_1)}\cap \lambda^\circ_{(\Gamma_2,\E_2,D_2)}=\emptyset$ for distinct $(\Gamma_1,\E_1,D_1)$, $(\Gamma_2,\E_2,D_2)$ as in the right hand side of equation \eqref{eq:deco-lambda}. This is true if $(\Gamma_1,\E_1)\ne (\Gamma_2,\E_2)$, because in this case  $\mathbb R^{E({\Gamma_1^{\E_1}})}_{>0}\cap \mathbb R^{E(\Gamma_2^{\E_2})}_{>0}=\emptyset$ (as subspaces of $\mathbb R^{E(\Gamma^\E)}$) and $\lambda^\circ_{(\Gamma_i,\E_i,D_i)}\subset \mathbb R^{E(\Gamma_i^{\E_i})}_{>0}$ for $i=1,2$. If $\Gamma':=\Gamma_1=\Gamma_2$ and $\E':=\E_1=\E_2$, then we can assume that $(\Gamma,\E,D)\ge (\Gamma',\E',D_1)$ and   $(\Gamma,\E,D)\ge (\Gamma',\E',D_2)$ have the same associated specialization $\wh\iota\col \Gamma^\E\ra {\Gamma'}^{\E'}$, otherwise $\mathbb R^{E({\Gamma_1^{\E_1}})}_{>0}\cap \mathbb R^{E(\Gamma_2^{\E_2})}_{>0}=\emptyset$ and we argue as before. Hence $D_1=\wh\iota_*(D)=D_2$, a contradiction. So the right hand side of equation \eqref{eq:deco-lambda} is contained in the left hand side.

On the other hand, let $(x_e)_{e\in E(\Gamma^\E)}\in \lambda_{(\Gamma,\E,D)}$. Let $(\Gamma,\E,D)\ge (\Gamma',\E',D')$ be the specialization of triples determined by the specialization $\wh\iota\col\Gamma^{\E}\ra {\Gamma'}^{\E'}$ contracting the edges $e$ of $E(\Gamma^\E)$ for which $x_e=0$ (so $D'=\wh{\iota}_*(D)$). Since $2D+\Div(\phi_{(\E,D)})=K_{\Gamma^{\E}}$, we have 
$2D'+\Div(\wh{\iota}_*(\phi_{(\E,D)}))=K_{{\Gamma'}^{\E'}}$ and by equation \eqref{eq:phieq} the flow $\wh\iota_*(\phi_{(\E,D)})$ is acyclic. Hence $(\Gamma',\E',D')$ is a semistable root-graph and  $(x_e)_{e\in E({\Gamma'}^{\E'})}\in \lambda_{(\Gamma',\E',D')}^\circ$. This concludes the proof of (3).

Notice that (2) and (4) are consequences of (3) and (1), hence we are done. 
\end{proof}

Recall the  $\phi_{(\E,D)}$-graphs $P_u$ of $\Gamma^\E$, for $u\in V(\Gamma^\E/P_{(\E,D)})$ (see Definition \ref{def:phi-subgraph}).

\begin{Lem-Def}\label{lem-def}
Let $(\Gamma,\E,D)\ge (\Gamma',\E',D')$ be a specialization of triples where $(\Gamma,\E,D)$ is a semistable root-graph. Assume that the  associated specialization $\wh\iota\col \Gamma^{\E}\ra {\Gamma'}^{\E'}$ satisfies one of the following conditions:
\begin{itemize}
    \item[(1)] $\wh\iota$ is the contraction of an edge $e\in P_{(\E,D)}\subset E(\Gamma)$.
    \item[(2)] 
     $\wh\iota$ is the contraction of  $E(P_u,P_{u'})$ for distinct vertices $u,u'\in V(\Gamma^{\E}/P_{(\E,D)})$ such that $E(P_u,P_{u'})$ is a $\phi_{(\E,D)}$-cut of $\Gamma^{\E}$. 
\end{itemize}
Then 
$(\Gamma,\E,D)\ge (\Gamma',\E',D')$ is a specialization of semistable root-graphs. We call  $(\Gamma,\E,D)\ge (\Gamma',\E',D')$  elementary (of
 type $1$ or $2$, respectively, if (1) or (2) holds).
\end{Lem-Def}

\begin{proof}
Since $2D+\Div(\phi_{(\E,D)})=K_{\Gamma^{\E}}$, we have $2D'+\Div(\wh{\iota}_*(\phi_{(\E,D)}))=K_{{\Gamma'}^{\E'}}$. So, by Proposition \ref{prop:stab-contract} we are left to show that $\wh{\iota}_*\phi_{(\E,D)}$ is acyclic. This is clear if  (1) holds, since $\phi_{(\E,D)}(e)=0$ for $e\in P_{(\E,D)}$. If  (2) holds, then $\wh{\iota}_*\phi_{(\E,D)}$ is acyclic by the definition of $\phi_{(\E,D)}$-cut.
\end{proof}

\begin{Cor}\label{cor:facet}
The following properties hold.
\begin{itemize}
    \item[(1)] For a semistable pseudo-root  $(\E,D)$ on a graph $\Gamma$, a cone $\lambda_{(\Gamma',\E',D')}\in \Sigma_{(\Gamma,\E,D)}$ is a facet of $\lambda_{(\Gamma,\E,D)}$ if and only if the associated specialization $(\Gamma,\E,D)\ge  (\Gamma',\E',D')$ is elementary.
     \item[(2)] Every specialization of semistable pseudo-roots can be written (up to isomorphisms) as a chain of elementary specializations.
     \item[(3)] An elementary specialization of semistable root-graphs can not be factored as a non-trivial chain of specializations of semistable root-graphs. 
\end{itemize}  
\end{Cor}

\begin{proof}
Consider an elementary specialization $ (\Gamma,\E,D)\ge (\Gamma',\E',D')$. If $ (\Gamma,\E,D)\ge (\Gamma',\E',D')$ is of type 1, then  $\#E(\Gamma')=\#E(\Gamma)-1$ and $\dim \Lambda_{(\Gamma',\E',D')}=\dim \Lambda_{(\Gamma,\E,D)}$. If $ (\Gamma,\E,D)\ge (\Gamma',\E',D')$ is of type 2, then $\#E(\Gamma')=\#E(\Gamma)-\#E(P_u,P_{u'})$ and 
\[
\dim \Lambda_{(\Gamma',\E',D')}=\dim \Lambda_{(\Gamma,\E,D)}-\#E(P_v,P_{v'})+1.
\]
Hence, in both cases,  $\lambda_{(\Gamma',\E',D')}$ is a facet of $\lambda_{(\Gamma,\E,D)}$. Conversely, assume that $\lambda_{(\Gamma',\E',D')}$ is a facet of $\lambda_{(\Gamma,\E,D)}$ and let $(\Gamma,\E,D)\ge (\Gamma',\E',D')$ and $\wh{\iota}\col \Gamma^\E\ra\Gamma'^{\E'}$ be the associated specializations. In particular, $\wh\iota_*(\phi_{(\E,D)})$ is acyclic. If $(\Gamma,\E,D)\ge (\Gamma',\E',D')$ is not elementary, by what we have already shown we have that there are distinct vertices $u,u'\in V(\Gamma^\E/P_{(\E,D)})$ and a proper non-empty set $F\subset E(P_u,P_{u'})$ such that $\wh \iota$ contracts $E(P_u,P_{u'})\setminus F$ and does not contract $F$. Thus $\wh\iota$ sends each oriented edge in $\ora{F}$ to a positive $\wh{\iota}_*(\phi_{(\E,D)})$-loop in $(\Gamma')^{\E'}$, which is a contradiction. This proves (1). Notice that (2) and (3) are consequences of (1) and Theorem \ref{thm:face-morphism}. 
\end{proof}

\begin{Rem}
If $(\Gamma,\E,D)$ is $v_0$-quasistable, then 
%by Proposition \ref{prop:stab-contract} 
the statements of Theorem 
%\ref{prop:lambda-deco}, 
\ref{thm:face-morphism} and Corollary \ref{cor:facet} hold if we replace the word \emph{semistable} with \emph{$v_0$-quasistable} 
(specializations preserves semistability and quasistability).
\end{Rem}

Recall the definition of the map $\mc T_\E$ in equation \eqref{eq:proy}. 
For a semistable root-graph $(\Gamma,\E,D)$, we define the cone $\kappa_{(\Gamma,\E,D)}$ and its interior $\kappa^\circ_{(\Gamma,\E,D)}$ as:
\[
\kappa_{(\Gamma,\E,D)}=\mathcal{T}_\E(\lambda_{(\Gamma,\E,D)})
\quad\text{ and }\quad
\kappa^\circ_{(\Gamma,\E,D)}=\mathcal{\T}_\E(\lambda_{(\Gamma,\E,D)}\cap\R^{E(\Gamma^\E)}_{>0}).
\]
%where $\mathcal{T}_\E$ is the map defined in , and we let $\kappa^\circ_{(\Gamma,\E,D)}$ be the interior of $\kappa_{(\Gamma,\E,D)}$, that is $\kappa^\circ_{(\Gamma,\E,D)}=\mathcal{\T}_\E(\lambda_{(\Gamma,\E,D)}\cap\R^{E(\Gamma^\E)}_{>0})$. 
By construction, we have
\begin{equation}\label{eq:kappa-sigma}
\kappa_{(\Gamma,\E,D)}\subset \sigma_{(\Gamma,\E,D)}\subset \mathbb R^{E(\Gamma^\E)}/\mc L_\E.
\end{equation}

%It is easy to check that $\lambda_{(\Gamma_2,\E_2,D_2)}=\lambda_{(\Gamma_1,\E_1,D_1)}\cap \mathbb R^{E(\Gamma^{\E_2})}$, namely $\lambda_{(\Gamma_2,\E_2,D_2)}$ is obtained intersecting $\lambda_{(\Gamma_1,\E_1,D_1)}$ with the . Thus we get an inclusion $\lambda_{(\Gamma_2,\E_2,D_2)}\subset \lambda_{(\Gamma_1,\E_1,D_1)}$, which is a face morphism. 

\begin{Rem}\label{rem:simple-root}
Notice that if $(\E,D)$ is simple, then we have $\mc L_\E=0$ in  equation \eqref{eq:proy}, and hence $\lambda_{(\Gamma,\E,D)}=\kappa_{(\Gamma,\E,D)}$.
\end{Rem}

\begin{Rem}\label{rem:kappa-inclusion}
If $(\Gamma,\E,D)\ra (\Gamma',\E',D')$ is a specialization of semistable root-graphs, then
using Remark \ref{rem:inclu} and Theorem \ref{thm:face-morphism}, we deduce that
\[
\kappa_{(\Gamma',\E',D')}=\kappa_{(\Gamma,\E,D)}\cap\sigma_{(\Gamma',\E',D')}=\kappa_{(\Gamma,\E,D)}\cap \mathbb R^{E(\Gamma^{\E'})}/\mc L_{\E'},
\]
and hence we get an inclusion $\kappa_{(\Gamma',\E',D')}\subset \kappa_{(\Gamma,\E,D)}$.
\end{Rem}

\begin{Rem}\label{rem:face-morphism}
As we shall see in Proposition \ref{prop:refiT}, the inclusion $\kappa_{(\Gamma',\E',D')}\subset \kappa_{(\Gamma,\E,D)}$ might fail to be a face morphism when $(\E,D)$ is polystable and $(\E',D')$ is $v_0$-quasistable. 
Nevertheless, 
%if $(\Gamma_1,\E_1,D_1)$ and $(\Gamma_1,\E_1,D_1)$ are simple, then by Remarks \ref{rem:simple-root} and Lemma \ref{lem:face-morphism} the inclusion $\kappa_{(\Gamma_2,\E_2,D_2)}\subset \kappa_{(\Gamma_1,\E_1,D_1)}$ is a face morphism. In
in Lemma \ref{lem:poly-pseudo-root}, we will see that it is a face morphism if $(\E,D)$ and $(\E',D')$ are polystable pseudo-roots. 
\end{Rem}

% We consider the sub-cone complex $T_{g,n}^{\trop}\subset \Theta^{\trop}_{g,n}$, called the \emph{moduli space of polystable tropical theta characteristics}
%obtained by restricting to the category $\mathbf{PSR}_{g,n}$. In other words we have:  
%\[
%T^{\trop}_{g,n}:=\underset{\longrightarrow}{\lim}\;\kappa_{(\Gamma,\E,D)}=\coprod_{[\Gamma,\E,D]} \kappa^\circ_{(\Gamma,\E,D)}/\Aut(\Gamma,\E,D),
%\]
%where the limit is taken over the category $\mathbf{PSR}_{g,n}$ and the union is taken over all equivalence classes $[\Gamma,\E,D]$ in $\mathcal{PSR}_{g,n}$.
%and $\mathbf{QR}_{g,n}$. 

\begin{Prop}\label{prop:dim-cell}
We have 
$\dim\kappa_{(\Gamma,\E,D)}=\#E(\Gamma)-b_1(\E,D)$, 
for every semistable root-graph  $(\Gamma,\E,D)$.
\end{Prop}

\begin{proof}
By \eqref{eq:dim-lambda} and the following computation:
\[
\dim\mc L_\E=b_0(\Gamma_\E)-1=b_0(\Gamma_\E/P_{(\E,D)})-1=b_0(\E,D)-1,
\]
we need only to check that $\mc L_\E\subset \Lambda_{(\Gamma,\E,D)}$, which easily follows by Theorem \ref{thm:uniquephi}.
\end{proof}

\begin{Def}
Let $X$ be a tropical curve. A \emph{root} of $X$ is a  divisor $\D$ on $X$ such that the divisor $2\D$ is linearly equivalent to the canonical divisor $K_X$. A root is \emph{unitary, semistable, polystable, $v_0$-quasistable} for some point $p_0$ of $X$, if so is $\D$ as a divisor on $X$. 
\end{Def}

\begin{Lem}\label{lem:comb-type}
Let $\D$ be a unitary semistable root of a tropical curve $X$ with model $\Gamma$. Then the combinatorial type $(\E,D)$ of $\D$ is a pseudo-root on $\Gamma$.
\end{Lem}

\begin{proof}
There is a rational function $f$ on $X$ such that 
$2\D+\Div(f)=K_X$. Since $\mc D$ is unitary, we have $\supp(\D)\subset V(\Gamma^\E)$. But $\supp(K_X)\subset V(\Gamma)$, hence $\supp(\Div(f))\subset V(\Gamma^\E)$. Then  $\Div(\phi_f)$ is a divisor on $\Gamma^\E$, and we have $2D+\Div(\phi_f)=K_{\Gamma^\E}$. Since $\phi_f$ is the flow of a rational function, it is acyclic, and we are done. 
\end{proof}

\begin{Prop}\label{prop:open-cone}
Let $(\Gamma,\E,D)$ be a semistable root-graph. The open cone $\kappa^0_{(\Gamma,\E,D)}$ parametrizes classes of pairs $(X,\D)$,  where $X$ is a $n$-pointed tropical curve with model $\Gamma$ and $\D$ is a semistable unitary root of $X$ with combinatorial  type $(\E,D)$, and $(X,\D)$ is equivalent to $(X',\D')$ if $X=X'$ and $\D$ is linearly equivalent to $\D'$.
\end{Prop}

\begin{proof}
By \cite[Proposition 4.6]{AAPT}, the open cone  $\sigma^\circ_{(\Gamma,\E,D)}$ parametrizes equivalence classes of pairs $(X,D)$, where $X$ is a tropical curve with model $\Gamma$ and $\D$ is a unitary divisor with combinatorial type $(\E,D)$ (hence $\D$ is semistable as a divisor on $X$), and two pairs $(X,\D)$ and $(X',D')$ are equivalent if $X=X'$ and the divisors $\D$ and $\D'$ are linearly equivalent. 

Consider a point $(x_e)_{e\in E(\Gamma^\E)}\in \mathbb R^{E(\Gamma^\E)}_{>0}$ whose class in $\sigma^\circ_{(\Gamma,\E,D)}$ corresponds to a pair $(X,\D)$. A model of $X$ is $\Gamma^\E$, with length function taking an edge $e\in E(\Gamma^\E)$ to $x_e$. If $\D$ is a root of $X$, then there is a rational function $f$ on $X$ such that $2\D+\Div(f)=K_X$. Thus $\Div(\phi_f)$ is a divisor on $\Gamma^\E$ such that $2D+\Div(\phi_f)=K_{\Gamma^\E}$. Hence $\phi_f=\phi_{(\E,D)}$ (by Theorem \ref{thm:uniquephi}) and $\phi_{(\E,D)}$ satisfies equations \eqref{eq:phieq} because it is the flow associated to a rational function on $X$. This implies that $(x_e)_{e\in E(\Gamma^\E)}$ is a point of $\lambda^\circ_{(\Gamma,\E,D)}$, so its class belongs to $\kappa^\circ_{(\Gamma,\E,D)}$. Conversely, if $(x_e)_{e\in E(\Gamma^\E)}$ is a point of $\lambda^\circ_{(\Gamma,\E,D)}$, then equations \eqref{eq:phieq} implies the existence of a rational function $f$ on $X$ such that $\phi_f=\phi_{(\E,D)}$. Thus $2D+\Div(\phi_f)=K_{\Gamma^\E}$, which implies that $2\D+\Div(f)=K_X$, i.e., that $\D$ is a root of $X$.
\end{proof}

\begin{Rem}
A similar argument can be used to prove the statement of Proposition \ref{prop:open-cone} where we replace the word \emph{semistable} either with \emph{polystable}, or with \emph{$v_0$-quasistable}. 
\end{Rem}

\subsection{The moduli space of tropical theta characteristics}

Let $\Gamma$ be a graph. 
Given a subset $\E\subset E(\Gamma)$,
we consider the linear injective map 
\begin{equation}\label{eq:delta}
\delta_\E\col \mb R^{E(\Gamma)}\ra \mathbb{R}^{E(\Gamma^\E)}
%\kappa_{(\Gamma,\E,D)}
\end{equation}
taking $(x_e)_{e\in E(\Gamma)}$ to $(x_{e'})_{e'\in E(\Gamma^\E)}$, where $x_{e'}=x_e$ if the edge $e'\in E(\Gamma^\E)$ is over the edge $e\in E(\Gamma)$. Denote by $\Delta_\E:=\delta_\E(\mathbb R^{E(\Gamma)}_{\ge0})$, so that we have 
\[
\Delta_\E=
\{(x_e)_{e\in E(\Gamma^\E)} | x_{e_1}=x_{e_2}\ge0,  \forall e_1,e_2 \text{ over the same edge } e\in \E \}\subset \mb R^{E(\Gamma^\E)}_{\ge0}.
\]

\begin{Lem}\label{lem:poly-pseudo-root}
Let $(\Gamma,\E,D)$ be a polystable root-graph. 
\begin{itemize}
    \item[(1)] We have isomorphisms:
\[
\mathbb R^{E(\Gamma)}_{\ge0}\stackrel{\underset{\simeq}{\delta_\E}}{\ra}\Delta_\E\stackrel{\underset{\simeq}{\mc T_\E}}{\ra}\kappa_{(\Gamma,\E,D)}\subset \mathbb R^{E(\Gamma^\E)}/\mathcal{L}_\E;
\]
    \item[(2)] If $(\Gamma,\E,D)\ra (\Gamma',\E',D')$ is a specialization of polystable root-graphs, then  
    we have a face morphism $\kappa_{(\Gamma',\E',D')}\ra \kappa_{(\Gamma,\E,D)}$. Conversely, every face of $\kappa_{(\Gamma,\E,D)}$ arises in this way.
\end{itemize}
\end{Lem}

\begin{proof}
We show (1).
It is clear that $\delta_\E\col \mathbb R^{E(\Gamma)}_{\ge0}\ra\Delta_\E$ is an isomorphism. 
 By Lemma \ref{lem:phizero}
 %, we have $\phi_{(\E,D)}(e_1)=\phi_{(\E,D)}(e_2)$,for every $e_1,e_2\in \ora{E}(\Gamma^\E)$ such that $t(e_1)=t(e_2)$ is an exceptional vertex. Hence 
 we have that $\Delta_\E\subset \lambda_{(\Gamma,\E,D)}$ and, by equation \eqref{eq:wV}, we have $\Delta_\E\cap\mathcal L_\E=0$.
 Therefore, by restricting $\mathcal T_\E$ to $\Delta_\E$, we obtain an injective linear map
 \begin{equation}\label{eq:restriction}
 \mathcal T_\E\col \Delta_\E\ra \kappa_{(\Gamma,\E,D)}.
 \end{equation}

 We show that the map in equation \eqref{eq:restriction} is surjective.
 Assume first that $(\E,D)$ is simple (recall Remark \ref{rem:simple-root}). Fix an spanning tree $T$ of $\Gamma$ not intersecting $\E$. Hence, for each $e'\in\ora\E$, there exists a unique oriented cycle $\gamma_{e'}$ on $\Gamma$ such that $E(\gamma_{e'})\setminus T=\{e'\}$. Then, seeing $\gamma_{e'}$ as an oriented cycle on $\Gamma^\E$, from the equation: 
 \begin{equation*}
\sum_{e\in \gamma_{e'}}\phi_{(\E,D)}(e)x_e=0,
\end{equation*}
we have that $x_{e_1}=x_{e_2}$ for $e_1, e_2$ over the associated unoriented edge to $e'\in\ora\E$. So, in this case, we also have $\kappa_{(\Gamma,\E,D)}=\lambda_{(\Gamma,\E,D)}\subset\Delta_\E$, concluding the proof of (1) when $(\E,D)$ is simple.

Next, assume that $(\E,D)$ is not simple. It suffices to prove that for every vector $v\in \lambda_{(\Gamma,\E,D)}$ there is a vector $w\in \mathcal L_\E$ such that $v+w\in \Delta_\E$. 
We argue by induction on $\#\E$. Recall that $u_e$ denotes the vector of the canonical basis of $\mathbb R^{E(\Gamma^\E)}$ corresponding to $e\in E(\Gamma^\E)$.
If $\#\E=1$, we have that $\mathcal{L}_\E$ is generated by $u_{e_1}-u_{e_2}$, where $e_1$ and $e_2$ are the edges over the unique edge in $\E$. In this case, we  conclude observing that, given a vector $v\in\lambda_{(\Gamma,\E,D)}$ with coordinates $(x_e)_{e\in E(\Gamma^\E)}$, 
\[
v+\frac{x_{e_2}-x_{e_1}}{2}(u_{e_1}-u_{e_2})\in\Delta_\E.
\]

If $\#\E>1$, we can choose $\E'\subset\E$ such that $\E'$ is a nonempty minimal disconnecting subset of edges of $\Gamma$.
%, and such that $|\E'|$ is maximal among the cardinality of these sets. If $|\E'|=1$, then 
Hence $\Gamma_{\E'}$ has two connected components, which we denote $\Gamma_1$ and $\Gamma_2$. Fix $i\in\{1,2\}$. For a vector $v\in\lambda_{(\Gamma,\E,D)}$ with coordinates $(x_e)_{e\in E(\Gamma^\E)}$, we let $v_i:=\pi_i(v)\in \R^{E(\Gamma_i^{\E_i})}$, where $\E_i=E(\Gamma_i)\cap\E$ and $\pi_i$ is the natural projection $\pi_i\col\mathbb R^{E(\Gamma^\E)}\ra \mathbb R^{E(\Gamma_i^{\E_i})}$. Since $v$ satisfies equations \eqref{eq:phieq} for cycles in $\Gamma$,  we see that $v_i$ satisfies equations \eqref{eq:phieq} for cycles in $\Gamma_i$, for $i=1,2$. Moreover, 
 set:
\[
P_i=E(\Gamma_i)\setminus\E_i
\quad
\text{ and }
\quad
D_i=D|_{\Gamma_i}-\sum_{v\in V(\Gamma_i)}\deg_{\E'}(v)v.
\]
Then $P_i$ is a cyclic subgraph, because $P=E(\Gamma)/\mathcal{E}$ is cyclic and $P=P_1\sqcup P_2$, and, by equation \eqref{eq:DP}, we have $(\E_i,D_i)=(\E_{P_i},D_{P_i})$. By Proposition \ref{prop:EP}, we have that  $(\E_i,D_i)$ is a polystable root on $\Gamma_i$ , and hence $v_i\in\lambda_{(\Gamma_i,\E_i,D_i)}$.\par

By the induction hypothesis, there exists $w_i\in\mathcal{L}_{\E_i}$ such that $v_i+w_i\in\Delta_{\E_i}$. Notice that for each $V\subset V(\Gamma_i)$ such that $E(\Gamma_i)\cap E(V,V^c)\subset\E_i$, we have that $E(V,V^c)\subset\E_i\cup\E'$. Then each vector in $\mathcal{L}_{\E_i}$ is the projection in $\R^{E(\Gamma_i^{\E_i})}$ of a vector in $\mathcal{L}_\E\cap\R^{E(\Gamma^\E)\setminus E(\Gamma_{_{3-i}}^{^{\E_{3-i}}})}$. Let $u_i$ be a vector in $\mathcal{L}_\E\cap\R^{E(\Gamma^\E)\setminus E(\Gamma_{_{3-i}}^{^{\E_{3-i}}})}$ such that its projection to $\R^{E(\Gamma_i^{\E_i})}$ is $w_i$. Define the vector $u:=v+u_1+u_2$.
Then $u$ projects to $v_i+w_i$ in $\R^{E(\Gamma_i^{\E_i})}$ via $\pi_i$, since $\pi_i(u_{3-i})=0$. So the coordinates $(y_e)_{e\in E(\Gamma^\E)}$ of $u$, for any edge $e\in\E\setminus\E'$, satisfy 
\begin{equation}
\label{eq:y=}
    y_{e_1}=y_{e_2},
\end{equation}
where $e_1$, $e_2$ are over $e$. 
Set $V_1=V(\Gamma_1)$. We have that $E(V_1,V_1^c)=\E'$ and $u_{V_1}\in \R^{E(\Gamma^\E)\setminus(E(\Gamma_1^{\E_1})\cup E(\Gamma_2^{\E_2}))}$ (recall Equation \eqref{eq:wV}). If $\#\E'=1$, with $e_1$ and $e_2$ the edges over the unique edge $e\in \E'$, then $u_{V_1}=u_{e_1}-u_{e_2}$ and hence we are done since:
\[
u+\frac{x_{e_2}-x_{e_1}}{2}u_{V_1}\in \Delta_\E.
\]

We are left with the case $\#\E'\ge2$.
For every two different edges $e',e''\in\E'$,  we can choose an oriented cycle $\gamma_{e',e''}$ of $\Gamma$ such that $E(\gamma_{e',e''})\cap\E'=\{e',e''\}$ (recall that $\E'$ is minimal and hence $V_1$ and $V_1^c$ induce connected subgraphs). Since $\mc L_\E\subset \Lambda_{(\Gamma,\E,\D)}$ by Remark \ref{rem:containing}, we have that the vector $u\in \Lambda_{(\Gamma,\E,D)}$ (recall Equation \eqref{eq:phieq}).
%Since $\mc L_\E$ is clearly contained in $\lambda_{(\Gamma,\E,D)}$, we have that $u\in \lambda_{(\Gamma,\E,D)}$.
%We can see that every vector in $\mathcal{L}_\E$ satisfies the equations which define $\lambda_{(\Gamma,\E,D)}$. Thus $u$ satisfies the equations which define $\lambda_{(\Gamma,\E,D)}$. 
Then:
\begin{equation*}
\sum_{e\in \gamma_{e',e''}}\phi_{(\E,D)}(e)y_e=0,
\end{equation*}
(where we see $\gamma_{e',e''}$ as an oriented cycle in $\Gamma^\E$) and so, if $e'_1,e'_2$ are over $e'$, and $e''_1,e''_2$ over $e''$, with $e'_1, e''_1$  incident to $V(\Gamma_1)$, we get:
\[
0=\sum_{e\in \gamma_{e',e''}}\phi_{(\E,D)}(e)y_e=y_{e'_1}-y_{e'_2}-(y_{e''_1}-y_{e''_2})+\sum_{e\in \gamma_{e',e''}\setminus \E'}\phi_{(\E,D)}(e)y_e.
\]
By Equation \eqref{eq:y=} and Lemma \ref{lem:phizero} we have that $\sum_{e\in \gamma_{e',e''}\setminus \E'}\phi_{(\E,D)}(e)y_e=0$ and hence
\begin{equation}\label{eq:con}
   y_{e'_1}-y_{e'_2}=y_{e''_1}-y_{e''_2}.
\end{equation}

  Fix an edge $e'\in\E'$ and let $e'_1,e'_2$ be the different edges of $\Gamma^\E$ over $e'$, with $e'_1$ incident to $V(\Gamma_1)$.  The vector $u+\frac{y_{e'_2}-y_{e'_1}}{2}u_{V_1}$, with coordinates $(z_e)_{e\in E(\Gamma^\E)}$, satisfies $z_{e_1}=y_{e_1}=y_{e_2}=z_{e_2}$ for every $e_1,e_2$ over an edge $e\in\E\setminus\E'$ and 
  \[
   z_{e''_1}=y_{e''_1}+\frac{y_{e'_2}-y_{e'_1}}{2}=y_{e''_2}-\frac{y_{e'_2}-y_{e'_1}}{2}=z_{e''_2}
  \]
for every edge $e''\in \E'$ (where in the second equality we used \eqref{eq:con}).\par
%, for $e_1,e_2$ over an edge $e\in\E'$ with $e_1$ incident to $V_1$. 

Moreover, for every edge $e\in \E$ we have that  $2z_{e_1}=2z_{e_2}=z_{e_1}+z_{e_2}=x_{e_1}+x_{e_2}$, where $e_1,e_2$ are the edges over $e$. Then the vector $v+u_1+u_2+\frac{y_{e'_2}-y_{e'_1}}{2}u_{V_1}$ has non-negative coordinates, so it is contained in $\Delta_\E$. This concludes the proof of (1).
 
Let $(\Gamma,\E,D)\ge (\Gamma',\E',D')$ be a specialization of polystable root-graphs. If $\Gamma\ra \Gamma'$ is the induced specialization, then we have a natural face morphism $\mathbb R^{E(\Gamma')}_{\ge0}\ra\mathbb R^{E(\Gamma)}_{\ge0}$ which,
by the first part of the proof can be identified with a face morphism $\kappa_{(\Gamma',\E',D')}\ra \kappa_{(\Gamma,\E,D)}$. It is clear that all the faces arises in this way, since every face of $\mathbb R^{E(\Gamma)}_{\ge0}$ is of type $\mathbb R^{E(\Gamma')}_{\ge0}$, for some graph $\Gamma'$ and some specialization $\Gamma\ra \Gamma'$.
%\cong \mathbb R^{E(\Gamma_2)}_{\ge0}=\mathbb R^{E(\Gamma_1)}_{\ge0}\cap\{(x_e)_{e\in E(\Gamma_2)}|x_e=0, \; \forall e\in E(\Gamma_1)\setminus E(\Gamma_2)\}.
%\]
\end{proof}

We are now in a position to introduce the following generalized cone complex.

\begin{Def}
We define the generalized cone complex, called the \emph{moduli space of tropical theta characteristics}, as
\[
T^{\trop}_{g,n}:=\underset{\longrightarrow}{\lim}\;\kappa_{(\Gamma,\E,D)}=\coprod_{[\Gamma,\E,D]} \kappa^\circ_{(\Gamma,\E,D)}/\Aut(\Gamma,\E,D),
\]
where the limit is taken over the category $\mathbf{PSR}_{g,n}$ and the union is taken over all equivalence classes $[\Gamma,\E,D]$ in $\mathcal{PSR}_{g,n}$.
\end{Def}
\begin{Rem}
  It is worth noting that although $\mathbb{R}^{E(\Gamma)}_{\geq0}$ is isomorphic to $\kappa_{(\Gamma,\E,D)}$ we take quotient via the group $\Aut(\Gamma,\E,D)$ which is a (possibly proper) subgroup of  $\Aut(\Gamma)$
\end{Rem}

\begin{Prop}\label{prop:modularT}
The generalized cone complex $T^{\trop}_{g,n}$ parametrizes equivalence classes of pairs $(X,\D)$, where $X$ is a stable $n$-pointed  tropical curve of genus $g$, and $\D$ is a polystable root of $X$, and where $(X,\D)$ is equivalent to $(X',\D')$ if there is an isomorphism $\iota\col X\ra X'$ such that $\iota_*(\D)$ is linearly equivalent to $\D'$.
\end{Prop}

\begin{proof}
Let $X$ be a $n$-pointed stable tropical curve of genus $g$ and let $\Gamma:=\Gamma_X$ be the genus-$g$ stable  graph with $n$ legs which is a model of $X$.
   By   Lemma \ref{lem:comb-type}, 
   a polystable root divisor $\D$ of $X$ has combinatorial type  
   equal to a polystable pseudo-root $(\E,D)$ on $\Gamma$.  Therefore, by Lemma \ref{prop:open-cone}, we see that $(X,\D)$ corresponds to a point in $\kappa^\circ_{(\Gamma,\E,D)}$, and hence to a point in $T^\trop_{g,n}$. If $(X,\D)$ and $(X',\D')$ are  equivalent, then there is an isomorphism $\iota\col X\to X'$ such that $\iota_*(\D)$ is equivalent to $\D'$. 
   Consider the three points $p_1,p_2,p_3$ in $\kappa^\circ_{(\Gamma,\E,D)}$ corresponding to $(X,\D)$, $(X',\iota_*(\D))$, $(X',\D')$. The points $p_1$ and $p_2$ will get identified in  $\kappa^\circ_{(\Gamma,\E,D)}/\Aut(\Gamma,\E,D)$, while $p_2$ and $p_3$ are equal in $\kappa^\circ_{(\Gamma,\E,D)}$. Hence, $p_1$ and $p_3$ will correspond to the same point in
   \[
   \kappa_{\Gamma,\E,D}^\circ/\Aut(\Gamma,\E,D)\subset T_{g,n}^{\trop}.
   \]
   
If $(X,\D)$ and $(X',\D')$ correspond to the same point  contained in some cell $\kappa^\circ_{(\Gamma,\E,D)}/\Aut(\Gamma,\E,D)$ of $T^\trop_{g,n}$, then there is an isomorphism $\Gamma=\Gamma_X\stackrel{\iota}{\cong}\Gamma_{X'}$ such that $\iota_*(\E,D)=(\E',D')$. Moreover, it follows  that the metrics of $X$ and $X'$ are equal, hence $\iota$ induces an isomorphism of tropical curves $\iota\col X\to X'$. This implies that $(X',\D')$ and $(X',\iota_*(\D))$ are the same point in $\kappa_{(\Gamma,\E,D)}^\circ$. By Lemma \ref{prop:open-cone}, the divisors $\D'$ and $\iota_*(\D)$ are equivalent, so $(X,\D)$ and $(X',\D')$ are equivalent, as required.\par

Conversely, given a triple $(\Gamma,\E,D)$ in $\mathbf{PSR}_{g,n}$, using again Lemma \ref{prop:open-cone} we see that every point in $\kappa_{(\Gamma,\E,D)}^\circ$ corresponds to a pair $(X,\D)$, where $X$ is a tropical curve and $\D$ is a polystable root of $X$. It is easy to see that we have established a one-to-one correspondence between the set of points of $T_{g,n}^\trop$ and the set of equivalence classes of pairs $(X,\D)$.
\end{proof}

Given a tropical curve $X$ with underlying graph $\Gamma$, and a cyclic subgraph $P$ of $\Gamma$, we let $\E_P:=E(\Gamma)\setminus E(P)$, and we define the divisor $\D_P$ on $X$ as: 
\begin{equation}\label{eq:DPoly}
\D_P=\sum_{v\in V(\Gamma)} \left(\deg_X(v)-\frac{\deg_P(v)}{2}-1+w_X(v)\right)v-\sum_{e\in \E_P}p_e,
\end{equation}
where $p_e$ is the mid-point of the edge $e$. Notice that the combinatorial type of $\D_P$ is $(\E_P,D_P)$, 
%, we denote by $(\E_P,D_P)$ the pseudo-divisor on $\Gamma$, where $\E_P=E(\Gamma)\setminus P$
hence  $\D_P$ is a polystable divisor on $X$ by Lemma \ref{prop:EP}.

Recall that a theta characteristic on a tropical curve is the class $[\mc D]$ of a divisor $\mc D$ such that $2\mc D$ is equivalent to $K_X$.

\begin{Prop}\label{prop:equiv}
Let $X$ be a tropical curve with underlying graph $\Gamma$. 
Then, the divisor $\D_P$ is a polystable root of $X$ for every cyclic subgraph $P$ of $\Gamma$. Conversely, every theta characteristic on $X$ is the class of a divisro $\D_P$, for a unique  cyclic subgraph $P$.
\end{Prop}

\begin{proof}
Consider a cyclic subgraph $P$ of $\Gamma$. 
Given $e\in \E_P$ incident to the vertices $v_1,v_2$ of $\Gamma$, we let $f_e$ be the rational function on $X$ with slope $1$ over the segments $[v_1,p_e]$ and  $[v_2,p_e]$, and slope $0$ everywhere else. We have
\[
\D_P+\sum_{e\in \E_P} \Div(f_e)=\sum_{v\in V(\Gamma)}\left(\frac{\deg_P(v)}{2}-1+w_X(v)\right)v+\sum_{e\in \E_P} p_e.
\]
 Thus, by Theorem \cite[Proposition 3.6]{CMP2}, we deduce  that every theta characteristic has exactly one representative equal to $\D_P$. 
\end{proof}

The following corollary is an immediate consequence of Propositions \ref{prop:modularT} and \ref{prop:equiv}, and justifies the name ``moduli space of theta characteristics" given to $T_{g,n}^\trop$.

\begin{Cor}
The generalized cone complex $T^{\trop}_{g,n}$ parametrizes equivalence classes of pairs $(X,[\D])$ where $X$ is a stable $n$-pointed  tropical curve of genus $g$ and $[\D]$ is a theta characteristic on $X$, and where $(X_1,[\D_1])$ is equivalent to $(X_2,[\D_2])$ if there is an isomorphism $\iota\col X_1\ra X_2$ such that $[\iota_*(\D_1)]=[\D_2]$.
\end{Cor}

By Propositions \ref{prop:modularP} and \ref{prop:modularT} it is clear that there is   an inclusion 
\[
\psi^\trop_T\col T^\trop_{g,n}\hookrightarrow P^\trop_{\mu,g,n},
\]
which is a morphism of generalized cone complexes. 
%taking the point parametrizing the class of a pair $(X,\D)$, for a tropical curve $X$ and a polystable square root $\D$ of $X$ to the point parametrizing class of the pair $(X,\D)$. Indeed, 
Indeed, for every genus-$g$ stable graph $\Gamma$ with $n$ legs  and a polystable pseudo-root $(\E,D)$ on $\Gamma$, the inclusion $\kappa_{(\Gamma,\E,D)}\subset \sigma_{(\Gamma,\E,D)}$ 
in equation \eqref{eq:kappa-sigma} factors
the composition:
\[
\kappa_{(\Gamma,\E,D)}\ra T_{g,n}^\trop \stackrel{\psi^\trop_T}{\longrightarrow}P_{\mu,g,n}^\trop.
\]
We can view $T_{g,n}^\trop$ as a generalized sub cone complex of $P_{\mu,g,n}^\trop$. The restriction of the forgetful morphism $\pi_P^\trop\col P_{\mu,g,n}^\trop\ra M_{g,n}^\trop$ to $T_{g,n}^\trop$ gives rise to a forgetful morphism of generalized cone complexes 
\[
\pi^\trop_T\col T_{g,n}^\trop\ra M_{g,n}^\trop
\]
taking the point parametrizing the class of a pair $(X,\D)$, for a tropical curve $X$ and a polystable square root $\D$ of $X$ to the point parametrizing the class of $X$.
%the tropical curve $X^\E$. As in \cite[Proposition 2.5.3]{CMP1}, the map $\pi^\trop_T$ is a morphism of generalized cone complexes.
%The map defined in equation \eqref{eq:psi} factors the composition

\subsection{Refining the moduli space $T_{g,n}^\trop$}
 
We construct a refinement of the moduli space $T_{g,n}^\trop$ of polystable tropical theta characteristics using $v_0$-quasistability, as we did for the universal tropical Jacobian $P_{\mu,g,n}^\trop$.

%\begin{Lem}\label{lem:simple-pseudo-root}
%
Recall that, if $(\Gamma,\E,D)$ is a simple semistable graph, then $\kappa_{(\Gamma,\E,D)}=\lambda_{(\Gamma,\E,D)}$ (see Remark \ref{rem:simple-root}). By Lemma \ref{thm:face-morphism}, if $(\Gamma,\E,D)\ge (\Gamma',\E',D')$ is a specialization of simple semistable root-graphs, then we have a face inclusion $\lambda_{(\Gamma',\E',D')}\ra \lambda_{(\Gamma,\E,D)}$, and every face of $\lambda_{(\Gamma,\E,D)}$ arises in this way.

\begin{Def}
We define the generalized cone complex, called the \emph{moduli space of simple semistable roots}, as
\[
\Theta^{\trop}_{g,n}:=\underset{\longrightarrow}{\lim}\;\lambda_{(\Gamma,\E,D)}=\coprod_{[\Gamma,\E,D]} \lambda^\circ_{(\Gamma,\E,D)}/\Aut(\Gamma,\E,D),
\]
where the limit is taken over the sub-category of $\mathbf{SR}_{g,n}$ and the union is taken over all equivalence classes $[\Gamma,\E,D]$ in the poset $\mathcal{SR}_{g,n}$.
\end{Def}

Recall that a quasistable pseudo-divisor is simple. 

\begin{Def}
The \emph{moduli space of quasistable roots} is the  sub-cone complex $R_{g,n}^{\trop}\subset \Theta^{\trop}_{g,n}$ obtained by restricting $\mathbf{SR}_{g,n}$ to the sub-category $\mathbf{QR}_{g,n}$, that is,
%\quad \text{ and } \quad
%R_{g,n}^\trop \subset \Theta^{\trop}_{g,n}
\[
R^{\trop}_{g,n}:=\underset{\longrightarrow}{\lim}\;\lambda_{(\Gamma,\E,D)}=\coprod_{[\Gamma,\E,D]} \lambda^\circ_{(\Gamma,\E,D)}/\Aut(\Gamma,\E,D),
\]
where the limit is taken over the category $\mathbf{QR}_{g,n}$ and the union is taken over all equivalence classes $[\Gamma,\E,D]$ in the poset $\mathcal{QR}_{g,n}$.
\end{Def}

%\begin{Def}
%Let $X$ be a tropical curve and $p_0$ be a point of $X$. A \emph{$p_0$-quasistable root} of $X$ is a $p_0$-quasistable divisor $\D$ on $X$ such that  $2\D$ is equivalent to $K_X$. 
%\end{Def}

%Notice that the combinatorial type of a $p_0$-quasistable root of a $v_0$-pseudo-quasistable-root of $\Gamma$.

\begin{Prop}\label{prop:modularR}
The generalized cone complex $R^{\trop}_{g,n}$ parametrizes equivalence classes of pairs $(X,\D)$, where $X$ is a stable $n$-pointed  tropical curve of genus $g$, and $\D$ is a $p_0$-quasistable root of $X$, and two pairs $(X,\D)$ and $(X',\D')$ are equivalent if there is an isomorphism $\iota\col X\ra X'$ such that $\iota_*(\D)=\D'$.
\end{Prop}

\begin{proof}
We can repeat verbatim the proof of Proposition \ref{prop:modularT}, replacing  the words ``polystable" with ``$p_0$-quasistable",  $\kappa_{(\Gamma,\E,D)}$ with $\lambda_{(\Gamma,\E,D)}$, and $\mathbf{PSR}_{g,n}$ with $\mathbf{QR}_{g,n}$.
\end{proof}

By \cite[Proposition 5.12]{AP1} and Proposition \ref{prop:modularR}, we have an inclusion
\[
\psi_R^\trop\col R_{g,n}^\trop\ra J_{\mu,g,n}^\trop
\]
which is a map of generalized cone complexes. Indeed, for every $v_0$-pseudo-root, the inclusion $\lambda_{(\Gamma,\E,D)}\subset \mb R^{E(\Gamma^\E)}_{\ge0}\cong\tau_{(\Gamma,\E,D)}$ 
factors the composition
\[
\lambda_{(\Gamma,\E,D)}\ra R_{g,n}^\trop\stackrel{\psi_R^\trop}{\longrightarrow}J_{\mu,g,n}^\trop. 
\]

Moreover,  
%We denote by $\rho_T^\trop\col R_{g,n}^\trop\ra T_{g,n}^\trop$ 
the composition
\[
%\rho_T^\trop\col 
R_{g,n}^\trop\stackrel{\psi_R^\trop}{\longrightarrow}J_{\mu,g,n}^\trop\stackrel{\rho_P^\trop}{\longrightarrow} P_{\mu,g,n}^\trop.
\]
factors through the inclusion $\psi_T^\trop\col T_{g,n}^\trop\hookrightarrow P_{\mu,g,n}^\trop$, because, if $\D$ is a $p_0$-quasistable root on a tropical curve $X$, then a polystable divisor $\D'$ on $X$ equivalent to $\D$ is clearly a polystable root of $X$. The induced map
\[
\rho_T^\trop\col R_{g,n}^\trop\ra T_{g,n}^\trop
\]
is a morphism of generalized cone complexes. Indeed, for every $v_0$-quasistable pseudo-divisor $(\E,D)$ on a graph $\Gamma$, we have that Proposition \ref{prop:minimal-signed-root} together with Remarks \ref{rem:simple-root} and \ref{rem:kappa-inclusion} give rise to an inclusion $\lambda_{(\Gamma,\E,D)}=\kappa_{(\Gamma,\E,D)}\subset \kappa_{(\Gamma,\pol(\E,D))}$ which factors the composition 
\[
\lambda_{(\Gamma,\E,D)}\ra R_{g,n}^\trop\stackrel{\rho_T^\trop}{\longrightarrow}T_{g,n}^\trop. 
\]

%$(X,\D')$, where $\D'$ is the unique polystable root on $X$ which is equivalent to $\D$ (see Lemma \ref{lem:equiv}). 
%polystable root of $X$ equivalent to $\D$.  

\begin{Lem}\label{lem:poly-root-spec}
Let $(\Gamma,\E,D)\ge(\Gamma',\E',D')$ be a specialization of semistable root-graphs. 
If $(\E,D)$ is polystable and $(\E',D')$ is $v_0$-quasistable, then there exists a specialization of semistable roots-graphs $(\Gamma,\E,D)\ge(\Gamma',\pol(\E',D'))$ factoring $(\Gamma,\E,D)\ge(\Gamma,\E',D')$.
\end{Lem}

\begin{proof}
We let $\iota\col \Gamma\ra \Gamma'$ and $\wh\iota\col  \Gamma^{\E}\ra {\Gamma'}^{\E'}$ be the specializations associated to the specialization  $(\Gamma,\E,D)\ge(\Gamma',\E',D')$. We set $(\wh\E,\wh D)=\pol(\E',D')$.
By Proposition \ref{prop:minimal-root}, we know that there is a specialization of semistable root-graphs $(\Gamma',\wh\E,\wh D)\ge (\Gamma',\E',D')$. We let $\wh \iota'\col {\Gamma'}^{\wh\E}\ra {\Gamma'}^{\E'}$ be the associated specialization. By Lemma \ref{lem:polyspec} we have a specialization of triples $(\Gamma,\E,D)\ge(\Gamma',\wh\E,\wh D)$ factoring $(\Gamma,\E,D)\ge (\Gamma',\E',D')$. We let $\wh j\col \Gamma^{\E}\ra {\Gamma'}^{\wh \E}$ be the induced specialization. 
By Corollary \ref{cor:graph-spec}, we need only to prove that $\wh j$ is induced by $\iota$.  
 Consider an edge $e\in \E$ and let $e_1,e_2$ be the exceptional edges  of $\Gamma^{\E}$ over $e$. We have to prove that, if $e_1$ is contracted by $\wh j$, then $e_2$ is contracted by $\wh j$ as well.  If $e_1$ was contracted by $\wh j$ and $e_2$ was not, then $\phi_{(\wh \E,\wh D)}(e_2)=0$ by Lemma \ref{lem:phizero}. Using that $\wh\iota_*(\phi_{(\E,D)})=\phi_{(\E',D')}=\wh{\iota}'_*(\phi_{(\wh \E,\wh D)})$, we would have: 
\[
|\phi_{(\E, D)}(e_2)|=|\phi_{(\E',D')}(e_2)|=|\phi_{(\wh \E,\wh D)}(e_2)|=0.
\]
%\[
%|\phi_{(\wh\E,\wh D)}(e_2)|=|\phi_{(\E',D')}(e_2)|=|\phi_{(\E,D)}(e_2)|=1,
%\]
%where the last equality holds by Theorem \ref{thm:uniquephi} (2). 
But we would get a contradiction to Lemma \ref{lem:phizero}, since $e_2$ is an exceptional edge of $\Gamma^\E$.
\end{proof}

\begin{Prop}\label{prop:refiT}
The morphism of generalized cone complexes $\rho_T^\trop\col R_{g,n}^\trop\ra T_{g,n}^\trop$ is a refinement. For every polystable pseudo-root $(\Gamma,\E,D)$, we have decompositions:
\[
\kappa^\circ_{(\Gamma,\E,D)}=\underset{(\E,D)=\pol(\E',D')}
{\coprod_{(\E',D')\in \mathbf{QR}_{g,n}}} \lambda^\circ_{(\Gamma,\E',D')}
\quad \text{ and } \quad
\kappa_{(\Gamma,\E,D)}=\underset{(\Gamma,\E,D)\ge (\Gamma',\E',D')}
{\coprod_{(\Gamma',\E',D')\in \mathbf{QR}_{g,n}}}\lambda^\circ_{(\Gamma',\E',D')}.
\]
where in the last formula $\ge$ means specialization of semistable root-graphs.
\end{Prop}

\begin{proof}
Notice the morphism of cone complexes $\rho_T^\trop$ is bijective because in an equivalence class of a polystable divisor $\D$ on a tropical curve $X$ there is exactly one $(v_0,\mu)$-quasistable representative; such a representative is a root of $X$ if and only if $\D$ is a root of $X$. Thus $\rho_T^\trop$ is a refinement.

The stated decompositions of $\kappa^\circ_{(\Gamma,\E,D)}$ and $\kappa_{(\Gamma,\E,D)}$ can be proved exactly as in Theorem \ref{thm:refiJ},  where we use the following computation for the former:
\[
\lambda^\circ_{(\Gamma,\E',D')} =\kappa^\circ_{(\Gamma,\E',D')} 
 = \kappa_{(\Gamma,\E',D')}\cap\sigma^\circ_{(\Gamma,\E',D')} \subset \kappa_{(\Gamma,\pol(\E',D'))}\cap \sigma^\circ_{(\Gamma,\pol(\E',D'))} =\kappa^\circ_{(\Gamma,\pol(\E',D'))}
\]
%\begin{align*}
%\lambda^\circ_{(\Gamma,\E',D')} & =\kappa^\circ_{(\Gamma,\E',D')} %\\
%& = \kappa_{(\Gamma,\E',D')}\cap\sigma^\circ_{(\Gamma,\E',D')}\\
% & \subset \kappa_{(\Gamma,\pol(\E',D'))}\cap \sigma^\circ_{(\Gamma,\pol(\E',D'))}\\
% & =\kappa^\circ_{(\Gamma,\pol(\E',D'))}
%\end{align*}
(recall Remark \ref{rem:kappa-inclusion}), and we use 
Proposition \ref{prop:minimal-root} and Lemma \ref{lem:poly-root-spec} for the latter.
\end{proof}

\section{Moduli of spin quasistable roots}

Given a spin semistable root-graph $(\Gamma,\E,D,s)$, we define the cones 
\begin{equation}\label{eq:spin-cones}
\lambda_{(\Gamma,\E,D,s)}:=\lambda_{(\Gamma,\E,D)}
\quad \text{ and } \quad
\kappa_{(\Gamma,\E,D,s)}:=\kappa_{(\Gamma,\E,D)}
\end{equation}
and their interiors:
\[
\lambda^\circ_{(\Gamma,\E,D,s)}:=\lambda^\circ_{(\Gamma,\E,D)}
\quad \text{ and } \quad
\kappa_{(\Gamma,\E,D,s)}^\circ=\kappa_{(\Gamma,\E,D)}^\circ.
\]
If $(\Gamma,\E,D,s)\ge (\Gamma',\E',D',s')$ is a specialization of spin polystable root-graphs, using Lemma \ref{lem:poly-pseudo-root} we get a face morphism $\kappa_{(\Gamma',\E',D',s')}\subset \kappa_{(\Gamma,\E,D,s)}$ and every face of $\kappa_{(\Gamma,\E,D,s)}$ arises in this way.

\begin{Def}
We define the generalized cone complex, called the \emph{moduli space of spin tropical curves}, as
\[
S^{\trop}_{g,n}:=\underset{\longrightarrow}{\lim}\;\kappa_{(\Gamma,\E,D,s)}=\coprod_{[\Gamma,\E,D,s]} \kappa^\circ_{(\Gamma,\E,D,s)}/\Aut(\Gamma,\E,D,s),
\]
where the limit is taken over the category $\mathbf{SP}_{g,n}$ and the union is taken over all equivalence classes $[\Gamma,\E,D,s]$ in $\mathcal{SP}_{g,n}$. 
\end{Def}

Recall that if $\D$ is a unitary semistable root of a tropical curve $X$ with model $\Gamma_X$, then its combinatorial type is a pseudo-root of $\Gamma_X$ (see Lemma \ref{lem:comb-type}). 

\begin{Prop}\label{prop:modularS}
The generalized cone complex $S_{g,n}^\trop$ parametrizes classes of triples $(X,\D,s)$, where $X$ is a stable $n$-pointed tropical curve of genus $g$, together with a polystable root $\D$ of $X$ with a sign function $s$ on the combinatorial type $(\E,D)$ of $\D$, where two triples $(X,\D,s)$ and $(X',\D',s')$ are equivalent if $(X,\D)$ and $(X',\D')$ are equivalent and $\wh\iota_*(s)=s'$, where $\wh\iota\col \Gamma_X^{\E}\ra \Gamma_{X'}^{\E'}$ is the induced isomorphism.
\end{Prop}

\begin{proof}
The argument is similar to the one used to prove Proposition \ref{prop:modularT} and is left to the reader. 
\end{proof}

\begin{Rem}
The moduli space $S^{\trop}_{g,n}$ is isomorphic (as a generalized cone complex) to the moduli space of tropical spin curves introduced in \cite{CMP1}, which parametrizes classes of triples $(X,P,s)$ where $X$ is a stable $n$-pointed tropical curve of genus $g$, together with a cyclic subgraph $P$ of $\Gamma_X$ and a sign function $s\col V(\Gamma_X/P)\ra \mathbb Z/2\mathbb Z$. The isomorphism takes the class of a triples $(X,P,s)$ to $(X,\D_P,s)$ (we use Remark \ref{rem:cyclic} and Proposition \ref{prop:equiv}). Indeed, Lemma \ref{lem:poly-pseudo-root} guarantees that the given bijection is a morphism of generalized cone complexes. 
\end{Rem}

By Propositions \ref{prop:modularT} and \ref{prop:modularS} we obtain a forgetful morphism of generalized cone complexes 
\[
\phi_S^\trop\col S_{g,n}^\trop\ra T_{g,n}^\trop
\]
taking the point parametrizing the class of a triples $(X,\D,s)$ to the point of $T_{g,n}^\trop$ parametrizing the class of the pair $(X,\D)$.
We also have the morphism of cone complexes
\[
\pi^\trop_S=\pi^\trop_T\circ\phi^\trop_S\col S_{g,n}^\trop\ra M_{g,n}^\trop
\]
taking the point parametrizing the class of $(X,\D,s)$ to the point parametrizing $X$.

\subsection{Refining the moduli space of spin tropical curves}

Recall that if $(\Gamma,\E,D,s)$ is simple (for example, $v_0$-quasistable), then $\lambda_{(\Gamma,\E,D,s)}=\kappa_{(\Gamma,\E,D,s)}$ (see Remark \ref{rem:simple-root}). By Lemma \ref{thm:face-morphism}, if $(\Gamma,\E,D,s)\ge (\Gamma',\E',D',s')$ is a specialization of spin simple semistable root-graphs, then we have a face inclusion $\lambda_{(\Gamma',\E',D',s')}\ra \lambda_{(\Gamma,\E,D,s)}$, and every face of $\lambda_{(\Gamma,\E,D,s)}$ arises in this way.

\begin{Rem}
It is worth mentioning that Lemma-Definition \ref{lem-def} and Corollary \ref{cor:facet} can be  rephrased and extended to spin semistable root-graphs and spin $v_0$-quasistable root-graphs.
\end{Rem}

\begin{Def}
We define the generalized cone complex, called the \emph{moduli space of spin quasistable roots}, as
\[
Q^{\trop}_{g,n}:=\underset{\longrightarrow}{\lim}\;\lambda_{(\Gamma,\E,D,s)}=\coprod_{[\Gamma,\E,D,s]} \lambda^\circ_{(\Gamma,\E,D,s)}/\Aut(\Gamma,\E,D,s),
\]
where the limit is taken over the category $\mathbf{SPQ}_{g,n}$ and the union is taken over all equivalence classes $[\Gamma,\E,D,s]$ in $\mathcal{SPQ}_{g,n}$.
\end{Def}

\begin{Rem}\label{rem:Q+Q-}
By Remark \ref{rem:parity}, $Q^{\trop}_{g,n}$ has two connected components $(Q^{\trop}_{g,n})^+$ and $(Q^{\trop}_{g,n})^-$ corresponding, respectively, to the posets $\mathcal{SPQ}^+_{g,n}$ and $\mathcal{SPQ}^-_{g,n}$. 
\end{Rem}

\begin{Prop}\label{prop:modularQ}
The generalized cone complex $Q_{g,n}^\trop$ parametrizes classes of triples $(X,\D,s)$, where $X$ is a stable $n$-pointed tropical curve of genus $g$, together with a $p_0$-quasistable root $\D$ of $X$ with a sign function $s$ on the combinatorial type $(\E,D)$ of $\D$, where two triples $(X,\D,s)$ and $(X',\D',s')$ are equivalent if $(X,\D)$ and $(X',\D')$ are equivalent and $\wh\iota_*(s)=s'$, where $\wh\iota\col \Gamma_{X}^{\E}\ra \Gamma_{X'}^{\E'}$ is the induced isomorphism.
\end{Prop}

\begin{proof}
The argument is similar to the one used to prove Proposition \ref{prop:modularT} and is left to the reader.
%We can repeat verbatim the proof of Proposition \ref{prop:modularT}, replacing  the words ``polystable" with ``$p_0$-quasistable",  $\kappa_{(\Gamma,\E,D)}$ with $\lambda_{(\Gamma,\E,D)}$, and $\mathbf{PSR}_{g,n}$ with $\mathbf{QR}_{g,n}$.
\end{proof}

By Propositions \ref{prop:modularR}, \ref{prop:modularS}, and \ref{prop:modularQ}, we obtain a commutative diagram of morphisms of generalized cone complexes:
\[
\SelectTips{cm}{11}
\begin{xy} <16pt,0pt>:
\xymatrix{
Q_{g,n}^\trop \ar[r]^{\phi_Q^\trop}\ar[d]^{\rho_S^\trop}& R_{g,n}^\trop  \ar[d]^{\rho_T^\trop}       \\
S_{g,n}^\trop\ar[r]^{\phi_S^\trop} &T_{g,n}^\trop
 }
\end{xy}
\]
Let us describe the new morphisms appearing in the diagram. Consider a point $q$ of $Q^\trop_{g,n}$ parametrizing the class of a triple $(X,\D,s)$. 
The map $\phi_Q^\trop$ takes $q$ to the point parametrizing the class of the pair $(X,\D)$. This is a morphism of cone complexes by the definition of the cone $\lambda_{(\Gamma,\E,D,s)}$ in equation \eqref{eq:spin-cones}.

If $\D'$ is a polystable divisor on $X$ equivalent to $\D$ and $(\E,D)$ is the combinatorial type of $\D$, then the combinatorial type of $\D'$ is $\pol(\E,D)$. The map $\phi_S^\trop$ takes $q$ to the point parametrizing the class $(X,\D',\pol(s))$. By the definition of the cone $\kappa_{(\Gamma,\E,D,s)}$ in equation \eqref{eq:spin-cones} it is easy to see that $\phi_S^\trop$ is  a morphism of generalized cone complexes.

\begin{Lem}\label{lem:signed-poly-root-spec}
Let $(\Gamma,\E,D,s)\ge(\Gamma',\E',D',s')$ be a specialization of spin semistable root-graphs.  
If $(\E,D,s)$ is polystable and $(\E',D',s')$ is $v_0$-quasistable, then there exists a specialization of spin semistable  root-graphs $(\Gamma,\E,D,s)\ge(\Gamma',\pol(\E',D',s'))$ factoring $(\Gamma,\E,D,s)\ge(\Gamma',\E',D',s')$.
\end{Lem}

\begin{proof}
 Set $(\wh\E,\wh D):=\pol(\E',D')$. 
By Lemma \ref{lem:poly-root-spec} we have a specialization of semistable root-graphs $(\Gamma,\E,D)\ge (\Gamma',\wh\E,\wh D)$ factoring the specialization of semistable root-graphs $(\Gamma,\E,D)\ge(\Gamma',\E',D')$. By Proposition \ref{prop:minimal-root} we have an identification $P_{(\E',D')}=P_{(\wh \E,\wh D)}$. Hence we get a commutative diagram of specializations:
\[
\SelectTips{cm}{11}
\begin{xy} <16pt,0pt>:
\xymatrix{
 \Gamma/P_{(\E,D)}\ar[drr]^{\wt\iota'} \ar[rr]^{\wt\iota}    & &  \ar[d]^{\cong} \Gamma'/P_{(\E',D')} \\
   & &\Gamma'/P_{(\wh\E,\wh D)}
 }
\end{xy}
\]
where $\wt\iota$ and $\wt\iota'$ are the specializations of Lemma \ref{lem:iota-s}. Hence 
$\wt\iota'_*(s)=\wt\iota_*(s)=s'=\pol(s')$,
and we are done.
\end{proof}

\begin{Thm}\label{thm:refiS}
The morphism of generalized cone complexes $\rho_S^\trop\col Q_{g,n}^\trop\ra S_{g,n}^\trop$ is a refinement. For every spin polystable pseudo-root $(\Gamma,\E,D,s)$, we have  decompositions:
\[
\kappa^\circ_{(\Gamma,\E,D,s)}=\underset{(\E,D,s)=\pol(\E',D',s')}
{\coprod_{(\E',D',s')\in \mathbf{SPQ}_{g,n}}} \lambda^\circ_{(\Gamma,\E',D',s')}
\quad\text{ and }\quad
\kappa_{(\Gamma,\E,D,s)}=\underset{(\Gamma,\E,D,s)\ge (\Gamma',\E',D',s')}
{\coprod_{(\Gamma',\E',D',s')\in \mathbf{SPQ}_{g,n}}}\lambda^\circ_{(\Gamma',\E',D',s')},
\]
where in the last formula $\ge$ means specialization of spin semistable pseudo-roots.
\end{Thm}

\begin{proof}
We know that the map $\rho_T^\trop$ is a bijection. 
Given a sign $s$ on a polystable pseudo-root $(\E,D)$ on a graph $\Gamma$ and a $v_0$-quasistable pseudo-root $(\E',D')$ on $\Gamma$ such that $\pol(\E',D')=(\E,D)$, we have that $s'=s$ is the unique sign on $(\E',D')$ such that $\pol(s')=s$ (we use that, by Proposition \ref{prop:minimal-root}, $P_{(\E,D)}=P_{(\E',D')}$). It follows that $\rho^\trop_S$ is bijective as well, hence it is a refinement.

For the stated decompositions, we can argue as in the proof of Proposition \ref{prop:refiT}, using Proposition \ref{prop:minimal-signed-root} and Lemma \ref{lem:signed-poly-root-spec} for the second decomposition.
\end{proof}

We sum up the morphisms of generalized cone complexes constructed so far in the following commutative diagram:
\[
\SelectTips{cm}{11}
\begin{xy} <16pt,0pt>:
\xymatrix{
Q_{g,n}^\trop \ar[r]^{\phi_Q^\trop}\ar[d]^{\rho_S^\trop}& R_{g,n}^\trop  \ar[d]^{\rho_T^\trop} \ar[r]^{\psi_R^\trop}     & J_{\mu,g,n}^\trop \ar[d]^{\rho_P^\trop}\ar[dr]^{\pi_J^\trop}  & \\
S_{g,n}^\trop\ar[r]^{\phi_S^\trop} &T_{g,n}^\trop\ar[r]^{\psi^\trop_T}& P_{\mu,g,n}^\trop\ar[r]^{\pi_P^\trop} & M_{g,n}^\trop
 }
\end{xy}
\]
In the diagram, the $\rho^\trop$ maps are refinements, the $\pi^\trop$ and the $\phi^\trop$ maps are forgetful maps, and the $\psi^\trop$ maps are injections.
We conclude the section with an example illustrating the decompositions given by Theorem \ref{thm:refiS}.

\begin{Exa}\label{exa:injective}
Consider a graph $\Gamma$ with two vertices connected by two edges. Set $V(\Gamma)=\{v_0,v_1\}$ and $E(\Gamma)=\{e_0,e_1\}$. Consider the polystable pseudo-root $(\E,D)$ on $\Gamma$, where $\E=E(\Gamma)$ and 
\[
D=v_0+v_1-w_0-w_1.
\]
Here $w_i\in V(\Gamma^\E)$ is the vertex in the interior of $e_i$. Then $\phi_{(\E,D)}$ has value 1 over every $e\in\ora{E}(\Gamma^\E)$ oriented as in Figure \ref{fig1} and $P_{(\E,D)}$ has no edges. Consider the sign function $s:V(\Gamma/P_{(\E,D)})\ra \mathbb Z/2\mathbb Z$ which has value 0 on every $v\in V(\Gamma/P_{(\E,D)})=V(\Gamma)$.
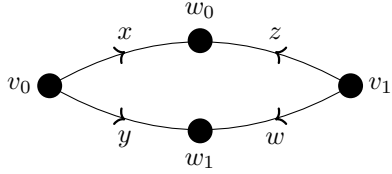
\begin{figure}[h]
    \centering
    \begin{tikzpicture}
    \draw (0,0) to [bend right] (4,0);
    \draw (0,0) to [bend left] (4,0);
    \node at (0,0) [circle, fill]{};
    \node at (-0.4,0) {$v_0$};
    \node at (4.4,0) {$v_1$};
    \node at (4,0) [circle, fill]{};
    \node at (2,0.6) [circle, fill]{};
    \node at (2,1) {$w_0$};
    \node at (2,-0.6) [circle, fill]{};
       \node at (2,-1) {$w_1$};
    \node at (1,-0.7) {$y$};
    \node at (1,0.7) {$x$};
    \node at (3,-0.7) {$w$};
    \node at (3,0.7) {$z$};
    \draw[->, line  width=0.3mm] (0.9975,0.449) to (1.0025,0.4514);
    \draw[->, line  width=0.3mm] (3.0025,0.449) to (2.9975,0.4514);
    \draw[->, line  width=0.3mm] (0.9975,-0.449) to (1.0025,-0.4514);
    \draw[->, line  width=0.3mm] (3.0025,-0.449) to (2.9975,-0.4514);
    \end{tikzpicture}
\caption{The refinement of $\Gamma$.}
\label{fig1}
\end{figure}

We have the injection $\kappa_{(\Gamma,\E,D,s)}=\kappa_{(\Gamma,\E,D)}=\T_\E(\lambda_{(\Gamma,\E,D)})\subset \mathbb R^4/{\mc L_\E}$, where $\mc L_\E$  is the one-dimensional linear space spanned by the vector with coordinates $(1,1,-1,-1)$. By Lemma \ref{lem:poly-pseudo-root}, we have  identifications $\mathbb R^2_{\geq0}\stackrel{\delta_\E}{\cong}\Delta_{\E}\stackrel{\mathcal T_\E}{\cong}\kappa_{(\Gamma,\E,D,s)}\subset \mathbb R^2$. We get a commutative diagram:
\[
\SelectTips{cm}{11}
\begin{xy} <16pt,0pt>:
\xymatrix{
\R_{\geq0}^2  \ar[d]^{\id} \ar[rr]^{\delta_\E}     &&   \Delta_{\E} \ar[d]\ar[rr]^{\T_\E} && \kappa_{(\Gamma,\E,D,s)}\ar[d]^{\alpha}  \\
\mathbb R^2_{\geq0}\ar[rr]^{\delta_\E} && \mathbb R^4_{\geq0}\ar[rr]^{\T_\E} &&\sigma_{(\Gamma,\E,D)}
 }
\end{xy}
\]
where $\alpha$ is the inclusion in equation \eqref{eq:kappa-sigma}.
We let $(u,v)$ be the coordinates of $\mathbb R^2$ and $(x,y,z,w)$ the coordinates of $\mathbb R^4$.
The $v_0$-quasistable spin pseudo-roots $(\E',D',s')$ of $\Gamma$ such that  $\pol(\E',D',s')=(\E,D,s)$ are: 
\[
\begin{array}{ll}
(\E_1,D_1,s)=(\{e_1\},v_0-w_0,s),
\;\;
(\E_2,D_2,s)=(\{e_2\},v_0-w_1,s),
\;\;
(\E_3,D_3,s)=(\emptyset,v_0-v_1,s).
\end{array} 
\]
In Figure \ref{fig:quasi} we draw a picture of the underlying pseudo-roots.
\begin{figure}[h!]
    \centering
    \begin{tikzpicture}
    \draw (0,0) to [bend right] (2,0);
    \draw (0,0) to [bend left] (2,0);
    \draw[fill] (0, 0) node {} (0, 0) circle (0.1);
    \node at (-0.4,0) {$v_0$};
    \node at (2.4,0) {$v_1$};
    \draw[fill] (2, 0) node {} (2, 0) circle (0.1);
    \draw[fill] (1, 0.3) node {} (1, 0.3) circle (0.1);
    \node at (1,0.6) {$w_0$};

    \draw[->, line  width=0.3mm] (0.99875,-0.2875) to (1.00125,-0.2875);
    \node at (1,-0.6) {$y$};
    \node at (0.5,0.4) {$x$};
  
    \node at (1.5,0.4) {$z$};
    \draw[->, line  width=0.3mm] (0.49875,0.2245) to (0.50125,0.2257);
    \draw[->, line  width=0.3mm] (1.50125,0.2245) to (1.49875,0.2257);
    
     \draw (4,0) to [bend right] (6,0);
    \draw (4,0) to [bend left] (6,0);
    \draw[fill] (4, 0) node {} (4, 0) circle (0.1);
    \node at (3.6,0) {$v_0$};
    \node at (6.4,0) {$v_1$};
    \draw[fill] (6, 0) node {} (6, 0) circle (0.1);
    \node at (5,0.6) {$x$};
   \draw[fill] (5, -0.3) node {} (5, -0.3) circle (0.1);
       \node at (5,-0.6) {$w_1$};
    \node at (4.5,-0.4) {$y$};
    \node at (5.5,-0.4) {$w$};
       \draw[->, line  width=0.3mm] (4.99875,0.2875) to (5.00125,0.2875);
   \draw[->, line  width=0.3mm] (4.49875,-0.2245) to (4.50125,-0.2257);
    \draw[->, line  width=0.3mm] (5.50125,-0.2245) to (5.49875,-0.2257);

       \draw (8,0) to [bend right] (10,0);
    \draw (8,0) to [bend left] (10,0);
    \draw[fill] (8, 0) node {} (8, 0) circle (0.1);
    \node at (7.6,0) {$v_0$};
    \node at (10.4,0) {$v_1$};
    \draw[fill] (10, 0) node {} (10, 0) circle (0.1);
    \node at (9,0.6) {$x$};
       \node at (9,-0.6) {$y$};
       \draw[->, line  width=0.3mm] (8.99875,0.2875) to (9.00125,0.2875);
       \draw[->, line  width=0.3mm] (8.99875,-0.2875) to (9.00125,-0.2875);
    \end{tikzpicture}
\caption{The pseudo-roots $(\E_1,D_1)$,$(\E_2,D_2)$, $(\E_3,D_3)$.}
\label{fig:quasi}
\end{figure}
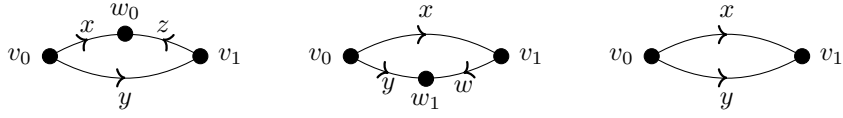

Given a point $(x,y,z,0)\in\lambda^\circ_{(\E_1,D_1,s)}$, we have $x-z-y=0$. Hence the points $(x,y,z,0)\in\lambda^\circ_{(\E,D,s)}$ and $(\frac{x+z}{2},\frac{y}{2},\frac{x+z}{2},\frac{y}{2})$ get identified with the same point $[\frac{x+z}{2},\frac{y}{2},\frac{x+z}{2},\frac{y}{2}]\in \kappa_{(\Gamma,\E,D)}\subset\R^4/\mathcal{L}_\E$. Under the identification $\kappa_{(\Gamma,\E,D)}\cong \mathbb R^2_{\ge0}$, the point $[\frac{x+z}{2},\frac{y}{2},\frac{x+z}{2},\frac{y}{2}]$ corresponds to the point $(u,v)=(\frac{x+z}{2},\frac{y}{2})\in\R^2_{\geq0}$. Thus we get  $u>v$, since $y=x-z$. It is easy to see that for every $u>v>0$, the point $(u,v)$ corresponds to the point $(u+v,2v,u-v,0)\in\lambda^\circ_{(\E_1,D_1,s)}$. A similar argument can be done for $(\E_2,D_2,s)$ and $(\E_3,D_3,s)$. Hence 
\[
\kappa^\circ_1:=\kappa^\circ_{(\Gamma,\E_1,D_1,s)}\cong\{(u,v)\in\R^2_{>0}\;|\;u>v\}
\]
\[
\kappa^\circ_2:=\kappa^\circ_{(\Gamma,\E_2,D_2,s)}\cong\{(u,v)\in\R^2_{>0}\;|\;u<v\}
\]
\[
\kappa^\circ_3:=\kappa^\circ_{(\Gamma,\E_3,D_3,s)}\cong\{(u,v)\in\R^2_{>0}\;|\;u=v\}.
\]
The decomposition in Theorem \ref{thm:refiS} of $\kappa^\circ_{(\E,D,s)}$ is $\kappa^\circ_{(\Gamma,\E,D,s)}
=\kappa^\circ_1\sqcup\kappa^\circ_2\sqcup \kappa^\circ_3$.

Next, for $i\in\{1,2\}$, consider the specialization $\iota_i\col\Gamma\to\Gamma_i:=\Gamma/\{e_{3-i}\}$.
 There is just one $v_0$-quasistable spin pseudo-root on $\Gamma_i$, which is  $(\E_i,D'_i,s'_i)=(\{e_i\},v_0-w_{i-1}, \iota_{i*}(s))$. Consider the specialization $\iota_3\col\Gamma\to\Gamma_3:=\Gamma/\{e_1,e_2\}$. There is just one spin pseudo-divisor on $\Gamma_3$, which is  $(\E_3,D'_3,s'_3)=(\emptyset,0, \iota_{3*}(s))$. We draw a picture of these psedo-roots in Figure \ref{fig:cont}.
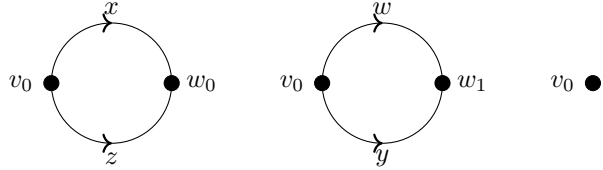
\begin{figure}[h]
    \centering
    \begin{tikzpicture}
    
    \draw (0.8,0) circle (0.8);
    %\draw (0,0) to [bend left] (2,0);
    \draw[fill] (0, 0) node {} (0, 0) circle (0.1);
    \node at (-0.4,0) {$v_0$};
    
    \draw[fill] (1.6, 0) node {} (1.6, 0) circle (0.1);
   
    \node at (2,0) {$w_0$};
  
    \draw[->, line  width=0.3mm] (0.79875,-0.8) to (0.80125,-0.8);
        \draw[->, line  width=0.3mm] (0.79875,0.8) to (0.80125,0.8);
   
    \node at (0.8,1) {$x$};
   
    \node at (0.8,-1) {$z$};
  
   \draw (4.4,0) circle (0.8);
   
    \draw[fill] (3.6, 0) node {} (3.6, 0) circle (0.1);
    \node at (3.2,0) {$v_0$};
  
    \draw[fill] (5.2, 0) node {} (5.2, 0) circle (0.1);
  
    \node at (5.6,0) {$w_1$};
 
    \draw[->, line  width=0.3mm] (4.39875,-0.8) to (4.40125,-0.8);
        \draw[->, line  width=0.3mm] (4.39875,0.8) to (4.40125,0.8);
   
    \node at (4.4,1) {$w$};

    \node at (4.4,-1) {$y$};
    
       \draw[fill] (7.2, 0) node {} (7.2, 0) circle (0.1);
     \node at (6.8,0) {$v_0$};

    \end{tikzpicture}
\caption{The pseudo-roots on $\Gamma_1,\Gamma_2,\Gamma_3$.}
\label{fig:cont}
\end{figure}

\noindent
The decomposition of $\kappa_{(\Gamma,\E,D,s)}\cong\R^2_{\geq0}$ is as prescribed by Theorem \ref{thm:refiS}, since:
\begin{align*}
    \kappa'^\circ_1:= \kappa^\circ_{(\Gamma_1,\E'_1, D'_1,s'_1)}&\cong\{(u,0);u\in\R_{>0}\}\\
    \kappa'^\circ_2:=\kappa^\circ_{(\Gamma_2,\E'_2,D'_2,s'_2)}&\cong\{(0,v);v\in\R_{>0}\}\\
    \kappa'^\circ_3:=\kappa^\circ_{(\Gamma_3,\E'_3,D'_3,s'_3)}&\cong\{(0,0)\}.\\
\end{align*}
The construction is illustrated in Figure \ref{fig2}, where we draw a section of the relevant cones.
\begin{figure}[h!]
    \centering
    \begin{tikzpicture}
    \draw (-6,2)--(-2,2);
       \draw[fill] (-6, 2) node {} (-6, 2) circle (0.1);
     %\node at (-6.2,2) [circle, fill]{}; %{$u$};
       \draw[fill] (-2, 2) node {} (-2, 2) circle (0.1);
    %\node at (-1.8,2)[circle, fill]{};% {$v$};
      \draw[fill] (-4, 2) node {} (-4, 2) circle (0.1);
     %\node at (-4,2) [circle, fill]{};
     \node at (-0.8,0.8) {$\sigma_{(\Gamma,\E,D)}$};
     \node at (-4,1) {$\kappa_{(\Gamma,\E,D,s)}$};
     \node at (-1.9,2.4) {$\kappa'^\circ_2$};
     \node at (-6,2.4) {$\kappa'^\circ_1$};
     \node at (-5,2.2) {$\kappa^\circ_1$};
     \node at (-3,2.2) {$\kappa^\circ_2$};
     \node at (-4,2.4) {$\kappa^\circ_3$};
    \node at (-0.9,2) {$\stackrel{\alpha}{\longrightarrow}$};
    \draw (0,0) rectangle(4,4);
    \node at (-0.2,0) {$x$};
    \node at (4.2,0) {$w$};
    \node at (-0.2,4) {$z$};
    \node at (4.2,4) {$y$};
    \draw (0,0) -- (4,4);
    \draw[line width=1pt] (0,2) -- (4,2);
    \node at (1.5,3) {$\tau_{(\Gamma,\E_1,D_1)}$};
    \node at (2.5,1) {$\tau_{(\Gamma,\E_2,D_2)}$};
     \node at (3,2.7) {\rotatebox{45}{$\tau_{(\Gamma,\E_3,D_3)}$}};
    \end{tikzpicture}
\caption{The decomposition of $k_{(\Gamma,\E,D,s)}$.}
\label{fig2}
\end{figure}
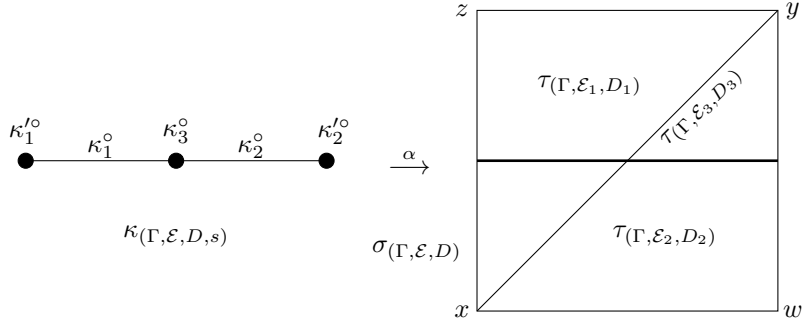
\end{Exa}

\section{Moduli of quasistable spin curves}

\subsection{Esteves universal Jacobian}\label{sec:Esteves-univ}

A \emph{(pointed) curve} will always mean a reduced connected projective scheme of dimension $1$ defined over an algebraically closed field $k$, whose singularities are nodes, together with a tuple of smooth points. We consider curves of genus $g\ge2$. The dual graph of  a curve $X$ is denoted $\Gamma_X$; we 
let $p_e$ be the node of $X$ corresponding to an edge $e\in E(\Gamma)$.

A \emph{chain of rational components} is a curve which is a union of smooth connected rational components $E_1,\dots E_n$ such that $E_i\cap E_j=\emptyset$ for $|i-j|>1$ and $\#(E_i\cap E_{i+1})=1$ for $i=1,\dots,n-1$.  A \emph{semistable modification} $X$ of a curve $C$ is a curve $X$ which is the union of the normalization $C_\Delta$ of $C$ at a subset $\Delta\subset C$ of nodes and chains $E_p=\cup_{1\le i\le n_p} E_{p,i}$ of rational components (called \emph{exceptional components}), for $p\in \Delta$, such that $\#(E_p\cap C_\Delta)=2$ and $\#(E_{p,1}\cap C_\Delta)=\#(E_{p,n_p}\cap C_\Delta)=1$. A \emph{quasistable modification} is a semistable modification such that $n_p=1$ for every $p\in \Delta$. There is a map $X\ra C$ contracting the exceptional components.

Recall that a pointed curve $X$ is \emph{stable} 
%(respectively, \emph{semistable}) 
if so is its dual graph. 
%We call \emph{exceptional} a smooth rational component of $X$ whose corresponding vertex $v$ in the dual graph $\Gamma$ of $X$ satisfies $\deg_\Gamma(v)=2$ and $L_\Gamma(v)=0$.
 We say that $X$
 is \emph{quasistable} if it is a quasistable modification of a stable curve.  %if it is semistable and two different exceptional components do not intersect.  
 Given a curve $X$, we denote by $\st(X)$ the \emph{stable model} of $X$, obtained by contracting all exceptional components of $X$. Similarly, we can define the stable model $\st(\mc X)\ra S$ for a family $\mc X\ra S$ of nodal curves.

%Given a polarization $\mu$ on $\Gamma$ say that $L$ is \emph{$\mu$-semistable}, \emph{$(v_0,\mu)$-quasistable}, \emph{$\mu$-polystable} if so is its combinatorial type. The definition naturally extends to families. In particular, we have the $(\sigma,\mu)$-quasistability for a 

Let $I$ be a torsion-free rank-$1$ sheaf on a nodal curve $X$. Consider a smooth point $p_0\in X$ and a polarization $\mu$ on the dual graph $\Gamma$ of $X$. For every subcurve $Y$ of $X$, we set $\mu(Y)=\sum_{v\in V_Y}\mu(v)$, where $V_Y\subset V(\Gamma)$ is the set of vertices corresponding to the components of $X$ contained in $Y$, and $\delta_Y:=\delta_{V_Y}=|Y\cap \ol{X\setminus Y}|$. Let $I_Y$ be the restriction of $I$ to $Y$ modulo torsion. We define:
\[
\beta_Y(I)=\deg I_Y-\mu(Y)+\frac{1}{2}\delta_Y.
\]

 We say that $I$ is \emph{$(p_0,\mu)$-quasistable} if $\beta_Y(I)\ge0$ for every proper subcurve $Y$ where the inequality is strict if $p_0\in Y$. The definition extends to families: given a family $\pi\col\mc X\ra S$ of nodal curves with a section $\sigma$ of $\pi$ through its smooth locus and a polarization $\mu$ on $\mc X/S$, we say that a sheaf $\mc I$ on $\mc X/S$ is \emph{$(\sigma,\mu)$-quasistable} if the restriction of $\mc I$ to every fiber $\pi^{-1}(s)$ is $(\sigma(s),\mu|_{\pi^{-1}(s)})$-quasistable for every geometric point $s\in S$.

The construction also extends to the universal setting, giving rise to a universal compactified Jacobian $\overline{\mc J}_{\mu,g,n}$. This is a   proper
and separated Deligne-Mumford stack over $\overline{\mc M}_{g,n}$ 
whose sections over a $k$-scheme $S$ are:
\[
\ol{\J}_{\mu,g,n}(S)=\frac{\left\{(\pi\col \mc X\ra S,\mc I)\big |\begin{array}{l} \pi\text{ is a family of stable $n$-pointed genus-$g$  curves}\\ \text{with sections $(\sigma_i)$ and $\mc I$ is a $(\sigma_1,\mu)$-quasistable}\\ \text{torsion-free rank-1 sheaf on $\mc X/S$}\end{array}\right\}}{\sim}
\] 
where $(\pi\col \mc X\ra S,\mc I)\sim(\pi'\col \mc X'\ra S,\mc I')$ if there is an $S$-isomorphism $f\col \mc X\to \mc X'$ and an invertible sheaf $\mc L$ on $S$, such that $\sigma'_i=f\circ\sigma_i$ are the sections of $\pi'$ and $\mc I\cong f^*\mc I'\otimes\pi^*\mc L$. The morphisms between two objects $(\pi\col \mc X\ra S,\mc I)$ and $(\pi'\col \mc X'\ra S',\mc I')$ are given by Cartesian diagrams:
\[
\SelectTips{cm}{11}
\begin{xy} <16pt,0pt>:
\xymatrix{
\mc X \ar[r]^{f}\ar[d]^{\pi}& \mc X'\ar[d]^{\pi'} \\
S\ar[r]^{h} & S'
 }
\end{xy}
\]
such that there is an isomorphism $f^*\mc I'\cong \mc I\otimes \pi^*(\mathcal L)$ for some invertible sheaf $\mathcal L$ on $S$, and $f\circ\sigma_i=\sigma'_i\circ h$.
We refer to \cite[Theorems A and B]{M15} and \cite[Corollary 4.4 and Remark 4.6]{KP} for more details. 

We now give another useful presentation of the stack $\ol{\J}_{\mu,g,n}$. Consider the category $\ol{\mc P}_{\mu,g,n}$ fibered in grupoids over the category of $k$-schemes, whose sections over a $k$-scheme $S$ are:
\[
\ol{\mc P}_{\mu,g,n}(S)=\frac{\left\{(\pi\col \mc X\ra S,\mc L) \big | \begin{array}{l} \pi\text{ is a family of $n$-pointed genus-$g$ quasistable}\\ \text{curves with sections $(\sigma_i)$ and $\mc L$ is a $(\sigma_1,\mu)$-} \\ \text{quasistable invertible sheaf on $\mc X/S$ with}\\ \text{degree $-1$ on exceptional components}\end{array}\right\}}{\sim}
\] 
where the equivalence relation and the morphisms between objects are the ones introduced above. 
%$(\pi\col \mc X\ra S,\mc L)\sim(\pi'\col \mc X'\ra S,\mc L')$ if there is an $S$-isomorphism $f\col \mc X\to \mc X'$, a $\st(\mc X)$-automorphism $g\col\mc X\ra \mc X$, and an invertible sheaf $\mc M$ on $S$, such that $\sigma'_i=f\circ\sigma_i$ are the sections of $\pi'$ and $g^*\mc L\cong f^*\mc L'\otimes\pi^*\mc M$.
%The morphisms between two objects $(\pi\col \mc X\ra S,\mc L)$ and $(\pi'\col \mc X'\ra S',\mc L')$ are given by Cartesian diagrams:
%\[
%\SelectTips{cm}{11}
%\begin{xy} <16pt,0pt>:
%\xymatrix{
%\mc X \ar[d]^{\psi}\ar[rr]^{f}& & \mc X'\ar[d]^{\psi'}\\
%\st(\mc X) \ar[rr]^{\st(f)}\ar[d]^{\st(\pi)}&&  \st(\mc X')\ar[d]^{\st(\pi')}\\
% S\ar[rr]^{h}&  & S 
% }
%\end{xy}
%\]
%such that $\st(\pi)\circ\psi=\pi$,   $\st(\pi')\circ\psi'=\pi'$, $f\circ\sigma_i=\sigma'_i\circ h$, and there is an isomorphism $f^*\mc L'\cong \mc L$.

\begin{Prop}\label{prop:isoJP}
There is an isomorphism $\ol{\mc J}_{\mu,g,n}\cong \ol{\mc P}_{\mu,g,n}$. %taking $(\pi\col \mc X\ra S,\sigma,\mc L)$ to $(\st(\pi)\col \st(\mc X)\ra S,\psi\circ\sigma,\psi_*(\mc L))$,.
\end{Prop}

\begin{proof}
%The formation of an admissible sheaf commutes with base change by \cite[Theorem 3.1 (2)]{EP}.
Given $(\pi\col\mc X\ra S,\mc L)\in \ol{\mc P}_{\mu,g,n}(S)$, we take $(\st(\pi)\col \st(\mc X)\ra S,\psi_*(\mc L))$, where 
%$\st(\mc X)\ra S$ is the stable model of $(\pi\col \mc X\ra S,\sigma)$ with 
 $\psi\col \mc X\ra \st(\mc X)$ is the contraction morphism and the sections of $\st(\pi)$ are given by composing the sections of $\pi$ with $\psi$.
%\in\ol{\mc J}_{\mu,g,n}(S)$. 
Conversely, given $(\pi\col \mc X\ra S,\mc I)\in\ol{\mc J}_{\mu,g,n}(S)$, 
we define $\mathbb P_{\mc X}(\mc I^\vee):=\Proj(\Sym(\mc I^\vee))$, with its natural structure morphism $\psi\col\mathbb P_{\mc X}(\mc I^\vee)\ra \mc X$, and we take  
$(\pi\circ\psi\col\mathbb P_{\mc X}(\mc I^\vee)\ra S,\mc O_{\mathbb P_{\mc X}(\mc I^\vee)}(-1))$, where the sections of $\pi\circ\psi$ are given by pulling-back the sections of $\pi$ via $\psi$. The fact that we have established an isomorphism between $\ol{\mc J}_{\mu,g,n}$ and $\ol{\mc P}_{\mu,g,n}$ follows from \cite[Theorems 3.1, 4.1]{EP} and a small modification of \cite[Propositions 5.4, 5.5]{EP}. 
\end{proof}

From now on, we will identify $\ol{\mc J}_{\mu,g,n}$ and $\ol{\mc P}_{\mu,g,n}$. For every pair $(X,L)$, where $X$ is a pointed quasistable curve of genus $g$ and $L$ is a $(\sigma,\mu)$-quasistable invertible sheaf on $X$ with degree $-1$ on exceptional components, we let $[X,L]$ be the geometric point of $\ol{\mc J}_{\mu,g,n}$ parametrizing $(X,L)$.

Let $X\ra C$ be a quasistable modification of a curve $C$ and let $L$ be an invertible sheaf on $X$ with degree $-1$ on exceptional components. 
%Consider an invertible sheaf $L$ on a $n$-pointed nodal curve $X$. 
We define the \emph{combinatorial type} $(\Gamma_C,\E_L, D_L)$, where  $(\E_L,D_L)$ is the pseudo-divisor on the dual graph $\Gamma_C$ of $C$ such that $\E_L\subset E(\Gamma_C)$ is the set of nodes of $C$ to which the exceptional components of $X$ contract and $D_L$ is given by the multidegree of $L$. One can see that $L$ is $(p_0,\mu)$-quasistable if and only if its combinatorial type is $(v_0,\mu)$-quasistable, where $v_0\in V(\Gamma)$ is the vertex corresponding to the component of $X$ containing $p_0$. We let:
\begin{equation}\label{eq:c-type}
\tau\col \ol{\mathcal J}_{\mu,g,n} \ra \mc{QD}_{\mu,g,n}
\end{equation}
  be the function taking $[X,L]$ to $(\Gamma_{\st(X)},\E_L,D_L)$.

%\begin{Rem}\label{rem:J-strat}
 By \cite[Proposition 6.4]{AP1}, the moduli space $\ol{\mc J}_{\mu,g,n}$ admits a 
 stratification indexed by $\mc{QD}_{\mu,g,n}$ whose strata 
 $\mc J_{(\Gamma,\E,D)}$ are given by the fibers of $\tau$, i.e., 
 for every $(\Gamma,\E,D)\in \mc{QD}_{\mu,g,n}$, 
 \[
 \mc J_{(\Gamma,\E,D)}=\{[X,L]\big| (\Gamma_{\st(X)},\E_L,D_L)\cong (\Gamma,\E,D)\}.
 \]
%\end{Rem}

%In what follows we will consider the universal compactified Jacobian $\ol{\J}_{\mu,g,n}$ where $\mu$ is the canonical polarization (recall Example \ref{exa:pol}).
%where $(\pi_1,\sigma_1,\I_1)\sim(\pi_2,\sigma_2,\I_2)$ if there exist a $S$-isomorphism $f\col \X_1\to \X_2$ and an invertible sheaf $\L$ on $S$, such that $\sigma_2=f\circ\sigma_1$ and $\I_1\cong f^*\I_2\otimes\pi_1^*\L$. We refer to \cite[Theorems A and B]{M15} and \cite[Corollary 4.4 and Remark 4.6]{KP} for more details on the stack $\ol{\J}_{\mu,g}$. In what follows we will consider the universal compactified Jacobian $\ol{\J}_{\mu,g}$ where $\mu$ is the canonical polarization (recall Example \ref{exa:pol}).

Next, consider a quasistable modification $X\ra C$ of a curve $C$ and an invertible sheaf $L$ on $X$ with degree $-1$ on exceptional components. Assume that the combinatorial type of $(X,L)$ is a semistable root-graph. Let
 $\{P_u\}_{u\in V(\Gamma^\E/P_{(\E,D)})}$ be the set of $\phi_{(\E,D)}$-graphs of $\Gamma^\E$ (see Definition \ref{def:phi-subgraph}). 

\begin{Def}
The \emph{$\phi_{(\E,D)}$-curves} of $X$ are the subcurves $Z_u$ of $X$ whose dual graph corresponds to $P_u$, for $u\in V(\Gamma^\E/P_{(\E,D)})$. Notice that any exceptional component of $X$ is a  $\phi_{(\E,D)}$-curve.
\end{Def}

%\begin{Rem}
%If $X$ is a quasistable curve and $L$ is a invertible sheaf on $X$ with degree $-1$ on exceptional components such that the combinatorial type of $(X,L)$ is a semistable root-graph, then 
%\end{Rem}

Notice that the two different $\phi_{(\E,D)}$-subcurves intersect each other at a (possibly empty) subset of nodes of $X$, and we have $X=\cup_{u\in V(\Gamma^\E/P_{(\E,D)})} Z_u$. 

%For every $u\in V(\Gamma^\E/P_{(\E,D)})$, we define the invertible sheaf $T_u$ on $Z_u$ as:
%\begin{equation}\label{eq:Tw}
%T_u:=\mc O_{Z_u}\Big(\underset{t(e)\in V(P_u)}{\sum_{e\in \ora{E}(P_u,P_u^c)}} \phi_{(\E,D)}(e)p_e\Big).
%\end{equation}

\begin{Not}\label{not:Nv}
Given $u\in V(\Gamma^\E/P_{(\E,D)})$ we denote by $\mc N_u$ the subset of vertices $u'\in V(\Gamma^\E/P_{(\E,D)})$ such that $\phi_{(\E,D)}(e)=-1$ for some (and hence all) edges $e\in \ora{E}(P_{u'},P_u)$ such that $t(e)\in V(P_u)$.
\end{Not}

 Hence for every $u\in V(\Gamma^\E/P_{(\E,D)})$ we get
\begin{equation}\label{eq:Vw}
\mc O_{Z_u}\Big(\underset{t(e)\in V(P_u)}{\sum_{e\in \ora{E}(P_u^c,P_u)}} \phi_{(\E,D)}(e)p_e\underset{t(e)\in V(P_u)}{-\sum_{e\in \ora{E}(P_u^c,P_u)}}p_e\Big)=\mc O_{Z_u}\Big(-2\sum_{u'\in \mc N_u}\;\sum_{e\in E(P_{u'},P_u)} p_e\Big).
\end{equation}

\begin{Def}\label{def:semi-spin-curve}
A \emph{semistable spin curve} is a pair $(X,L)$, where $X$ is a quasistable curve,
%modification of a  curve $C$, 
and $L$ is an invertible sheaf on $X$ such that: 
\begin{itemize}
\item[(1)] the restriction of $L$ to any exceptional component has degree $-1$.
    \item[(2)] the combinatorial type of $(X,L)$ is a semistable root-graph. 
    \item[(3)] 
     for every $\phi_{(\E_L,D_L)}$-curve $Z_u$ of $X$ we have an isomorphism $\wt L_u^{\otimes 2}\cong\omega_{Z_u}$, where 
\begin{equation}\label{eq:isow}
   \wt L_u:=L|_{Z_u}\otimes \mc O_{Z_u}\Big(-\sum_{u'\in \mc N_u}\sum_{e\in E(P_{u'},P_u)} p_e\Big).
     \end{equation}
\end{itemize}
%Notice that the definition is independent of the curve $C$.
We say that $(X,L)$ is \emph{polystable} and \emph{$p_0$-quasistable}, respectively, if so is its combinatorial type. 
\end{Def}

\begin{Def}\label{def:comb-type}
Let $(X,L)$ be a semistable spin curve and $(\Gamma,\E,D)$ be its combinatorial type. We define the sign function $s_L\col V(\Gamma/P_{(\E,D)})\ra \mathbb Z/2\mathbb Z$ taking $u\in V(\Gamma/P_{(\E,D)})$ to $s_L(u)$, the parity of $h^0(Z_u,\wt L_u)$. The semistable spin root-graph $(\Gamma,\E,D,s_L)$ is called the \emph{spin combinatorial type} of $(X,L)$.
 We say that $(X,L)$ is \emph{even} or \emph{odd} if so is $(\Gamma,\E,D,s_L)$.
\end{Def}

Notice that $s_L$ is indeed a sign function: if the weight $w(u)$ is zero for some $u\in V(\Gamma/P_{(\E,D)})$, then by Corollary \ref{cor:tree} we have that $Z_u\cong \mathbb P^1$ and $\wt L_u\cong \mc O_{\mathbb P^1}(-1)$, hence $s_L(u)=0$. 

\subsection{Spin curves and their moduli}

The moduli space of spin curves $\overline{\mc S}_{g,n}$ compactifies the moduli space $\mc S_{g,n}$ of isomorphism classes of pairs $(X,L)$, where $X$ is a $n$-pointed smooth curve of genus $g$ and $L$ is an invertible sheaf on $X$ such that $L^{\otimes 2}\cong\omega_X$. 
The moduli space $\ol{\mc S}_{g,n}$ is finite of degree $2^{2g}$ over $\ol{\mc M}_{g,n}$, hence $\ol{\mc S}_{g,n}$ has dimension $3g-3+n$.

We use a different, but equivalent, modular description of the moduli space $\ol{\mc S}_{g,n}$, which better suits our purposes. The closed points of $\ol{\mc S}_{g,n}$ are isomorphism classes of pairs $(X,L)$, where $X$ is a $n$-pointed quasistable curve and $L$ is an invertible sheaf
%with set of exceptional components $\Ex(X)$ and $L$ is an invertible sheaf on $X$ 
with degree $-1$ on exceptional components of $X$ and $L|_Y^{\otimes 2}(-2\sum_{E\in \Ex(X)}\sum_{p\in Y\cap E}p)\cong\omega_Y$, where $Y$ is the complement of the union of the exceptional components of $X$. Notice that $Y$ could be non-conneted. The space $\ol{\mc S}_{g,n}$ has two connected components, $\ol{\mc S}^+_{g,n}$ and $\ol{\mc S}^-_{g,n}$, corresponding to even and odd spin curves.

 \begin{Rem}\label{rem:S-strat}
 If $(X,L)$ is a pair parametrized by a point of $\ol{\mc S}_{g,n}$, then $(X,L)$ is a polystable spin curve. In fact, by \cite[Theorem 3.4.3]{CMP1},
 the moduli space $\ol{\mc S}_{g,n}$ admits a stratification indexed by the poset $\mc{SP}_{g,n}$. Recall that $\mc{SP}_{g,n}$ is isomorphic to the poset of isomorphism classes of triples $(\Gamma,P,s)$, for a graph $\Gamma$, a cyclic subgraph $P$ of $\Gamma$, and a sign function $s\col V(\Gamma/P)\ra \mathbb Z/2\mathbb Z$ (see Remark \ref{rem:spin-graphs-cat}).  The strata are defined, for every triple $(\Gamma,P,s)$, as: 
\[
 \mc S_{(\Gamma,P,s)}=\{[X,L]\in \ol{\mc S}_{g,n}\big| (\Gamma_{\st(X)},\E_L,D_L,s_L)\cong (\Gamma,\E_P,D_P,s)\}.
 \]
 \end{Rem}

%\subsection{Quasistable spin curves and their moduli}

The moduli space $\mc S_{g,n}$ of theta characteristics on $n$-pointed smooth curves of genus $g$ naturally sits inside $\ol{\mc J}_{\mu,g,n}$. 
We define $\ol{\mc Q}_{g,n}$ as  the closure of $\mc S_{g,n}$ inside $\ol{\mc J}_{\mu,g,n}$. By definition, we have inclusions over $\ol{\mc M}_{g,n}$:
\[
\ol{\mc S}_{g,n}\supset\mc S_{g,n}\subset\ol{\mc Q}_{g,n}\subset \ol{\mc J}_{\mu,g,n}. 
\]
In particular, the dimension of $\ol{\mc Q}_{g,n}$ is $3g-3+n$. Moreover, $\ol{\mc Q}_{g,n}$ is a
 Deligne-Mumford separated stack, since so is $\ol{\mc J}_{\mu,g,n}$.

 Our next goal is to give a modular interpretation of the space $\ol{\mc Q}_{g,n}$ using root graphs. Before, we need to prove an easy result.

\begin{Lem}\label{lem:refi-flow} 
A flow $\phi$ on a graph is acyclic if and only if there are positive integers $h_e$, for each $e\in E(\Gamma)$ such that
\[
\sum_{e\in \gamma}\phi(e)h_e=0
\]
for every oriented cycle $\gamma$ of $\Gamma$. In particular, if $\wh \Gamma$ is the refinement of $\Gamma$ obtained by inserting $h_e-1$ vertices over the edge $e$, then
%preserving the loops of $\Gamma$ 
$\Div(\phi)$ is principal as a divisor on $\wh\Gamma$.
\end{Lem}

\begin{proof}
Consider the vector space $\mathbb R^{E(\Gamma)}$ with coordinates $\{x_e\}_{e\in E(\Gamma)}$. We define the linear subspace $\Lambda$ of $\mathbb R^{E(\Gamma)}$ given by the equations:
\begin{equation}\label{eq:gammae}
\sum_{e\in \gamma}\phi(e)x_e=0,
\end{equation}
 where $\gamma$ runs over the set of oriented cycles of $\Gamma$. 
 There is a vector $(h_e)_{e\in E(\Gamma)}\in\Lambda\cap \mathbb Z_{>0}^{E(\Gamma)}$ if and only if $\Div(\phi)$ is principal on the refinement of $\Gamma$ obtained by inserting exactly $h_e-1$ vertices in the interior of each edge $e\in E(\Gamma)$. 
 %In this case, no new vertex is inserted in a loop of $\Gamma$. 
 It is clear that $\Lambda\cap \mathbb Z_{>0}^{E(\Gamma)}$ is empty if $\phi$ is not acyclic. 

If $\phi$ is acyclic and nontrivial, then $|V(\Gamma)|\ge2$. The space $\Lambda$ is given by intersecting at most $b_1(\Gamma)$ independent linear hyperspaces in a space of dimension $|E(\Gamma)|$, hence $\dim\Lambda\ge|V(\Gamma)|-1\ge1$. 
%, we deduce that $\Lambda\cap \mathbb Z^{E(\Gamma)}$ is not empty. 
Since $\phi$ is acyclic, for every oriented cycle $\gamma$ for which \eqref{eq:gammae} is not a trivial equation, there are two nonempty subset 
%we can write each oriented cycle $\gamma$ as the disjoin union of two nonempty sets 
$\gamma^+,\gamma^-\subset \gamma$ such that $\phi(e)\ne0$ for $e\in \gamma^+\cup \gamma^-$, so equation \eqref{eq:gammae} reads as:
\[
\sum_{e\in \gamma^+}|\phi(e)|x_e=\sum_{e\in \gamma^-}|\phi(e)|x_e.
\]
We deduce that $\Lambda\cap \mathbb Z_{>0}^{E(\Gamma)}$ is not empty.
\end{proof}

\begin{Def}
A \emph{one-parameter deformation} of a curve $X$ is a family of curves $\pi\col \mc X\ra B=\Spec(k[[t]])$ whose special fiber is $X$. A \emph{twister} of $\mc X$ is an invertible sheaf of $\mc X$ associated to an effective Cartier divisor of $\mc X$ supported on the special fiber. A \emph{smoothing of $X$} is a one-parameter deformation with smooth generic fiber. 
\end{Def}

 \begin{Lem}\label{lem:twister}
Let $\pi\col \mc X\ra B$ be a one-parameter deformation of a curve $X$ with generic fiber $X_\eta$. Let $\Gamma$ and $\Gamma_\eta$ be the dual graphs of $X$ and $X_\eta$. For each edge $e\in E(\Gamma)\setminus E(\Gamma_{\eta})$ we assume that the local \'etale equation of $\mathcal{X}$ at the node of $X$ corresponding to $e$ is $xy=t^{h_e}$, where $x,y$ are local étale coordinates and $h_e$ is a positive integer. Let $T$ be a twister of $\mc X$ and $D_T$ be the divisor on $\Gamma$ given by the multidegree of $T|_X$. Then $D_T=\Div(\phi)$ for some acyclic flow $\phi$ on $\Gamma$ such that $\phi(e)=0$ for every $e\in \ora{E}(\Gamma_\eta)\subset \ora{E}(\Gamma)$. 
Moreover, if $\widehat{\Gamma}$ is the refinement of $\Gamma$ obtained by inserting $h_e-1$ vertices on each edge $e\in E(\Gamma)\setminus E(\Gamma_{\eta})$, then $D_T$, seen as a divisor on $\widehat{\Gamma}$, is principal. 
 \end{Lem}
 
 \begin{proof}
 Denote by $(X_v)_{v\in V(\Gamma)}$ the components of $X$. Since $T$ is supported on $X$, we can write $T=\mc O_{\mc X}(\sum_{v\in V(\Gamma)} \alpha_v X_v)$, for some $\alpha_v\in \mathbb Z$. 
 
  We claim that $\alpha_{s(e)}=\alpha_{t(e)}$ for every $e\in \ora{E}(\Gamma_\eta)\subset \ora{E}(\Gamma)$. 
  The completion of the local ring of $\mathcal{X}$ at the node $p_e$ corresponding to $e$ is
  \[
  \widehat{\mathcal O}_{\mathcal X,p_e}\cong R:=\frac{k[[x,y,t]]}{(xy)},
  \]
  where $x,y$ are local \'etale coordinates.
  Assume that $f$ is the local equation of $T$. In particular, $f$ is not a divisor of $0$ in $R$. Since $T$ is supported on the special fiber, we have that $(t)\subset \sqrt{(f)}$ and hence $t^m\in (f)$, which means that $t^m=f\overline{f}$ for some $\overline{f}\in R$, so there exists an integer $\alpha$ such that $f=t^\alpha u$ with $u$ invertible. This proves that $\alpha_{s(e)}=\alpha_{t(e)}=\alpha$.
  
%  Indeed, let $\sigma\col B\ra \mc X$ be the section through the singular locus of $\mc X$ over $B$ corresponding to $e$. Consider $\mc Y$ be a connected component of the normalization of $\mc X$ along $\sigma(B)$ and let $\mc Y\stackrel{\varphi}{\ra}\mc X\stackrel{\pi}{\ra}  B$ the induced family, with special fiber $Y$. Then $\varphi^*T$ is a twister of $\mc Y$, so it is trivial along the smooth locus of $\mc Y$ over $B$, implying that $\alpha_{s(e)}=\alpha_{t(e)}$.

% Let $e\in \ora{E}(\Gamma_\eta)\subset \ora{E}(\Gamma)$ and $\sigma\col B\ra \mc X$ be the section through the singular locus of $\mc X$ over $B$ corresponding to $e$. Let $\mc Y$ be a connected component of the normalization of $\mc X$ along $\sigma(B)$ and let $\mc Y\stackrel{\varphi}{\ra}\mc X\stackrel{\pi}{\ra}  B$ the induced family, with special fiber $Y$. 
 %There is a section $\tau\col B\ra\mc Y$ such that $\varphi\circ\tau(b)=\sigma(b)$. 
 %Then $\varphi^*T$ is a twister of $\mc Y$ and hence it is trivial along the smooth locus of $\mc Y$ over $B$. We get $a_{\sigma(0),i}=0$, $\forall\;i$, and hence $\phi(e)=0$.
 
 Next, for a node $p_e$ of $X$, with $e\in \ora{E}(\Gamma)\setminus \ora{E}(\Gamma_{\eta})$, we let $xy=t^{h_e}$ be the local equation of $\mc X$ at $p_e$, for some positive integer $h_e$, where $x,y$ are local \'etale coordinates. Since $T$ is Cartier, there is an integer $a_{e}$ such that   $\alpha_{t(e)}-\alpha_{s(e)}=h_e a_{e}$. Notice $a_e+a_{\overline{e}}=0$ for $e$ and $\overline{e}$ with the same associated edge.  
    This means that 
    \begin{equation}
    \label{eq:twisterXv}
    T|_{X_v}\cong \mc O_{X_v}(\sum_{\substack{e\in \ora{E}(\Gamma)\setminus \ora{E}(\Gamma_\eta)\\  s(e)=v}} a_{e} p_e).
    \end{equation}
      Define the flow $\phi$ on $\Gamma$ such that $\phi(e)=-a_{e}$ for every $e\in \ora{E}(\Gamma)\setminus \ora{E}(\Gamma_\eta)$ and $\phi(e)=0$ for every $e\in \ora{E}(\Gamma_\eta)$. Then we have $D_T=\Div(\phi)$.
  
We show that $\phi$ is acyclic. 
  If $e_1,\dots,e_n$ was a no-null $\phi$-cycle, with $v_i:=s(e_i)=t(e_{i-1})$, then 
 \begin{equation}
 \label{eq:phiprin}
 0=\alpha_{v_1}-\alpha_{v_n}+\sum_{1\le i\le n-1}(\alpha_{v_{i+1}}-\alpha_{v_i})=\sum_{1\le i\le n} h_{e_i} a_{e_i}=-\sum_{1\le i\le n} h_{e_i} \phi(e_i)
 \end{equation}
 where $\phi(e_i)\geq 0$ and $\phi(e_i)$ is not always zero, 
 which is a contradiction. Hence $\phi$ is acyclic. 
 
 The last statement follows from Equation \eqref{eq:phiprin} and Lemma \ref{lem:refi-flow}.
 %To finish the proof, we note that Equation \eqref{eq:phiprin} implies that $\phi$ induces a rational function $f$ on the tropical curve $(\Gamma,\ell)$ where $\ell(e)=h_e$ for every $e\in E(\Gamma)\setminus E(\Gamma_{\eta})$ and $\ell(e)=1$ for $e\in E(\Gamma_{\eta})$. Since $(\widehat{\Gamma},\widehat{\ell})=(\Gamma,\ell)$, where $\widehat{\ell}(e)=1$ for every $e\in E(\widehat{\Gamma})$ we have that $\Div(\phi)=\Div(f)$ is principal in $\widehat{\Gamma}$.
 %If $\mc X$ is smooth at the nodes of $X$ corresponding to the edges in $E(\Gamma)\setminus E(\Gamma_\eta)$, then $D_T=\sum_{v\in V(\Gamma)} \alpha_v\Div(v)$, hence $D_T$ is principal.
 \end{proof}

\begin{Thm}\label{thm:modularQgn}
%Let $[X,L]$ be a geometric point of $\ol{\mc P}_{\mu,g,n}$, where 
Let $X$ be a quasistable $n$-pointed curve and $L$ be an invertible sheaf on $X$. The following conditions are equivalent
\begin{itemize}
    \item[(1)] $[X,L]$ is a geometric point of $\ol{\mc Q}_{g,n}$.
     \item[(2)] 
     $(X,L)$ is a $p_0$-quasistable spin curve.
\end{itemize}  
\end{Thm}

\begin{proof}
Assume that $[X,L]$ is a geometric point of $\ol{\mc Q}_{g,n}$.
%Let $(X,\sigma,L)$ be a $p_0$-quasistable spin curve and
Set $(\Gamma,\E,D):=\tau([X,L])\in \mc{QD}_{\mu,g,n}$. There is a smoothing  $\pi\col\mc X\ra B$ of $X$  %where $B$ is the spectrum of a discrete valuation ring with generic point $\eta$, with smooth generic fiber $\mc X_\eta$ and $X$ as a special fiber, 
and an invertible sheaf $\mc L$ on $\mc X/B$ such that $\mc L|_X\cong L$ and $\mc L|_{X_\eta}^{\otimes 2}\cong \omega_{X_{\eta}}$, where $X_\eta$ is the generic fiber of $\pi$.
%Let $\varphi\col\mc Y\ra \mc X$ be a desingularization of $\mc X$. Let $Y$ and $Y_\eta$ be the special and generic fibers of $\mc Y\ra \mc X\ra B$. Let $\mc M := \varphi^*(\mc L)$ and $M:=\mc M|_Y$. 
%The dual graph $\wh\Gamma$ of $Y$ is a refinement of $\Gamma^\E$. We denote by $\wh D$ the multidegree of $M$, which is the divisor $D$ seen as a divisor on $\wh \Gamma$ through the natural injective map $\Div(\Gamma^\E)\ra \Div(\wh \Gamma)$.  Now, $\mc Y$ and $\mc X$ are isomorphic over $\eta$, hence $\mc M|_{\mc Y_\eta}^{\otimes 2}\cong \omega_{\mc Y_\eta}$,
There is a twister $T$ of $\mc X$ such that
\begin{equation}\label{eq:twister}
\mc L^{\otimes 2}\otimes T\cong \omega_{\mc X}.
\end{equation}
By Lemma \ref{lem:twister}, the multidegree of the invertible sheaf $T$ is a principal divisor $D_T$ on $\Gamma^{\mc E}$ and $D_T=\Div(\phi)$, for an acyclic flow $\phi$ on $\wh \Gamma$. By \eqref{eq:twister}, we deduce the relation $2 D+\Div(\phi)=K_{ \Gamma^{\mc E}}$. 
%Now, $\wh D$ is $v_0$-quasistable on $\wh \Gamma$, because so is $D$ on $\Gamma^\E$. This implies that if $v_1,\dots,v_n$ are the vertices of $\wh\Gamma$ in the interior of an edge  $e\in E(\Gamma)$, then there is at most one $v_i$ such that $\wh D(v_i)\ne 0$, and in this case we have $e\in \E$ and $v_i\in V(\Gamma^\E)$, with $\wh D(v_i)=-1$. On the other hand, if $\wh D(v_i)=0$, then $\phi(e)=\phi(e')$ where $e,e'$ are the exceptional edges such that $t(e)=s(e')=v_i$.
%Thus the flow $\phi$ is a well-defined flow on $\Gamma^\E$ and we have that $2D+\Div(\phi)=K_{\Gamma^\E}$. 
Hence $(\Gamma,\E,D)$ is a root-graph.  By Remark \ref{rem:PED} and equations \eqref{eq:twisterXv}, 
\eqref{eq:Vw}, \eqref{eq:twister} we have an isomorphism as in equation \eqref{eq:isow}, for every $\phi_{(\E,D)}$-curve $Z_v$ of $X$.

 Next assume that $(X,L)$ is a $v_0$-quasistable spin curve. Let $(\Gamma,\E,D)$ be the combinatorial type of $(X,L)$. By Lemma \ref{lem:refi-flow} there is a refinement $\wh\Gamma$ of $\Gamma$ 
%preserving the loops of $\Gamma$
such that $\Div(\phi_{(\E,D)})$ is principal as a divisor on $\wh \Gamma$. Let $Y$ be the nodal curve having $\wh \Gamma$ as dual graph such that $\st(Y)=X$, and $\pi_Y\col Y\ra X$ be the induced contraction map.  By equations %\eqref{eq:Tw}, 
\eqref{eq:Vw} and \eqref{eq:isow}, the multidegree of $\pi_Y^*(\omega_X\otimes L^{-2})$ is equal to $\Div(\phi_{(\E,D)})$ as a divisor on $\wh \Gamma$. Therefore, by \cite[Corollary 6.9]{EM} there is a smoothing  $\pi\col \mc Y\ra B$ of $Y$ and a  twister $T$ of $\mc Y$ such that $\pi_Y^*(\omega_X\otimes L^{-2})\cong T|_Y$. Hence by \cite[Remark 3.0.6]{CCC} we get an invertible sheaf $\mc M$ on $\mc Y/B$ such that $\mc M|_Y\cong \pi_Y^*(L)$ endowed with an isomorphism $\mc M^{\otimes 2}\otimes T\cong \omega_{\mc Y/B}$. If $\pi\col \mc Y\ra \mc X$ is the contraction of the exceptional components of $Y$ over $X$, then $(\pi_*\mc M)^{\otimes 2}\otimes \pi_*(T)\cong \omega_{\mc X/B}$ and $\pi_*\mc M|_X\cong L$. Hence $[X,L]$ is  a geometric point of $\ol{\mc Q}_{g,n}$ as $(X,L)$ is the limit of a family of theta characteristics on smooth curves.
\end{proof}

The previous result justifies the following definition.

\begin{Def}
We call $\ol{\mc Q}_{g,n}$ the \emph{moduli space of quasistable spin curves}. 
%, whose objects are called \emph{$p_0$-quasistable spin curves}.
\end{Def}

\subsection{The stratification of $\ol{\mc Q}_{g,n}$}

Next, consider a $v_0$-quasistable spin root-graph $(\Gamma,\E,D,s)$. We let $\mc Q_{(\Gamma,\E,D,s)}$ be the locus in $\ol{\mc Q}_{g,n}$ defined as:
\[
\mc Q_{(\Gamma,\E,D,s)}=\{[X,L]\big| 
(\Gamma_{\st(X)},\E_L,D_L,s_L)\cong(\Gamma,\E,D,s)\}.
\]
We have an inclusion $\mc Q_{(\Gamma,\E,D,s)}\subset \mc J_{(\Gamma,\E,D)}$. Recall the definition of $\widetilde{\mc M}_\Gamma$ in equation \eqref{eq:MGamma-tilde}. Let $\mc C_\Gamma\ra \mc M_\Gamma$ be the universal family over $\mc M_\Gamma$ and set $\widetilde{\mc C}_\Gamma:=\mc C_\Gamma\times_{\mc M_\Gamma}\widetilde{\mc M}_{\Gamma}$. Let $\widetilde{\mc C}_{\Gamma,\E}$ be the partial normalization of $\widetilde{\mc C}_{\Gamma}$ at the nodes corresponding to the edges of $\E$.
 Let $\widetilde{\mc J}_{(\Gamma,\E,D)}$ be the Jacobian over $\widetilde{\mc M}_\Gamma$ parametrizing invertible sheaves of multidegree $D_\E$ over the family $\widetilde{\mc C}_{\Gamma,\E}\to \widetilde{\mc M}_\Gamma$.
Define $\widetilde{\mc Q}_{(\Gamma,\E,D,s)}\subset \widetilde{\mc J}_{(\Gamma,\E,D)}$ by the following fiber diagram
\[
\begin{tikzcd}
\widetilde{\mc Q}_{(\Gamma,\E,D,s)}\ar[r]\ar[d] & \widetilde{\mc J}_{(\Gamma,\E,D)}\ar[d]\\
\mc Q_{(\Gamma,\E,D,s)}\ar[r] & \mc J_{(\Gamma,\E,D)}
\end{tikzcd}
\]

Recall the definition of $b_1(\E,D)$ in equation \eqref{eq:bED}.

\begin{Lem}\label{lem:Porteus}
Every irreducible component of $\widetilde{\mc Q}_{(\Gamma,\E,D,s)}$ has dimension at least $\dim\mc M_\Gamma+b_1(\E,D)$.
\end{Lem}
\begin{proof}
We have that $\widetilde{\mc C}_{\Gamma,\E}$ can be written as a union $\bigcup_{u\in V(\Gamma_\E/P_{(\E,D)})}{\mc Z}_u$, where $\mc Z_u\ra \widetilde{\mc M}_\Gamma$ is a family whose fibers have dual graph isomorphic to the connected component $P_u$ of $\ol{P}_{(\E,D)}$ corresponding to $u$. Let $\sigma_{\mc N_u}$ be the set of sections of $\mc Z_u\ra \widetilde{\mc M}_\Gamma$ whose corresponding legs of $P_u$ is obtained from an edge $e\in E(\Gamma^\E)$ connecting $u$ and a vertex $u'\in \mc N_u$ (see Notation \ref{not:Nv}).
 Let $\mc L$ be the universal invertible sheaf over the family 
\[
\widetilde{\mc  C}_{\Gamma,\E}\times_{\widetilde{\mc M}_{\Gamma}}\widetilde{\mc J}_{(\Gamma,\E,D)}\ra \widetilde{\mc J}_{(\Gamma,\E,D)}. 
\]
Set $\widetilde{\mc Z}_u:=\mc{Z}_u\times_{\widetilde{\mc M}_{\Gamma}}\widetilde{\mc J}_{(\Gamma,\E,D)}$. Let $\mc{L}_u$ be the restriction of $\mc{L}$ to the family $f_u\col \widetilde{\mc Z}_u\ra \widetilde{\mc J}_{(\Gamma,\E,D)}$, and  $\widetilde{\sigma}_{\mc N_u}$ be the set of sections of $f_u$ obtained as the pull-back of the sections in $\sigma_{\mc N_u}$. We define 
\[
\widetilde{\mc L}_u:=\mc{L}_u\otimes \mc O_{\widetilde{\mc Z}_u}(-\sum_{\sigma_i\in \sigma_{\widetilde{\mc N}_u}}\sigma_i(\widetilde{\mc J}_{(\Gamma,\E,D)})).
\]
Finally, let $\mc Q_u$ be the locus of points $p$ of $\widetilde{\mc J}_{(\Gamma,\E,D)}$ where $\widetilde{\mc L}_u^{\otimes 2}|_{f_u^{-1}(p)}\cong \omega_{f_u^{-1}(p)}$.

%\[
%\left(\mc{L}_v\otimes \mc O_{\widetilde{\mc Z}_v}(-\sum_{\sigma_i\in \sigma_{\widetilde{\mc N}_v}}\sigma_i(\widetilde{\mc J}_{(\Gamma,\E,D)}))\right)^2\cong\omega_{\mc{Z}_v}.
%\]

We claim that the codimension of $\mc Q_u$ in $\widetilde{\mc J}_{(\Gamma,\E,D)}$ is at most the genus $g(P_u)$ of $P_u$. Indeed, 
%let $f_v\col \mc{Z}_v\to \widetilde{\mc J}_{\Gamma,\E,D}$ and 
define $\mc{J}_u$ as the relative Jacobian of $f_u$ parametrizing invertible sheaves of multidegree equal to the multidegree of $\omega_{f_u}$. 
%the canonical sheaf of $\mc{Z}_v$. 
We have that $\mc{J}_u\to \widetilde{\mc J}_{(\Gamma,\E,D)}$ has two sections induced by $\omega_{f_u}$ and by $\widetilde{\mc L}_u^{\otimes 2}$, and $\mc{Q}_u$ is set-theoretically the locus where these sections agree. Since the relative dimension of $\mc{J}_u\to \widetilde{\mc J}_{(\Gamma,\E,D)}$ is $g(P_u)$ and $\mc{J}_u$ is smooth, each irreducible component of $\mc Q_u$ has codimension at most $g(P_u)$. 

Let $\mc Q_{u,s(u)}$ be the locus in $\mc Q_u$ where $h^0(X,\widetilde{\mc L}_v|_X)\equiv s(u)\mod (2)$ for every fiber $X$ of $f_u$. By deformation invariance of the parity of $h^0(X,\widetilde{\mc L}_u|_X)$, each irreducible component of $\mc Q_{u,s(u)}$ is an irreducible component of $\mc Q_u$. Hence   
%we have that 
each irreducible component of  $\mc{Q}_{u,s(u)}$ 
has codimension at most $g(P_u)$ in $\widetilde{\mc J}_{(\Gamma,\E,D)}$. 
  By Theorem \ref{thm:modularQgn} we have $\wt{\mc Q}_{(\Gamma,\E,D,s)}=\bigcap_{u\in V(\Gamma_\E/P_{\E,D})} \mc{Q}_{u,s(u)}$  (set-theoretically). Then, by the claim, $\wt{\mc Q}_{(\Gamma,\E,D,s)}$ has codimension at most $\sum_{u\in V(\Gamma_\E/P_{(\E,D)})}g(P_u)$ in $\widetilde{\mc J}_{(\Gamma,\E,D)}$. Thus:
  \begin{align*}
      \dim \widetilde{\mc Q}_{(\Gamma,\E,D,s)}&\geq \dim \widetilde{\mc J}_{(\Gamma,\E,D)}-\sum_{u\in V(\Gamma_\E/P_{(\E,D)})}g(P_u)\\
                                &=\dim(\mc M_{\Gamma})+g(\Gamma_\E)-\sum_{u\in V(\Gamma_\E/P_{(\E,D)})}g(P_u)\\
                                &=\dim(\mc M_{\Gamma})+g(\Gamma_\E/P_{(\E,D)})-\sum_{u\in V(\Gamma_\E/P_{(\E,D)})}g(P_u)\\
                                &=\dim(\mc M_{\Gamma})+b_1(\Gamma_\E/P_{(\E,D)})\\                                &=\dim(\mc M_{\Gamma})+b_1(\E,D),
  \end{align*}
which concludes the proof.
\end{proof}

\begin{Prop}\label{prop:irred}
Let $(\Gamma,\E,D,s)$ be a $v_0$-quasistable spin root-graph. Given $u\in V(\Gamma/P_{(\E,D)})$, let $P_u$ be the associated connected components of $\ol{P}
_{(\E,D)}$ and set $s_u=s(u)$. Then there are morphisms
\[
\widetilde{\mc Q}_{(\Gamma,\E,D,s)}\stackrel{\rho}{\ra}\prod_{u\in V(\Gamma/P_{(\E,D)})}\mc S_{(P_u,P_u,s_u)}
\stackrel{\theta}{\ra}\mc S_{(\Gamma,P_{(\E,D)},s)}.
\]
where $\theta$ is \'etale and the fibers of $\rho$ are isomorphic to $(k^*)^{b_1(\E,D)}$.
In particular
%$\widetilde{\mc Q}_{(\Gamma,\E,D,s)}$ and 
$\mc Q_{(\Gamma,\E,D,s)}$ is irreducible of codimension $\#E(\Gamma)-b_1(\E,D)$ in $\ol{\mc Q}_{g,n}$.
\end{Prop}

\begin{proof}
The map $\theta$ is the natural \'etale morphism defined in the proof of \cite[Proposition 3.6.3]{CMP1}, where we use that $P_{\pol(\Gamma,\E,D)}=P_{(\E,D)}$ by Proposition \ref{prop:minimal-root} and $\pol(s)=s$ by Definition \ref{def:spin-pol}. 

Next, consider a $k$-scheme $S$. A section $(\pi\col\mc X\ra S, \mc L)\in\widetilde{\mc Q}_{(\Gamma,\E,D,s)}(S)$ gives rise to a family $\pi_u\col\mc Z_u\ra S$ of curves whose fiber have dual graph isomorphic to $P_u$. 
%Hence we have $\mc X=\cup_{v\in V(\Gamma^\E/P_{(\E,D)}} \mc Z_v$. 
Let $\sigma_{\mc N_u}$ be the set of sections of $\pi_u$ whose corresponding legs of $P_u$ is obtained from an edge $e\in E(\Gamma^\E)$ connecting $u$ and a vertex $u'\in \mc N_u$ (see Notation \ref{not:Nv}). The map $\rho$ in the statement takes $(\mc X\ra S, \sigma, \mc L)$ to  
\[
\rho(\mc X\ra S, \mc L)=\prod_{u\in V(\Gamma/P_{(\E,D)})}(\pi_u\col\mc Z_u\ra S,\mc L|_{\mc Z_u}\otimes \mc O_{\mc Z_u}\Big(-\sum_{\sigma_i\in\sigma_{\mc N_u}}\sigma_i(S))\Big).
\]
Indeed, by Theorem \ref{thm:modularQgn}, we have that $\rho(\mc X\ra S, \mc L)\in \prod_{u\in V(\Gamma/P_{(\E,D)})}\mc S_{(P_u,P_u,s_u)}$. 

The fiber of $\rho$ passing through a geometric point $[X, L]$ of $\widetilde{\mc Q}_{(\Gamma,\E,D,s)}$  is given by all possible gluing of $L|_{\mc Z_u}$ to an invertible sheaf on $X$, where gluings at exceptional components of $X$ do not count for the dimension of the fiber. Hence the fibers of $\rho$ are as claimed in the statement.

By \cite[Theorem 4.2.4]{CMP1} we have that $\mc S_{(P_u,P_u,s_u)}$ is irreducible. Hence, from what we have shown so far, we get that  there is a unique irreducible component of $\widetilde{\mc Q
}_{(\Gamma,\E,D,s)}$ that dominates $\prod_{v\in V(\Gamma/P_{(\E,D)})}\mc S_{(P_u,P_u,s_u)}$. The last product has dimension equal to $\dim\mc M_\Gamma$, since it is  \'etale over $\mc S_{(\Gamma,P_{(\E,D),s})}$,
%The dimension of the last product is $\dim\mc M_\Gamma$, since it is \'etale over $\mc S_{(\Gamma,P_{(\E,D),s})}$, 
which is \'etale over $\mc M_\Gamma$. 
Therefore, by Lemma \ref{lem:Porteus}, we see that  $\widetilde{\mc Q}_{(\Gamma,\E,D,s)}$ %and $\widetilde{\mc Q}_{(\Gamma,\E,D,s)}$ are
is irreducible and its dimension is $\dim \mc M_\Gamma+b_1(\E,D)$. Since $\dim \ol{\mc Q}_{g,n}=3g-3+n$, we can conclude just observing that there is a finite (in fact \'etale) map $\widetilde{\mc Q}_{(\Gamma,\E,D,s)}\ra 
\mc Q_{(\Gamma,\E,D,s)}$. 
\end{proof}

\begin{Prop}\label{prop:pol-spin}
Let $\pi\col\mc X \to B$ be a one-parameter deformation of a curve $X'$, with generic fiber $X$, and $\mc L$ an invertible sheaf  on $\mc X/B$ such that $(X,\mc L|_{X})$ and $(X',\mc L|_{X'})$ are, respectively, semistable spin curves of spin combinatorial types $(\Gamma,\E,D,s)$ and $(\Gamma',\E',D',s')$. Then there is a specialization $(\Gamma',\E',D',s')\ge (\Gamma,\E,D,s)$ of spin semistable root-graphs.
\end{Prop}

\begin{proof}
We have a specialization of triples  $\iota\col(\Gamma',\E',D')\ge (\Gamma,\E,D)$ (see \cite[Proposition 4.4.2]{CC}). Adding a suitable number of sections, we can make $\pi$ a family of $n$-pointed stable curves. Hence 
we can assume $\E=\E'=\emptyset$.

\smallskip

\emph{Case 1}. Assume that $(\Gamma,\emptyset,D)$ is stable. We have $\mc L|_X^2=\omega_X$, and hence $\omega_{\mc X}\otimes \mc L^{-2}\cong T$ for some twister $T$ of $\mc X$. Let $D_T$ be the divisor on $\Gamma'$ given by the multidegree of $T|_{X'}$. By Lemma \ref{lem:twister} we have $D_T=\Div(\phi)$, for some acyclic flow $\phi$ on $\Gamma'$ such that $\phi(e)=0$ for every $e\in \ora{E}(\Gamma)\subset \ora{E}(\Gamma')$. Since $2D'+\Div(\phi)=K_{\Gamma'}$, we deduce that $\phi=\phi_{(\emptyset,D')}$ (we use Theorem \ref{thm:uniquephi}).  We have $\iota_*(\phi_{(\emptyset,D')})=\iota_*(\phi)=0=\phi_{(\emptyset,D)}$, which means that $(\Gamma',\emptyset,D')\ge (\Gamma,\emptyset,D)$ is a specialization of root-graphs.

Set $(\wt\E',\wt D')=\pol(\emptyset,D')$. To conclude, we show the existence of a specialization of polystable root-graphs    $(\Gamma',\wt\E',\wt D',\pol(s'))\ge (\Gamma,\emptyset,D,s)$.  
Indeed, if $\wt{\iota}\col\Gamma'/P_{(\wt \E',\wt D')}\ra\Gamma/P_{(\emptyset,D)}$ is the  associated specialization, then 
$\wt\iota_*(s')=\wt \iota_*(\pol(s'))=s$ 
 (we use the identification  $\Gamma'/P_{(\emptyset,D')}\cong\Gamma'/P_{(\wt \E',\wt D')}$).

For every edge $e\in E(\Gamma')\setminus E(\Gamma)$ let $k_e$ be the positive integer such that the local \'etale equation of $\mc X$ at the node of $X'$ corresponding to $e$ is $xy=t^{k_e}$, where $x,y$ are local \'etale coordinates. Let $\wh{\Gamma}'$ be the refinement of $\Gamma'$ obtained by inserting $k_e-1$ vertices in the interior of each $e\in E(\Gamma')\setminus E(\Gamma)$. Notice that $\phi_{(\emptyset,D')}$ induces a natural acyclic flow on $\wt{\Gamma}'$, which we continue to denote by $\phi_{(\emptyset,D')}$.
 By Lemma \ref{lem:twister}, we have that
 %there is  refinement $\wh\Gamma'$ of $\Gamma''$ (and hence of $\Gamma'$) such that 
 $\Div(\phi_{(\emptyset,D')})$ is a principal divisor on $\wh\Gamma'$, i.e., $\Div(\phi_{(\emptyset,D')})=\sum_{v\in V(\wh{\Gamma}')} a_v \Div(v)$ for $a_v\in \mathbb Z$. 
 Notice that $\phi(\wh e)=a_{s(\wh e)}-a_{t(\wh e)}$, for every $\wh e\in \ora{E}(\wh{\Gamma}')$ and $k_e\phi(e)=a_{s(e)}-a_{t(e)}$ for every $e\in \ora{E}(\Gamma')$. %we have $E(V_i, V_j)\ne\emptyset$ if and only if $|i-j|\le 1$. 
 Recall that $|\phi(e)|\le 1$ for every $e\in \ora{E}(\wh{\Gamma}')$, by Theorem \ref{thm:uniquephi}.
 Let $\wt{\Gamma}'$ be the refinement of $\Gamma'$ obtained from $\wh{\Gamma}'$ by inserting one vertex in the interior of any edge $e\in E(\wh{\Gamma}')$. For every $e\in \ora{E}(\Gamma')$ we let $f_0,\dots,f_{2k_e-1}\in\ora{E}(\wt{\Gamma}')$ be the oriented edges over $e$, and 
 we set %$v_{e,0}:=s(f_0)=s(e)$, $v_{e,2k_e}:=t(f_{2k_2-1})=t(e)$, and 
 $v_{e,i}:=t(f_{i-1})=s(f_i)$, for every $i=1,\dots,2k_e-1$. If  $\phi(e)=-1$, we set
 \[
 a_{e,i}:=
 \begin{cases}
 a_{s(e)}+i & \text{ for }i=1,\dots,k_e-1 \\
a_{t(e)} & \text{ for } i=k_e,\dots,2k_e-1.
  \end{cases}
\] 
 If $\phi(e)=0$, we set $a_{e,i}:=a_{s(e)}=a_{t(e)}$. For each vertex $\wt{v}\in V(\wt{\Gamma}')$ we define $b_{v}$ as 
 \[
 b_{\wt{v}}:=
 \begin{cases}
  a_{e,i}& \text{ if  $\wt v=v_{e,i}$ for some $e\in \ora{E}(\Gamma')$ and  $i=1,\ldots, 2k_e-1$,}\\
  a_{v}&\text{ if $\wt{v}=v\in V(\Gamma')\subset V(\wt{\Gamma}')$}.
 \end{cases}
 \]
 We note that $b_{\wt{v}}$ is well defined. In Figure \ref{fig:a} we illustrate the construction for $\phi(e)=-1$. 
 \begin{figure}[h]
    \centering
    \begin{tikzpicture}
    \draw (-7.2,0) -- (7.2,0);
    \draw[fill] (-7.2, 0) node {} (-7.2, 0) circle (0.1);
    \draw[fill] (-3.6, 0) node {} (-3.6, 0) circle (0.1);
    \draw[fill] (0, 0) node {} (0, 0) circle (0.1);
    \draw[fill] (3.6, 0) node {} (3.6, 0) circle (0.1);
     \draw[fill] (7.2, 0) node {} (7.2, 0) circle (0.1);
    
       \draw[->, line  width=0.3mm] (-5.4,0) to (-5.39,0);
       \draw[->, line  width=0.3mm] (-1.8,0) to (-1.79,0);
        \draw[->, line  width=0.3mm] (1.8,0) to (1.81,0);
       \draw[->, line  width=0.3mm] (5.4,0) to (5.41,0);
        \node at (-5.4,0.3) {$1$};
         \node at (-1.8,0.3) {$1$};
          \node at (1.8,0.3) {$1$};
         \node at (5.4,0.3) {$1$};
         \node at (-7.2,-0.3) {$a_{s(e)}$};
           \node at (7.2,-0.3) {$a_{t(e)}$};
        \node at (-3.6,-0.3) {$a_{s(e)}+1$};
       \node at (3.6,-0.3) {$a_{s(e)}+3$};
         \node at (0,-0.3) {$a_{s(e)}+2$};
         
         \draw (-7.2,-1.5) -- (7.2,-1.5);
    \draw[fill] (-7.2, -1.5) node {} (-7.2, -1.5) circle (0.1);
    \draw[fill] (-3.6, -1.5) node {} (-3.6, -1.5) circle (0.1);
    \draw[fill] (0, -1.5) node {} (0, -1.5) circle (0.1);
    \draw[fill] (3.6, -1.5) node {} (3.6, -1.5) circle (0.1);
     \draw[fill] (7.2, -1.5) node {} (7.2, -1.5) circle (0.1);
     \draw[fill] (-5.4, -1.5) node {} (-5.4, -1.5) circle (0.08);
    \draw[fill] (-1.8, -1.5) node {} (-1.8, -1.5) circle (0.08);
    \draw[fill] (1.8, -1.5) node {} (1.8, -1.5) circle (0.08);
     \draw[fill] (5.4, -1.5) node {} (5.4, -1.5) circle (0.08);
      \draw[->, line  width=0.3mm] (-6.3,-1.5) to (-6.29,-1.5);
       \draw[->, line  width=0.3mm] (-4.5,-1.5) to (-4.49,-1.5);
        \draw[->, line  width=0.3mm] (-2.7,-1.5) to (-2.69,-1.5);
       \draw[->, line  width=0.3mm] (-0.9,-1.5) to (-0.89,-1.5);
       \draw[->, line  width=0.3mm] (0.9,-1.5) to (0.91,-1.5);
        \draw[->, line  width=0.3mm] (2.7,-1.5) to (2.71,-1.5);
       \draw[->, line  width=0.3mm] (4.5,-1.5) to (4.51,-1.5);
        \draw[->, line  width=0.3mm] (6.3,-1.5) to (6.31,-1.5);
        \node at (-7.2,-1.8) {$a_{s(e)}$};
        \node at (-5.4,-1.8) {$a_{s(e)}+1$};
          \node at (-3.6,-1.8) {$a_{s(e)}+2$};
            \node at (-1.8,-1.8) {$a_{s(e)}+3$};
            \node at (0,-1.8) {$a_{t(e)}$};
              \node at (1.8,-1.8) {$a_{t(e)}$};
                \node at (3.6,-1.8) {$a_{t(e)}$};
                  \node at (5.4,-1.8) {$a_{t(e)}$};
                    \node at (7.2,-1.8) {$a_{t(e)}$};
    \end{tikzpicture}
\caption{The case $\phi(e)=-1$.}
\label{fig:a}
\end{figure}
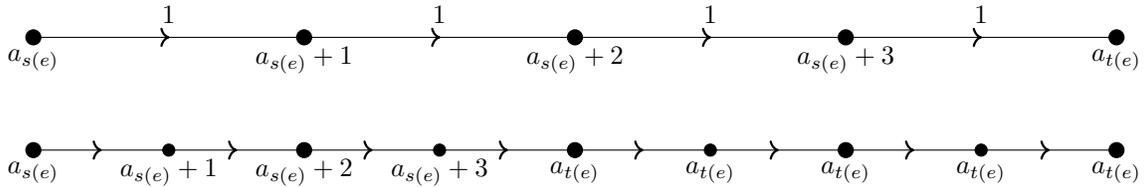

 Consider the diagram 
 \begin{equation}
\SelectTips{cm}{11}
\begin{xy} <16pt,0pt>:
\xymatrix{
 \mc Y\ar[r] \ar[rd]&    \mc X\times_B B\ar[r] \ar[d]& \mc X\ar[d]
 \\
& B \ar[r]^{t\ra t^2} & B
 }
\end{xy}
\end{equation}
where $\mc Y$ is the desingularization of $\mc X\times_B B$ at the nodes of $X'$ corresponding to edges in $E(\Gamma')\setminus E(\Gamma)$. Denote by $\varphi\col \mc Y\ra \mc X$ the induced map.  The dual graph of the special fiber $Y$ of $\mc Y\ra B$ is $\wt{\Gamma'}$. We let $Y_{\wt{v}}$ be the component of $Y$ corresponding to the vertex $\wt{v}\in V(\wt{\Gamma}')$. Define the twister $\wt T$ on $\mc Y$ as:
 \[
 \wt T:=\mc O_{\mc Y}\Big(\sum_{\wt v\in V (\wt{\Gamma}')}b_{\wt v}Y_{\wt v}\Big).
 \]
 % $\sum_{v\in V} X'_v$ where $\beta_D(V)=0$. 
   Let $\psi\col\mc Y\ra \wt{\mc X}$ be the contraction of the exceptional components of the special fiber of $\mc Y\ra B$ over which $\varphi^*(\mc L)\otimes \wt T$ has degree zero. Set $\wt{\mc L}:=\psi_*(\varphi^*(\mc L)\otimes \wt T)$. 
 Then $(\wt{\mc X},\wt{\mc L})$ is a family of spin curves over $B$ whose generic fiber is isomorphic to $(X,\mc L|_X)$ and the special fiber has spin combinatorial type $(\Gamma',\wt\E',\wt D',\pol(s'))$. By Remark \ref{rem:S-strat} we get a specialization of polystable root-graphs $(\Gamma',\wt\E',\wt D',\pol(s'))\ge (\Gamma,\emptyset,D,s)$, as wanted.

\smallskip
\emph{Case 2}. Assume that $(\Gamma,\emptyset,D)$ is not stable. 
 We argue by induction on $\#E(\Gamma)$. There is a proper subset $V\subset V(\Gamma)$ such that $\beta_D(V)=0$. Set $V':=\iota^{-1}(V)$.   Let $(\Gamma_i,\emptyset, D_i^V,s_i^V)$ and $(\Gamma'_i,\emptyset,D_i^{V'},s_i^{V'})$ be the spin semistable root-graphs as in Remark \ref{rem:D-split}.  
An edge $e\in \mc E(V,V^c)$ gives rise to a section $\sigma_e\col B\ra \mc X$ of $\pi$. We let $\pi_1,\pi_2\col\mc X_i\ra B$ be the two families of curves obtained by taking the normalization of $\mc X$ at the union of $\sigma_e(B)$, $\forall\; e\in E(V,V^c)$. The family $\pi_i$ could fail to be quasistable. Though, we can and will consider the family to be  quasistable simply adding a suitable number of sections to $\pi_i$. Set
\[
\mc M_1:=\mc L|_{\mc X_1}
\text{ and }
\mc M_2:=\mc L|_{\mc X_2}\otimes \mc O_{\mc X_2}\Big(-\sum_{e\in E(V,V^c)} \sigma_e(B)\Big).
\]
 If $X_i$ and $X'_i$ are the generic and special fiber of $\pi_i$, respectively, then $(X_i,\mc M_i|_{X_i})$ and $(X'_i,\mc M_i|_{X'_i})$ are semistable spin curves of spin combinatorial type $(\Gamma_i,\emptyset,D_i^V,s_i^V)$ and $(\Gamma'_i,\emptyset,D_i^{V'},s_i^{V'})$.
By the induction hypothesis, there is a specialization of semistable root-graphs $\iota_i\col(\Gamma'_i,\emptyset,D_i^{V'},s_i^{V'})\ge (\Gamma_i,\emptyset,D_i^V,s_i^V)$. We argue that $\iota_*(\phi_{(\emptyset,D')})=\phi_{(
\emptyset,D)}$: this is clear away from $E(V,V^c)$ by the existence of the specialization $\iota_i$, while $\iota_*(\phi_{(\emptyset,D')})(e)=\phi_{(\emptyset,D)}(e)=\pm1$, $\forall\;e\in \ora{E}(V,V^c)=\ora{E}(V',V'^c)$ by Theorem \ref{thm:uniquephi} (we use the natural inclusion $E(\Gamma)\subset E(\Gamma')$ to identify $\ora{E}(V,V^c)$ with $\ora{E}(V',V'^c)$, and that $\beta_{D'}(V')=0$ by Proposition \ref{prop:stab-contract}). If $\wt\iota\col \Gamma'/P_{(\emptyset,D')}\ra \Gamma/P_{(\emptyset,D)}$ is the associated specialization, we also deduce that $\wt{\iota}_*(s')=s$, and we are done.
\end{proof}

\begin{Prop}\label{prop:strata-cont} 
     Let $(X',L')$ be a semistable spin curves and $(\Gamma',\E',D',s')$ be its spin combinatorial type.  Let $\iota\col(\Gamma',\E',D',s')\ge (\Gamma,\E,D,s)$ be a specialization of spin semistable root-graphs.
    Then there is a one-parameter deformation $\mc X \to  B$ of  $X'$, with generic fiber $X$, and an invertible sheaf $\mc L$ on $\mc X/B$ such that $\mc L|_{X'}\cong L'$ and $(X,\mc L|_{X})$ is a semistable spin curves with spin combinatorial type $(\Gamma,\E,D,s)$.
%\end{itemize}
\end{Prop}

\begin{proof}
Adding a suitable number of section, we can make $X'$ a $n$-pointed stable  curve, hence 
we can assume $\E=\E'=\emptyset$.
 Let 
 $\wt{\iota}\col \Gamma'/P_{(\emptyset,D')}\ra \Gamma/P_{(\emptyset,D)}$ be the specializations associated to $(\Gamma',\emptyset,D',s')\ge (\Gamma,\emptyset,D,s)$.

\smallskip

\emph{Case 1.} Assume that $(\Gamma,\emptyset,D)$ is stable. Then $\phi_{(\emptyset,D)}=0$ by Lemma \ref{lem:stable}, hence $P_{(\emptyset,D)}=\Gamma$ and  $\#V(\Gamma/P_{(\emptyset,D)})=1$. By Corollary \ref{cor:facet}, we can assume that $\iota$ is elementary. If it is of type $1$, then $\phi_{(\emptyset,D')}=0$, so $(\Gamma',\emptyset,D',s')$ is stable by Lemma \ref{lem:stable}, and we are done by Remark \ref{rem:S-strat}. 

%Notice that $\E=\emptyset$ and $D=K_\Gamma/2$. 

If the specialization $\iota$ is elementary of type $2$, then 
%then there are just two connected components $P'_1,P'_2$ of $\ol{P}_{(\E',D')}$. We 
we can write 
$V(\Gamma')=V(P'_1)\cup V(P'_2)$, where $\iota$ is the contraction of $E(P'_1,P'_2)$. Let $\Gamma'_1,\dots,\Gamma'_n$ be the connected components of the subgraph of $\Gamma'$ having $E(P'_1,P'_2)$ as a set of edges, and $\Gamma'_{n+1},\dots,\Gamma'_m$ be the vertices in $V(\Gamma')\setminus \cup_{1\le i\le n} V(\Gamma'_i)$. Let $X'_i$ be the subcurve of $X'$ corresponding to $\Gamma'_i$, for all $i\in\{1,\dots,m\}$.
%\dots,\Gamma'_n$ and $X'_{n+1},\dots,X'_m$ the remaining components of $X$. 
For $i\in\{1,\dots,n\}$, the multidegree of  $\omega_{X'_i}\otimes L'|_{X'_i}^{-2}$ is equal to the restriction of $\Div(\phi_{(\emptyset,D')})$ to $\Gamma'_i$, which is a principal divisor on $\Gamma'_i$. By \cite[Corollary 6.9]{EM}, there is a smoothing $\mc X_i\ra B$ of $X'_i$ and a twister $T_i$ of $\mc X_i$ such that $T_i|_{X'_i}\cong\omega_{X'_i}\otimes L'|_{X'_i}^{-2}$. For $i\in\{n+1,\dots,m\}$, consider the trivial family $\mc X_i=X'_i\times B$.
Every edge $e\in E(\Gamma)$ gives rise to two sections $\sigma^i_e,\sigma^j_e\col B\ra \mc X_i$, for $i$ and $j$ such that  $\Gamma'_i$ and $\Gamma'_j$ are incident to $e$ (seen as an edge of $\Gamma'$). We let $\pi\col\mc X\ra B$ be the family obtained from $\sqcup \mc X_i$ by identifying $\sigma^i_e(b)$ and $\sigma^j_e(b)$, $\forall\; b\in B$. The generic fiber $X$ of $\pi$ has $\Gamma$ as a dual graph. 

We cover $\mc X$ with two open subsets $U$ and $V$, where $U\cong(\sqcup_{1\le i\le n} \mc X_i)\setminus (\sqcup_{i,e}\sigma^i_e(B))$ and $V$ is the complement of the singular points of $X'$ corresponding to edges $e\in E(P'_1,P'_2)$. We let $T_U$ be the sheaf on $U$ restricting to $T_i|_{U\cap \mc X_i}$ over $U\cap \mc X_i$. Since $\Pic(B)=0$, we have $T_U|_{U\cap V}\cong \mc O_{U\cap V}$, and hence we can glue $T_U$ and $\mc O_V$ along $U\cap V$ to get a sheaf $T$ on $\mc X$ such that $T|_{X'}\cong \omega_{X'}\otimes L'^{-2}$ and $T$ is trivial away from $X'$. By \cite[Remark 3.0.6]{CCC} there is an invertible sheaf $\mc L$ such that $\mc L|_{X'}\cong L'$ and $\mc L^{\otimes 2}\cong \omega_{\mc X/B}\otimes T^{-1}$. Since $\mc L|_X^{\otimes 2}\cong \omega_X$, we have $\E_{\mc L|_X}=\emptyset$ and $D_{\mc L|_X}=K_\Gamma/2=D$. 
By Proposition \ref{prop:pol-spin} we have a specialization $(\Gamma',\emptyset,D',s')\ra (\Gamma,\emptyset,D,s_{\mc L|_X})$ of semistable root-graphs, then $s_{\mc L|_X}=\wt{\iota}_*(s')=s$, and we are done.

\smallskip
\emph{Case 2.} Assume that $(\Gamma,\emptyset,D)$ is not stable. We argue by induction on $\#E(\Gamma)$. There is a proper subset $V\subset V(\Gamma)$ such that $\beta_D(V)=0$. Set $V':=\iota^{-1}(V)\subset V(\Gamma')$ and let $(\Gamma'_i,\emptyset,D_i^{V'},s_i^{V'})\ge (\Gamma_i,\emptyset,D_i^V,s_i^V)$ be the specialization
 of spin semistable root-graphs in Remark \ref{rem:D-split} for $i=1,2$. 
Let $X'_i$ be the subcurve of $X'$ corresponding to $\Gamma'_i$, where we add a marked point in correspondence of any $e\in E(\Gamma'_1,\Gamma'_2)$. Set $L'_i:=L'|_{X'_i}$. Define
\[
M'_1:=L'_1
\; \text{ and } \;
M'_2:=L'_2\otimes\mc O_{X'_2}\Big(-\sum_{e\in E(\Gamma'_1,\Gamma'_2)}  p_e\Big)
\]
  By construction, $(X'_i,M'_i)$ is a semistable spin curve of type $(\Gamma'_i,\emptyset,D_i^{V'},s_i^{V'})$. Let $\pi_i\col\mc X_i\ra B$ be the family obtained by the induction hypothesis applied to the specialization $(\Gamma',\emptyset,D_i^{V'},s_i^{V'})\ge(\Gamma,\emptyset,D_i^V,s_i^V)$.
Every edge $e\in E(\Gamma_1,\Gamma_2)\subset \Gamma$ gives rise to a section $\sigma^i_e$  of $\pi_i$. Let $\pi\col\mc X\ra B$ be the family obtained from $\mc X_1\sqcup \mc X_2$ by identifying a point $\sigma^1_e(b)$ with $\sigma^2_e(b)$, $\forall\;b\in B$. By construction, the generic fiber $X$ of $\pi$ has $\Gamma$ as a dual graph and the special fiber of $\pi$ is $X'$. 

Set $n:=\#E(\Gamma_1,\Gamma_2)=\#E(\Gamma'_1,\Gamma'_2)$. There are exact sequences
%\[
%0\ra \mc O^*_{\mc X} \ra \nu_*\mc O^*_{\mc X_1\sqcup \mc X_2} 
%\text{ and }
%\]
\[
0\ra \mc O^*_{\mc X} \ra \nu_*\mc O^*_{\mc X_1\sqcup \mc X_2}\ra \mc K\ra 0
\;\; \text{ and }\;\;
0\ra \mc O_{X'}^*\ra \mc O^*_{X'_1\sqcup X'_2}\ra \prod_{e\in E(\Gamma'_1,\Gamma'_2)} k^*_{p_e}\ra 0
\]
where $\mc K$
is a principal bundle over $B$ with fiber $(k^*)^n$. The associated long exact sequences in cohomology gives rise to a commutative diagram
%\[
%(k[[t]]^*)^n\ra\Pic(\mc X)\ra\Pic(\mc X_1)\times \Pic(\mc X_2)\ra 0.
%\]
\[
\SelectTips{cm}{11}
\begin{xy} <16pt,0pt>:
\xymatrix{
  H^0(\mc X,\mc K)\ar[d]^{\rho_0}\ar[r]^{\delta}&    \Pic(\mc X)\ar[r] \ar[d]^{\rho}& \Pic(\mc X_1)\times \Pic(\mc X_2)\ar[d]^{(\rho_1,\rho_2)}\ar[r] & 0
 \\
 (k^*)^n\ar[r]& \Pic(X') \ar[r] & \Pic(X'_1)\times  \Pic(X'_2)\ar[r] &  0 
 }
\end{xy}
\]
where $\rho,\rho_0,\rho_1,\rho_2$ are the natural restriction maps. 
%Let $\mc Q$ be the kernel of the natural restriction $\mc K\ra (k^*)^n$. 
Notice that $\mc K$ and $\prod_{e\in E(\Gamma'_1,\Gamma'_2)} k^*_{p_e}$ are supported on the image $\Sigma\subset \mc X$ of the sections $\sigma_e^i$. Moreover, $\pi|_{\Sigma}$ is just a collection of isomorphisms. Then $\pi|_{\Sigma*}$ is exact, and hence $\pi_*(\mc K)\to (k^*)^n$  is surjective. Since $B$ has only one closed point and only one open set containing this point, we get a surjection $H^0(\mc X,\mc K)\ra (k^*)^n$ which we can identify with $\rho_0$. Hence $\rho_0$ is surjective. It is clear that $\rho,\rho_1,\rho_2$ are surjective as well.

Let $\mc L\in \Pic(\mc X)$ be an invertible sheaf lifting $(\mc L_1,\mc L_2)\in \Pic(\mc X_1)\times \Pic(\mc X_2)$, where 
\[
\mc L_1=\mc M_1
\; \text{ and } \;
\mc L_2=\mc M_2\otimes \mc O_{\mc X_2}\Big(\sum_{e\in E(\Gamma'_1,\Gamma'_2)} \sigma^2_e(B)\Big).
\]
A simple diagram chasing shows the existence of an element $\alpha\in H^0(\mc X,\mc K)$ such that $\rho(\mc L\otimes \delta(\alpha))=L'$.  
 %by means of the element $\alpha\in (k[[t]]^*)^n$. 
% By construction, 
The family $\pi \col\mc X\ra B$ and the sheaf $\mc L\otimes \delta(\alpha)\in \Pic(\mc X)$ are as required by the statement.
\end{proof}

\begin{Cor}\label{cor:strata}
The following conditions are equivalent: 
\begin{itemize}
    \item[(1)] $\mc Q_{(\Gamma',\E',D',s')}\cap \ol{\mc Q}_{(\Gamma,\E,D,s)}\ne\emptyset$.
    \item[(2)] $(\Gamma',\E',D',s')\ge (\Gamma,\E,D,s)$ in the poset $\mc{SPQ}_{g,n}$.
    \item[(3)] $\mc Q_{(\Gamma',\E',D',s')}\subset \ol{\mc Q}_{(\Gamma,\E,D,s)}$.
\end{itemize}
\end{Cor}

\begin{proof}
The result readily follows from Propositions \ref{prop:pol-spin} and \ref{prop:strata-cont}.
\end{proof}

We are ready to prove the main result of the section.

\begin{Thm}\label{thm:Q-strat}
The following decomposition is a stratification of $\ol{\mc Q}_{g,n}$ indexed by $\mc {SPQ}_{g,n}$:
\begin{equation}\label{eq:strat}
\ol{\mc Q}_{g,n}=\coprod_{(\Gamma,\E,D,s)\in \mc {SPQ}_{g,n}}\mc Q_{(\Gamma,\E,D,s)}.
\end{equation}
In particular, $\ol{\mc Q}_{g,n}$ has exactly two connected components $\ol{\mc Q}^+_{g,n}$ and $\ol{\mc Q}^-_{g,n}$ corresponding to even and odd quasistable spin curves and:
\begin{itemize}
 \item[(1)] the maps $\ol{\mc Q}^+_{g,n}\ra\ol{\mc M}_{g,n}$ 
    and $\ol{\mc Q}^-_{g,n}\ra\ol{\mc M}_{g,n}$ are finite exactly over the locus of tree-like curves, and hence $\ol{\mc Q}_{g,n}^+$ and $\ol{\mc Q}_{g,n}^-$ are not isomorphic, respectively, to $\ol{\mc S}_{g,n}^+$ and $\ol{\mc S}_{g,n}^-$.
    \item[(2)] $\codim_{\ol{\mc Q}_{g,n}}\mc Q_{(\Gamma,\E,D,s)}= \#E(\Gamma)-b_1(\E,D)$.
    \item[(3)] $\codim_{\ol{\mc Q}_{g,n}}\mc Q_{(\Gamma,\E,D,s)}=1$ if and only if one of the following conditions hold:
    \begin{itemize}
        \item[(a)] $\#V(\Gamma)=\#E(\Gamma)=1$.
        \item[(b)]$\#V(\Gamma)=2$,  $\E=\emptyset$, and $P_{(\E,D)}=0$.
    \end{itemize}    
     \item[(4)]
    we have a reverse-inclusion bijection between the strata of $\ol{\mc Q}_{g,n}$ and the cells of  $Q_{g,n}^{\trop}$  switching dimension and codimension.
\end{itemize}
\end{Thm}

\begin{proof}
By Proposition \ref{prop:irred} and Corollary \ref{cor:strata}  we have that the decomposition in equation \eqref{eq:strat} is a stratification indexed by $\mc {SPQ}_{g,n}$. The statement about the  connected components follows by Remark \ref{rem:Q+Q-}. By Propositions \ref{prop:dim-cell} and \ref{prop:irred} we have that (2) and (4) hold.

We prove (1). Assume that $\Gamma$ is not a tree-like graph. By upper semicontituity, without loss of generality, we can assume that $V(\Gamma)=2$. Let $e_1,e_2$ be edges of $\Gamma$ that are not loops and set $\E=E(\Gamma)\setminus\{e_1,e_2\}$. 
Let $\phi$ be the acyclic flow
 on $\Gamma$ such that, for $e\in \ora{E}(\Gamma)$, we have $\phi(e)=1$ if either $t(e)$ is an exceptional vertex or  $s(e)=v_0$. 
By construction, we have that $K_{\Gamma^\E}(v)-\Div(\phi)(v)$ is even for every $v\in V(\Gamma^\E)$, hence there is a divisor $D$ on $\Gamma^\E$ such that $2D+\Div(\phi)=K_{\Gamma^\E}$. It is easy to check that $D$ is $v_0$-quasistable.
Then $(\Gamma,\E,D)$ is a $v_0$-quasistable root-graph with $P_{(\E,D)}=0$. If $H$ is the subgraph of $\Gamma$ with $e_1,e_2$ as edges, then we have $b_1(\E,D)=b_1(\Gamma_\E/P_{(\E,D)})=b_1(H)=1$.  We can easily find an even and an odd sign functions on $(\Gamma,\E,D)$, which we denote by $s^+$ and $s^-$. So, by Proposition \ref{prop:irred} the maps $\mc Q_{(\Gamma,\E,D,s^+)}\ra \ol{\mc M}_{g,n}$ and $\mc Q_{(\Gamma,\E,D,s^-)}\ra \ol{\mc M}_{g,n}$ are not finite.

Conversely, let $\Gamma$ be a tree-like grah and $T$ be its spanning tree. Let $(\Gamma,\E,D,s)$ be a $v_0$-quasistable spin root-graph. Then $\Gamma_\E$ is connected, hence $\E$ is formed by loops of $\Gamma$. Since $\phi_{(\E,D)}$ is acyclic and $P_{(\E,D)}$ is cyclic, we have that $\E\cup E(P_{(\E,D)})$ is the set of loops of $\Gamma$. Hence  $b_1(\E,D)=b_1(\Gamma_\E/P_{(\E,D)})=b_1(T)=0$ and, by Proposition \ref{prop:irred} the map $\mc Q_{(\Gamma,\E,D,s)}\ra \ol{\mc M}_{g,n}$ is finite.

We prove (3). We have $\#E(\Gamma)-b_1(\E,D)=1$ if and only if $\#\E+\#E(P_{(\E,D)})+\#V(\Gamma_{\E}/P_{(\E,D)})=2$, which is satisfied if either (a) or (b) holds. Moreover,
\begin{enumerate}
    \item[(i)]  if $\#V(\Gamma_\E/P_{(\E,D)})=1$, $\#\E=1$ and $P_{(\E,D)}=0$, then (a) holds.
    \item[(ii)]  if $\#V(\Gamma_\E/P_{(\E,D)})=1$, $\E=\emptyset$ and $\#E(P_{(\E,D)})=1$, then (a) holds.
    \item[(iii)] if $\#V(\Gamma_\E/P_{(\E,D)})=2$, then $\E=\emptyset$ and $P_{(\E,D)}=0$, and hence (b) holds.
\end{enumerate}
The proof is complete.
\end{proof}

\section*{Acknowledgements}

\noindent
This is a part of the Ph.D. thesis of the third author supervised by the first and second author. 
We thank Ethan Cotterill and Margarida Melo for some discussions and for the positive comments  on a preliminary version of the paper.

\bigskip
\medskip

\noindent{Alex Abreu, Marco Pacini,  and Danny Taboada \\
Universidade Federal Fluminense, Rua Prof. M. W. de Freitas, S/N\\ 
Niter\'oi, Rio de Janeiro, Brazil. 24210-201.}

\end{document}